\documentclass[12pt,a4paper,reqno]{amsart}
\usepackage{amsmath}
\usepackage{amsfonts}
\usepackage{amssymb}
\usepackage{bm} 

\newcommand\R{{\mathbf{R}}}
\newcommand\C{{\mathbf{C}}}
\newcommand\Z{{\mathbf{Z}}}

\newcommand\e{{\mathbf{e}}}
\renewcommand\H{{\mathbf{H}}}

\renewcommand\S{{\mathcal{S}}}
\newcommand\SC{{\mathcal{SC}}}
\newcommand\bigO{{\mathcal{O}}}
\newcommand\Sch{{\operatorname{Schwartz}}}

\newcommand\E{{\mathrm{E}}}

\newcommand\Hom{{\operatorname{Hom}}}

\newcommand\WM{{\operatorname{WM}}}
\newcommand\WMC{{\operatorname{WMC}}}

\newcommand\Rot{{\operatorname{Rot}}}
\newcommand\Frame{{\operatorname{Frame}}}
\newcommand\Energy{{\dot{\mathcal{H}^1}}}

\newcommand\loc{{\operatorname{loc}}}
\newcommand\const{{\operatorname{const}}}
\newcommand\eps{{\varepsilon}}

\newcommand\dist{{\operatorname{dist}}}

\theoremstyle{plain}
  \newtheorem{theorem}[subsection]{Theorem}

  \newtheorem{proposition}[subsection]{Proposition}
  \newtheorem{lemma}[subsection]{Lemma}
  \newtheorem{corollary}[subsection]{Corollary}

\theoremstyle{remark}
  \newtheorem{remark}[subsection]{Remark}
  
  \newtheorem{example}[subsection]{Example}

\theoremstyle{definition}
  \newtheorem{definition}[subsection]{Definition}

\include{psfig}

\begin{document}

\title[Global regularity of wave maps V]{Global regularity of wave maps V.  Large data local wellposedness and perturbation theory in the energy class}
\author{Terence Tao}
\address{Department of Mathematics, UCLA, Los Angeles CA 90095-1555}
\email{tao@math.ucla.edu}

\vspace{-0.3in}
\begin{abstract}
Using the harmonic map heat flow and the function spaces of Tataru \cite{tataru:wave2} and the author \cite{tao:wavemap2}, we establish a large data local well-posedness result in the energy class for wave maps from two-dimensional Minkowski space $\R^{1+2}$ to hyperbolic spaces $\H^m$.  This is one of the five claims required in \cite{tao:heatwave} to prove global regularity for such wave maps.  
\end{abstract}

\maketitle

\section{Introduction}

\subsection{The energy space}

This paper is a technical component of a larger program \cite{tao:heatwave} to establish large data global regularity for the initial value problem for two-dimensional wave maps into hyperbolic spaces.  Specifically, we establish in this paper a large data local well-posedness result for this problem in the energy space $\Energy$ constructed in \cite{tao:heatwave2}, which is a key ingredient in the arguments of \cite{tao:heatwave}. 

We begin by recalling some key features of this energy space. Fix $m \geq 1$; we allow all implied constants to depend on $m$.  Let $\H = (\H^m,h) \equiv SO(m,1) / SO(m)$ be the $m$-dimensional \emph{hyperbolic space}, i.e. the simply-connected $m$-dimensional Riemannian manifold of constant negative sectional curvature $-1$.  Here $SO(m,1)$ and $SO(m)$ are the special orthogonal groups of Minkowski space $\R^{1+m}$ and Euclidean space $\R^m$ respectively.  We define \emph{classical data} to be a pair $\Phi = (\phi_0,\phi_1)$, where $\phi_0: \R^2 \to \H$ is a smooth map which differs from a constant $\phi_0(\infty)$ by a Schwartz function (where we embed $\H$ in $\R^{1+m}$ to define the Schwartz space), and $\phi_1: \R^2 \to T \H$ is a Schwartz function such that $\phi_1(x)$ lies in the tangent plane $T_{\phi_0(x)} \H$ of $\H$ at $\phi_0(x)$ for every $x \in \R^2$, and let $\S$ be the space of all classical data; this can be given the structure of a topological space by using the Schwartz topology.  With regards to wave maps, one should interpret $\phi_0$ and $\phi_1$ as being the initial position and initial velocity respectively of a (classical) wave map at some time.  We observe the Lorentz rotation symmetry
\begin{align}
\Rot_U: (\phi_0(x), \phi_1(x)) &\mapsto (U\phi_0(x), dU(\phi_0(x))(\phi_1(x))) \label{rotate-data}  
\end{align}
for any Lorentz rotation $U \in SO(m,1)$ that acts continuously on $\S$.

Given any classical initial data $\Phi = (\phi_0,\phi_1)$, one can define the \emph{energy}
\begin{equation}\label{energy-def}
 \E(\Phi) := \frac{1}{2} \int_{\R^2} |\partial_x \phi_0|_{\phi_0^* h}^2 + |\phi_1|_{\phi_0^* h}^2\ dx
\end{equation}
where $\phi_0^* h$ is the pullback of the metric $h$ by $\phi_0$.  The energy functional $\E$ is a continuous map from $\S$ to $[0,+\infty)$.

In \cite{tao:heatwave2}, an energy space $\Energy$ was constructed with the following properties\footnote{In \cite{tao:heatwave2}, several additional properties of this space were also established, but they will not be needed in this paper and so we have omitted them.}:

\begin{theorem}[Energy space]\label{energy-claim}\cite{tao:heatwave2}  There exists a complete metric space $\Energy$ with a continuous map $\iota: \S \to \Energy$, that obeys the following properties:
\begin{itemize}
\item[(i)] $\iota(\S)$ is dense in $\Energy$.
\item[(ii)] $\iota$ is invariant under the action \eqref{rotate-data} of the Lorentz rotation group $SO(m,1)$, thus $\iota(\Rot_U \Phi) = \iota(\Phi)$ for all $\Phi \in \S$. Conversely, if $\iota(\Phi) = \iota(\Psi)$, then $\Psi = \Rot_U(\Phi)$ for some $U \in SO(m,1)$.
\item[(iii)] The energy functional $\E: \S \to [0,+\infty)$ extends to a continuous map $\E: \Energy \to [0,+\infty)$ (after quotienting out by Lorentz rotations).  
\item[(iv)] For any $\Phi \in \Energy$, we have $\E(\Phi) = d( \Phi, \const )^2$, where $\const := \iota(p,0)$ is the image of the constant map $(p,0)$ for any $p \in \H$.
\end{itemize}
\end{theorem}

\begin{remark} From (i), (ii) one can interpret $\Energy$ as a metric completion of the quotient $SO(m,1) \backslash \S$ of $\S$ by the action of the Lorentz group.  The construction of this space $\Energy$ involves the harmonic map heat flow and will be reviewed in Section \ref{energy-sec}, as the details of this construction will be important in our arguments.
\end{remark}

\subsection{Wave maps}

Define a \emph{classical wave map} to be a pair $(\phi,I)$, where $I$ is a time interval, and $\phi: I \times \R^2 \to \H$ is a smooth map which differs from a constant $\phi(\infty) \in \H$ by a Schwartz function in space\footnote{We say that a function $\phi: I \times \R^2 \to \R$ is \emph{Schwartz} if it is smooth, and $\partial_t^j \partial_x^k \phi$ is rapidly decreasing in space for all $j,k \geq 0$, uniformly in time.  We let $\Sch(I)$ denote the space of all Schwartz functions, with the usual Frechet space topology.}, and which obeys the equation
\begin{equation}\label{cov}
 (\phi^*\nabla)^\alpha \partial_\alpha \phi = 0,
\end{equation}
where $\phi^*\nabla$ is the pullback of the Levi-Civita connection $\nabla$ on $\H$ by $\phi$.  Here and in the sequel we use Greek indices $\alpha, \beta =0,1,2$ to range over the coordinates of Minkowski space $\R^{1+2}$ (with the usual metric $-dt^2+dx_1^2+dx_2^2$), with the usual summation, raising, and lowering conventions (and setting $x_0 := t$).  Roman indices $i,j=1,2$ will be summed over the spatial coordinates only (with no attempt to raise or lower coordinates).

Observe that for any time $t \in I$, the data $\phi[t] := (\phi(t), \partial_t \phi(t))$ of a classical wave map $(\phi,I)$ lies in $\S$, and indeed $\phi$ can be viewed as a smooth curve $\phi: I \to \S$.  We refer to \cite{kman.barrett}, \cite{kman.selberg:survey}, \cite{shatah-struwe}, \cite{struwe.barrett}, \cite{tataru:survey}, \cite[Chapter 6]{tao:cbms}, \cite{rod}, \cite{krieger:survey} for surveys of the initial value problem for wave maps, which is of course the primary concern of this project.  As is well known, wave maps have a conserved energy $\E(\phi) = \E(\phi(t))$, which in this two spatial dimensional setting is also invariant under the natural scale invariance $\phi^{(\lambda)}(t,x) := \phi( \frac{t}{\lambda}, \frac{x}{\lambda} )$ for this problem.

The purpose of this paper is to establish the following large data local well-posedness result for wave maps in the energy class:

\begin{theorem}[Large data local-wellposedness in the energy space]\label{lwp-claim}  For every time $t_0 \in \R$ and every initial data $\Phi_0 \in \Energy$ there exists a \emph{maximal lifespan} $I \subset \R$, and a \emph{maximal Cauchy development} $\phi: t \mapsto \phi[t]$ from $I \to \Energy$, which obeys the following properties:
\begin{itemize}
\item[(i)] (Local existence) $I$ is an open interval\footnote{An \emph{interval} in this paper is a connected subset of $\R$ with non-empty interior.  It can be closed, open, or half-open, and it can be bounded or unbounded.} containing $t_0$.
\item[(ii)] (Strong solution) $\phi: I \to \Energy$ is continuous.
\item[(iii)] (Persistence of regularity) If $\Phi_0 = \iota( \tilde \Phi_0 )$ for some classical data $\tilde \Phi_0$, then there exists a classical wave map $(\tilde \phi,I)$ with initial data $\tilde \phi[t_0] = \tilde \Phi_0$ such that $\phi[t] = \iota(\tilde \phi[t])$ for all $t \in I$.
\item[(iv)] (Continuous dependence)  If $\Phi_{0,n}$ is a sequence of data in $\Energy$ converging to a limit $\Phi_{0,\infty}$, and $\phi_n: I_n \to \Energy$ and $\phi_\infty: I_\infty \to \Energy$ are the associated maximal Cauchy developments on the associated maximal lifespans, then for every compact subinterval $K$ of $I_\infty$, we have $K \subset I_n$ for all sufficiently large $n$, and $\phi_n$ converges uniformly to $\phi$ on $K$ in the $\Energy$ topology.
\item[(v)] (Maximality)  If $t_* \in \R$ is a finite endpoint of $I$, then $\phi(t)$ has no convergent subsequence in $\Energy$ as $t \to t_*$.
\end{itemize}
\end{theorem}

This result is one of five claims required in \cite{tao:heatwave} to establish global regularity for classical wave maps; of the other four claims, three\footnote{One of these claims, namely the non-existence of non-trivial travelling or self-similar wave maps, was only proven conditionally on Theorem \ref{lwp-claim}.  However we will not use that claim in this paper, so the arguments are not circular.} of them (including the construction of the energy space, Theorem \ref{energy-claim}) were proven in \cite{tao:heatwave2}; the remaining claim, concerning the existence of almost periodic maximal Cauchy developments in the case when global regularity fails, will be established in the sequels \cite{tao:heatwave4}, \cite{tao:heatwave5} to this paper.

Theorem \ref{lwp-claim} has the following corollary:

\begin{corollary}[Small data global-wellposedness in the energy space]\label{lwp-claim-small}  There exists an absolute constant $\eps_0 > 0$ (depending only on $m$) for which the following statements hold: given every time $t_0 \in \R$ and every initial data $\Phi_0 \in \Energy$ with $\E(\Phi_0) < \eps_0$ there exists a \emph{global Cauchy development} $\phi[t]$ from $\R$ to $\Energy$, which obeys the following properties:
\begin{itemize}
\item[(i)] (Strong solution) $\phi: \R \to \Energy$ is continuous.
\item[(ii)] (Persistence of regularity) If $\Phi_0 = \iota( \tilde \Phi_0 )$ for some classical data $\tilde \Phi_0$, then there exists a classical wave map $(\tilde \phi,\R)$ with initial data $\tilde \phi[t_0] = \tilde \Phi_0$ such that $\phi[t] = \iota(\tilde \phi[t])$ for all $t \in \R$.
\item[(iii)] (Continuous dependence)  If $\Phi_{0,n}$ is a sequence of data in $\Energy$ with $\E(\Phi_{0,n}) < \eps_0$ converging to a limit $\Phi_{0,\infty}$ with $\E(\Phi_{0,\infty}) < \eps_0$, and $\phi_n: I_n \to \Energy$ and $\phi_\infty: I_\infty \to \Energy$ are the associated global Cauchy developments, then $\phi_n$ converges locally uniformly to $\phi$ on compact intervals in the $\Energy$ topology.
\end{itemize}
\end{corollary}

\begin{proof} In view of Theorem \ref{lwp-claim}, the only claim that needs verification is that the maximal lifespan of any data of sufficiently small energy is global.  Assume for contradiction that this were not the case, then one can find a sequence $\Phi_n \in \Energy$ with energy going to zero, whose maximal Cauchy developments $\phi_n: I_n \to \Energy$ from the initial time $t_n$ were not global.  By time translation symmetry\footnote{See \cite{tao:heatwave} for a justification of the various symmetries of the wave maps equation in the energy class.} we can take $t_n = 0$; by scaling symmetry one can ensure that $\dist( \partial I_n, 0 ) \to 0$ as $n \to \infty$.  By Theorem \ref{energy-claim}(iv), we know that $\Phi_n$ converges to constant data, which of course has a global classical solution to the wave maps problem.  But this now contradicts Theorem \ref{lwp-claim}(iii).
\end{proof}

Corollary \ref{lwp-claim-small}, when specialised to classical data, gives a global regularity result for small energy:

\begin{corollary}[Small energy regularity]\label{small-reg}  Let $\eps_0$ be as in the preceding corollary.  If $\Phi \in \S$ with $\E(\Phi) < \eps_0$ and $t_0 \in \R$, then there is a unique global classical wave map $\phi: \R \to \S$ with $\phi[t_0] = \Phi$.
\end{corollary}

\begin{remark} Corollary \ref{small-reg} for the case $m=2$ of the hyperbolic plane target was established in \cite{krieger:2d} (see also \cite{tao:wavemap2}, \cite{tataru:wave3} for closely related results in other manifolds).  It is likely that one can partially reverse the above implications and deduce Theorem \ref{lwp-claim} from Theorem \ref{lwp-claim-small} by exploiting finite speed of propagation (see \cite{sogge:wave} for an example of such an argument for the energy-critical wave equation, or Section \ref{subcrit-sec} below for such an argument in the subcritical regularity setting).  However there is a technical obstacle to doing so, namely that the energy topology that we use for $\Energy$ is constructed via the harmonic map heat flow, and so it is not immediately obvious that this topology is compatible with localisation in space.  While it is likely that this obstacle could be resolved with additional effort, we have chosen not to do so here.
\end{remark}

\begin{remark} Very recently, a small energy global well-posedness result for Schr\"odinger maps into the sphere $S^2$ was established in \cite{bej}.  There are many common features between the results here and those in \cite{bej}, most notably the reliance on the caloric gauge.  However, the function spaces used for the Schr\"odinger map and wave map equation are very different from each other; also, there are additional technical difficulties in setting up the iteration scheme in the large energy case which can be avoided in the small energy setting.
\end{remark}

\subsection{Organisation of the paper}

The large data local wellposedness claims in Theorem \ref{lwp-claim} will be proven as follows.  After setting out our basic notation in Section \ref{notation-sec}, we shall establish a subcritical local well-posedness theory in Section \ref{subcrit-sec}.  In Section \ref{reduce-sec}, we use this theory to reduce matters to the task of establishing some abstract \emph{a priori} estimates for classical wave maps (Theorem \ref{apriori-thm}).  In order to prove this theorem, we introduce the caloric gauge in Section \ref{energy-sec}, and reduce matters to establishing analogous \emph{a priori} estimates for classical wave maps in the caloric gauge (Theorem \ref{apriori-thm2}).

To achieve this, we develop some abstract parabolic regularity theory in Section \ref{parab-sec} and use it to obtain good control on the heat flow at the initial time in Section \ref{initial-sec}.  Then, in Section \ref{func0-sec} we abstractly describe the hyperbolic function spaces we will need for the \emph{a priori} estimates (Theorem \ref{func}).  In Section \ref{ap-sec} we use the estimates in Theorem \ref{func} (and the abstract parabolic regularity theory mentioned earlier) to establish Theorem \ref{apriori-thm2}.  Finally, in Section \ref{func0-sec} we recall the spaces constructed in \cite{tao:wavemap2} which will allow us to verify Theorem \ref{func} relatively quickly.

\subsection{Acknowledgements}

This project was started in 2001, while the author was a Clay Prize Fellow.  The author thanks Andrew Hassell and the Australian National University for their hospitality when a substantial portion of this work was initially conducted, and to Ben Andrews and Andrew Hassell for a crash course in Riemannian geometry and manifold embedding, and in particular to Ben Andrews for explaining the harmonic map heat flow.  The author also thanks Mark Keel for background material on wave maps, Daniel Tataru for sharing some valuable insights on multilinear estimates and function spaces, and to Igor Rodnianski and Jacob Sterbenz for valuable discussions.  The author is supported by NSF grant DMS-0649473 and a grant from the Macarthur Foundation.

\section{Notation and basic estimates}\label{notation-sec}

In order to efficiently manage the rather large amount of technical computations, it is convenient to introduce a substantial amount of notation, as well as some basic estimates that we will use throughout the paper.

\subsection{Small exponents}

We will need to fix three small exponents
$$ 0 < \delta_0 \ll \delta_1 \ll \delta_2 \ll 1$$
which are absolute constants, with each $\delta_i$ being sufficiently small depending on all higher $\delta_i$.  The exact choice of constants is not important, but for sake of concreteness one could take $\delta_i := 10^{-10^{3-i}}$, for instance.  All the implied constants in the asymptotic notation below can depend on these exponents.

\begin{remark}
The interpretation of these constants in our argument will be as follows. The largest constant $\delta_2$ is the exponent that quantifies certain useful exponential decays in frequency-localised linear, bilinear, and trilinear estimates when the ratio of two frequencies becomes favorable.  The intermediate constant $\delta_1$ is used to design the weakly frequency localised space $S_k$ (and its variant $S_{\mu,k}$, adapted to the large data theory) that we will encounter later in this paper.  The smallest constant $\delta_0$ is used to control the fluctuation of the frequency envelopes $c(s)$ that we will use to control solutions.  
\end{remark}

It will be useful to have some standard cutoff functions that compare two frequency parameters $k, k'$:

\begin{definition}[Cutoff functions]  Given any integers $k, k'$, we define $\chi_{k \geq k'} = \chi_{k' \leq k} := \min( 1, 2^{-(k-k')} )$ and $\chi_{k = k'} := 2^{-|k-k'|}$.
\end{definition}

Thus $\chi_{k' \leq k}$ is weakly localised to the region $k' \leq k$, and similarly for the other cutoffs. In practice we shall usually raise these cutoffs to an exponent such as $\pm \delta_0$, $\pm \delta_1$, or $\pm \delta_2$.

\subsection{Asymptotic notation}

The dimension $m$ of the target hyperbolic space $\H^m$ is fixed throughout the paper, and all implied constants can depend on $m$.

We use $X = O(Y)$ or $X \lesssim Y$ to denote the estimate $|X| \leq CY$ for some absolute constant $C > 0$, that can depend on the $\delta_i$ and the dimension $m$ of the target hyperbolic space.  If we wish to permit $C$ to depend on some further parameters, we shall denote this by subscripts, e.g. $X = O_k(Y)$ or $X \lesssim_k Y$ denotes the estimate $|X| \leq C_k Y$ where $C_k > 0$ depends on $k$.  

Note that parameters can be other mathematical objects than numbers.  For instance, the statement that a function $u: \R^2 \to \R$ is Schwartz is equivalent to the assertion that one has a bound of the form $|\partial_x^k u(x)| \lesssim_{j,k,u} \langle x \rangle^{-j}$ for all $j,k \geq 0$ and $x \in \R^2$, where $\langle x \rangle := (1+|x|^2)^{1/2}$.

\subsection{Schematic notation}

We use $\partial_x$ as an abbreviation for $(\partial_1, \partial_2)$, thus for instance 
\begin{align*}
|\partial_x \phi|_{\phi^* h}^2 &= \langle \partial_i \phi, \partial_i \phi \rangle_{\phi^* h}^2 \\
|(\phi^* \nabla)_x \partial_x \phi| &= \langle (\phi^* \nabla)_i \partial_j \phi, (\phi^* \nabla)_i \partial_j \phi \rangle_{\phi^* h}^2
\end{align*}
with the usual summation conventions.

We use juxtaposition to denote tensor product; thus for instance if $\psi_x := (\psi_1, \psi_2)$, then $\psi_x^2 = \psi_x \psi_x$ denotes the rank $2$ tensor with the four components $\psi_i \psi_j$ for $i,j = 1,2$; similarly, $\partial_x^2$ is the rank $2$ tensor operator with four components $\partial_i \partial_j$ for $i,j=1,2$, and so forth.

If $X$ is a tensor-valued quantity, we use $\bigO( X )$ to denote an expression which is \emph{schematically} of the form $X$, which means that it is a tensor whose components are linear combinations of those in $X$, with coefficients being absolute constants (depending only on $m$).  Thus for instance, if $v, w \in \R^m$, then the anti-symmetric matrix $v \wedge w := v w^\dagger - w v^\dagger$ has the schematic form $v \wedge w = \bigO( v w )$.  If the coefficients in the schematic representation depend on a parameter, we will denote this by subscripts.  Thus for instance we have the \emph{Leibniz rule}
\begin{equation}\label{leibnitz}
\partial_x^j \bigO( \phi \psi ) = \sum_{j_1,j_2 \geq 0: j_1+j_2=j} \bigO_j( \partial_x^{j_1} \phi \partial_x^{j_2} \psi )
\end{equation}
and similarly for products of three or more functions.

\subsection{Difference notation}

Throughout this paper, we adopt the notational conventions $\delta f := f' - f$ and $f^* := (f,f')$ for any field $f$ for which the right-hand side makes sense.  For future reference we observe the discretised Leibniz rule
\begin{equation}\label{disc-leib-eq}
\delta(fg) = \bigO( (\delta f) g^* ) + \bigO( f^* \delta g )
\end{equation}
and similarly for products of three or more functions.

\subsection{Function spaces}  
We use the usual Lebesgue spaces $L^p_x(\R^2)$ and Sobolev spaces $H^s_x(\R^2)$, and create spacetime norms such as $L^q_t L^p_x(I \times \R^2)$ in the usual manner.  Later on we shall also use the more complicated spaces adapted to the wave maps problem from \cite{tao:wavemap2} (see also \cite{tataru:wave2}).

If $X$ is a normed space for scalar-valued functions, we also extend $X$ to functions $\phi := (\phi_1,\ldots,\phi_m)$ taking values in a standard finite-dimensional vector space such as $\R^m$ with the convention
$$ \|\phi\|_X := (\sum_{j=1}^m \|\phi_j\|_X^2)^{1/2}.$$
Note that if $X$ was already a Hilbert space on scalar functions, it continues to be a Hilbert space on vector-valued functions, and the orthogonal group $O(m)$ on that space acts isometrically on this space.  If $X$ is merely a normed vector space, then the orthogonal group is no longer isometric, but the action of an element of this group has operator norm bounded above and below by constants depending only on $m$.

\begin{definition}[Littlewood-Paley projections] 
Let $\varphi(\xi)$ be a radial bump function supported in the ball $\{ \xi \in \R^2: |\xi| \leq \tfrac {11}{10} \}$ and equal to $1$ on the ball $\{ \xi \in \R^2: |\xi| \leq 1 \}$.  For each integer $k$, we define the Fourier multipliers
\begin{align*}
\widehat{P_{\leq k} f}(\xi) &:= \varphi(\xi/2^k) \hat f(\xi)\\
\widehat{P_{>k} f}(\xi) &:= (1 - \varphi(\xi/2^k)) \hat f(\xi)\\
\widehat{P_k f}(\xi) &:= (\varphi(\xi/2^k) - \varphi(2\xi/2^k)) \hat f(\xi).
\end{align*}
We similarly define $P_{<k}$ and $P_{\geq k}$.
\end{definition}

\subsection{The linear heat equation}\label{linear-heat}

Throughout the paper we use $\Delta := \frac{\partial^2}{\partial x_1^2} + \frac{\partial^2}{\partial x_2^2}$ to denote the (spatial) Laplacian on $\R^2$. We use $e^{s\Delta}$ for $s > 0$ to denote the free heat propagator
\begin{equation}\label{heat-prop}
e^{s\Delta} u(x) := \frac{1}{4\pi s} \int_{\R^2} e^{-|x-y|^2/4s} u(y)\ dy.
\end{equation}
From Young's inequality we easily establish the parabolic regularity estimate
\begin{equation}\label{heat-lp}
\| \partial_x^j e^{s\Delta} u \|_{L^q_x(\R^2)} \lesssim_{p,q,j} s^{\frac{1}{p}-\frac{1}{q}+\frac{j}{2}} \| e^{s\Delta} u \|_{L^p_x(\R^2)} 
\end{equation}
valid for all $s > 0$, $j \geq 0$, and $1 \leq p \leq q \leq \infty$.  

We recall \emph{Duhamel's formula}
\begin{equation}\label{duh}
 u(s_1) = e^{(s_1-s_0)\Delta} u(s_0) + \int_{s_0}^{s_1} e^{(s_1-s)\Delta} (\partial_s u - \Delta u)(s)\ ds
\end{equation}
for any continuous map $s \mapsto u(s)$ from the interval $[s_0,s_1]$ to the space of tempered distributions on $\R^2$, which can be either scalar or vector valued.  

We also need a Strichartz-type estimate:

\begin{lemma}[Parabolic Strichartz estimate]\label{ubang} Let $\psi: \R^+ \times \R^2 \to \R$ be smooth, with all derivatives uniformly Schwartz in space.  Then we have
$$ (\int_0^\infty \| \psi \|_{L^\infty_x(\R^2)}^2\ ds)^{1/2} \lesssim \|\psi(0)\|_{L^2_x(\R^2)} + \int_0^\infty \| (\partial_s - \Delta) \psi(s) \|_{L^2_x(\R^2)}\ dx.$$
\end{lemma}

\begin{proof} By \eqref{duh} and Minkowski's inequality it suffices to establish this claim for free solutions to the heat equation, thus $\psi(s) = e^{s\Delta} \psi(0)$.  But this follows from \cite[Lemma 2.5]{tao:heatwave2}.
\end{proof}

\subsection{Frequency envelopes}

We first recall a useful definition from \cite{tao:wavemap}.

\begin{definition}[Frequency envelope]\label{freqenv}\cite{tao:wavemap}  Let $E > 0$.  A \emph{frequency envelope} of energy $E$ is a map $c: \R^+ \to \R^+$ with
\begin{equation}\label{cse}
\int_0^\infty c(s)^2 \frac{ds}{s} = E
\end{equation}
such that
\begin{equation}\label{sm}
c(s') \leq \max( (s'/s)^{\delta_0}, (s/s')^{\delta_0} ) c(s)
\end{equation}
for all $s, s' > 0$.
\end{definition}

\begin{remark} The estimate \eqref{sm} is asserting that $c(s)$ can grow at most as fast as $s^{\delta_0}$, and decay at most as rapidly as $s^{-\delta_0}$.  In \cite{tao:wavemap} a discretised version of this concept was used, with the continuous variable $s$ being replaced by the discrete variable $2^{-2k}$ for integer $k$.  However, as the heat flow uses a continuous time variable $s$, it is more natural to use the continuous version of a frequency envelope.  In \cite{tataru:wave3} it was observed that one could take asymmetric envelopes, in which $c(s)$ is allowed to decay faster as $s$ decreases than when $s$ increases.  However, due to our use of heat flow (which does not have as good frequency damping properties near the frequency origin as Littlewood-Paley operators) it is not convenient to use these asymmetric envelopes in our arguments.
\end{remark}

\begin{remark} Observe from \eqref{cse}, \eqref{sm} that if $c$ is a frequency envelope of energy $E$, then we have the pointwise bounds
\begin{equation}\label{cse-sup}
c(s) \lesssim \sqrt{E}
\end{equation}
for all $s > 0$.
\end{remark}

We also define the frequency $k(s) \in \Z$ of a heat-temporal variable $s>0$ by the formula
\begin{equation}\label{ks-def}
k(s) := \lfloor \log_2 s^{-1/2} \rfloor,
\end{equation}
thus $2^{-2k(s)} \sim s$.  

We record two useful Gronwall-type inequalities relating to these frequency envelopes.  The first lemma involves integral inequalities coming from $s=+\infty$:

\begin{lemma}[Gronwall-type inequality from $s=+\infty$]\label{gron-lem} Let $c$ be a frequency envelope of energy at most $E$.  Let $f, g: \R^+ \to \R^+$ be locally integrable functions with $\lim_{s \to \infty} f(s) = 0$ such that $g$ obeys \eqref{sm}, and
$$ f(s) \leq g(s) + \int_s^\infty f(s') c(s') ((s'-s)/s')^{-\theta} (s'/s)^{-3\delta_0}\ \frac{ds'}{s'}$$
for all $s > 0$ and some $0 \leq \theta < 1$.  Then we have
$$ f(s) \lesssim_{E,\theta} g(s)$$
for all $s > 0$.
\end{lemma}

\begin{proof}   We allow implied constants to depend on $E,\theta$.  If we let $\tilde f(s) := \sup_{s \leq s' \leq 2s} f(s)$, then we easily compute (using \eqref{sm}) that
$$ \int_s^\infty f(s') c(s') ((s'-s)/s')^{-\theta} (s'/s)^{-3\delta_0}\ \frac{ds'}{s'}
\lesssim 
\int_s^\infty \tilde f(s') c(s') (s'/s)^{-3\delta_0}\ \frac{ds'}{s'}$$
and thus
$$ \tilde f(s) \lesssim g(s) + \int_s^\infty \tilde f(s') c(s')  (s'/s)^{-3\delta_0}\ \frac{ds'}{s'}.$$
Because of this, we see that the $\theta \neq 0$ case follows from the $\theta = 0$ case.  Henceforth we take $\theta=0$.

Since $f(s') c(s')$ goes to zero as $s' \to \infty$, we see for $S$ sufficiently large that
$$ f(s) \leq s^{3\delta_0} S^{-3\delta_0} + g(s) + \int_s^S f(s') c(s') (s'/s)^{-3\delta_0}\ \frac{ds'}{s'}$$
for all $0 < s \leq S$.  If we set $G(s) := S^{-3\delta_0} + g(s) s^{-3\delta_0}$ and
$F(s) := \max( f(s) s^{-3\delta_0} - G(s), 0)$, we thus have
$$ F(s) \leq \int_s^S (F(s')+G(s)) c(s') \frac{ds'}{s'}.$$
Applying Gronwall's inequality, we conclude that
$$ F(s) \leq \exp( \int_s^S c(s') \frac{ds'}{s'} ) F(S) + \int_s^S \exp( \int_s^{s'} c(s'') \frac{ds''}{s''} ) G(s') c(s')\ \frac{ds'}{s'}.$$
From \eqref{cse} and Cauchy-Schwarz we have
$$ \int_s^{s'} c(s'') \frac{ds''}{s''}  \lesssim \sqrt{\log \frac{s'}{s}}$$
and hence
$$ \exp( \int_s^{s'} c(s'') \frac{ds''}{s''} ) \lesssim (s'/s)^{\delta_0}.$$
We thus have
$$ F(s) \lesssim (S/s)^{\delta_0} F(S) + \int_s^S (s'/s)^{\delta_0} G(s') c(s') \frac{ds'}{s'}.$$
If we let $S \to \infty$, then $(S/s)^{\delta_0} F(S) \to 0$ by hypothesis on $f$, and thus
$$ f(s) s^{-3\delta_0} - g(s) s^{-3\delta_0} \lesssim \int_s^\infty (s'/s)^{\delta_0} g(s') (s')^{-3\delta_0} c(s')\ \frac{ds'}{s'}.$$
Using \eqref{cse-sup} we obtain the claim.
\end{proof}

By making the change of variables $s \mapsto 1/s$ we also have a variant from $s=0$:

\begin{lemma}[Gronwall-type inequality from $s=0$]\label{gron-lem-2} Let $c$ be a frequency envelope of energy at most $E$.  Let $f, g: \R^+ \to \R^+$ be locally integrable functions with $\lim_{s \to 0} f(s) = 0$ such that $g$ obeys \eqref{sm}, and
$$ f(s) \leq g(s) + \int_s^\infty f(s') c(s') ((s-s')/s)^{-\theta} (s'/s)^{3\delta_0}\ \frac{ds'}{s'}$$
for all $s > 0$ and some $0 \leq \theta < 1$.  Then we have
$$ f(s) \lesssim_{E,\theta} g(s)$$
for all $s > 0$.
\end{lemma}

\section{Subcritical local existence theory}\label{subcrit-sec}

Before we study wave maps in the energy class, we first review the (much simpler) subcritical local existence theory for classical wave maps, in which the bounds are allowed to depend on smoother norms than the energy norm, and in particular on norms which are subcritical with respect to scaling, of $H^{1+\eps}$ type. This theory was essentially worked out by Klainerman and Selberg \cite{kman.selberg} (see also the critical Besov space refinement in \cite{tataru:wave2}); the only (slight) wrinkle we need to deal with here is the Lorentz rotation-invariance of the problem, which can distort the local coordinates slightly.  This theory is not used in the ``quantitative'' components of the argument, but is needed for more ``qualitative'' components, such as ensuring that various continuity arguments can be justified, or that an \emph{a priori} hypothesis can be removed.

For the purposes of this qualitative analysis it is convenient to represent the hyperbolic space $\H = (\H^m,h)$ concretely in an ambient Minkowski space $\R^{1+m}$ as the upper unit hyperboloid
\begin{equation}\label{hyperdef}
 \H := \{ (t,x) \in \R^{1+m}: t = +\sqrt{1 + |x|^2} \} \subset \R^{1+m}
\end{equation}
with the metric $dh^2$ induced from the Minkowski metric $dg^2 = -dt^2 + dx^2$ on $\R^{1+m}$, and with the obvious action of $SO(m,1)$.  The connection
$\phi^*\nabla$ on the pullback tangent bundle $\phi^* T\H$ can then be written in these coordinates as
\begin{equation}\label{phicor}
(\phi^*\nabla)_i \psi(x) = \partial_i \psi(x) - \langle \psi(x), \partial_i \phi(x) \rangle_{\R^{1+m}} \phi(x).
\end{equation}
whenever $\psi$ is Minkowski-orthogonal to $\phi$ (i.e. $\langle \psi, \phi \rangle_{\R^{1+m}}=0$), while the wave maps equation \eqref{cov} in these coordinates becomes
\begin{equation}\label{wavemap-hyper}
\Box \phi = \langle \partial^\alpha \phi, \partial_\alpha \phi \rangle_{\R^{1+m}} \phi
\end{equation}
where $\Box := \partial^\alpha \partial_\alpha$ is the d'Lambertian operator.

We observe the zero tension property
\begin{equation}\label{zerotor-phi}
 (\phi^*\nabla)_i \partial_j \phi = (\phi^*\nabla)_j \partial_i \phi
\end{equation}
and the constant negative curvature property
\begin{equation}\label{curv-phi}
(\phi^*\nabla)_i (\phi^*\nabla)_j \psi - (\phi^*\nabla)_j (\phi^*\nabla)_i \psi 
= - (\partial_i \phi \wedge \partial_j \phi) \psi
\end{equation}
for any section $\psi$ of $\phi^* T\H$, where $\partial_i \phi \wedge \partial_j \phi \in \Gamma( \Hom( \phi^* T\H \to \phi^* T\H ) )$ is the anti-symmetric rank $(1,1)$ tensor defined by the formula
$$(\partial_i \phi \wedge \partial_j \phi) \psi = \partial_i \phi \langle \partial_j \phi, \psi \rangle_{\phi^* h} - (\partial_j \phi \wedge \partial_i \phi) \psi.$$

We fix a smooth scalar cutoff function $\eta \in C^\infty_0(\R^2)$ that equals $1$ on the ball $B(0,1)$ and vanishes outside of $B(0,2)$. Given classical initial data $\Phi := (\phi_0,\phi_1) \in \S$ and $s \geq 1$, we define the Sobolev-type norms
\begin{equation}\label{hkdef}
 \| \Phi \|_{{\mathcal H}^s_\loc} := \|f\|_{L^\infty_x(\R^2)}
 \end{equation}
and 
\begin{equation}\label{hkdef-loc}
 \| \Phi \|_{{\mathcal H}^s} := \| \Phi \|_{{\mathcal H}^s_\loc} + \|f\|_{L^2_x(\R^2)}
\end{equation}
where $f(x_0)$ is the local Sobolev norm near $x_0$, defined as
$$ f(x_0) := \inf_{U \in SO(m,1)} \| \eta(\cdot-x_0) U(\phi_0) \|_{H^s_x(\R^2)} + \| \eta(\cdot-x_0) U(\phi_1) \|_{H^{s-1}_x(\R^2)}.$$
\begin{remark} These (rather artificial) quantities are not exactly norms (after all, $\S$ is not a vector space) but should be viewed as nonlinear analogues of the more familiar Sobolev class $H^s_x(\R^2) \times H^{s-1}_x(\R^2)$ (and their localised counterpart) in the theory of scalar nonlinear wave equations.  Observe from construction that these ``norms'' are invariant with respect to Lorentz rotations.
\end{remark}
 
We have the following (reasonably standard) subcritical local existence result:

\begin{theorem}[Subcritical local existence]\label{Class}  If $\Phi = (\phi_0,\phi_1)$ is classical initial data, $s > 1$, and $t_0 \in \R$, then there exists a time $T > 0$ depending only on $s$ and $\|\Phi\|_{H^s_\loc}$, and a classical wave map $(\phi, [t_0-T,t_0+T])$ with initial data $\phi[t_0] = \Phi$.  Furthermore, we have the persistence of regularity result
\begin{equation}\label{phit-grow-energy}
\| \phi[t] \|_{{\mathcal H}^{s'}_\loc} \lesssim_{s, s', \| \Phi \|_{{\mathcal H}^s_\loc}, \| \Phi \|_{{\mathcal H}^{s'}_\loc}} 1
\end{equation}
and
\begin{equation}\label{phit-grow-energy-2}
\| \phi[t] \|_{{\mathcal H}^{s'}} \lesssim_{s, s', \| \Phi \|_{{\mathcal H}^s_\loc}, \| \Phi \|_{{\mathcal H}^{s'}}} 1
\end{equation}
for all $s' > 1$.
\end{theorem}

\begin{proof}  By time translation invariance we may take $t_0=0$.  We allow all implied constants to depend on $s$ and the quantity $\| \Phi \|_{{\mathcal H}^s_\loc}$, thus for every $x_0$ there exists $U_{x_0} \in SO(m,1)$ such that
\begin{equation}\label{eta-sog}
 \| \eta(\cdot-x_0) U_{x_0}(\phi_0) \|_{H^s_x(\R^2)} + \| \eta(\cdot-x_0) U_{x_0}(\phi_1) \|_{H^{s-1}_x(\R^2)} \lesssim 1.
\end{equation}
In particular, from Sobolev embedding we see that $U_{x_0}(\phi_0(x_1)) = O(1)$ whenever $|x_0-x_1| \leq 1$.  As a consequence, we see that if $|x_0-x_1|=O(1)$, then $U_{x_0}$ and $U_{x_1}$ only differ by a left-multiplication by a bounded element of $SO(m,1)$.

For any $x_0 \in \R^2$, we can find classical data $(\phi_{0,x_0}, \phi_{1,x_0})$ with
\begin{equation}\label{phio} 
 \| \phi_{0,x_0} \|_{H^s_x(\R^2)} +  \| \phi_{1,x_0} \|_{H^{s-1}_x(\R^2)} \lesssim 1
 \end{equation}
which agrees with $(U_{x_0} (\phi_0), U_{x_0}(\phi_1))$ on $B(x_0,1/2)$; for instance we can take
$$ \tilde \phi_0( x_0 + r \omega ) := U_{x_0}(\phi_0)( x_0 + f(r) \omega )$$
and
$$ \tilde \phi_1( x_0 + r \omega ) := g(r) U_{x_0}(\phi_1)( x_0 + f(r) \omega )$$
for $i=0,1$, $r \geq 0$, and $\omega \in S^1$, where $f: [0,+\infty) \to [0,1)$ is a smooth function with $f(r) = r$ for $r \leq 1/2$ and $f(r)=0$ for $r \geq 1$, and $g: [0,+\infty) \to [0,1]$ is a smooth function with $g(r)=1$ for $r \leq 1/2$ that vanishes for $r \geq 1$; the bound \eqref{phio} follows from \eqref{eta-sog} in the case when $s \geq 1$ is an integer by direct computation, and the general case follows by an interpolation argument (one can also use the fractional chain rule).  For similar reasons we have the more general estimate
\begin{equation}\label{phio-2} 
\begin{split}
 \| \phi_{0,x_0} \|_{H^{s'}_x(\R^2)} +  \| \phi_{1,x_0} \|_{H^{s'-1}_x(\R^2)} &\lesssim_{s'} \\ 
  \| \eta(\cdot-x_0) U_{x_0}(\phi_0) \|_{H^{s'}_x(\R^2)} &+ \| \eta(\cdot-x_0) U_{x_0}(\phi_1) \|_{H^{s'-1}_x(\R^2)} 
  \end{split}
\end{equation}
for all $s' > 1$.

We may now apply the local existence theory from\footnote{Strictly speaking, the results in that paper only handle the case when the above norm is sufficiently small, and for a model equation closely related to \eqref{wavemap-hyper}.  However, it is not too difficult to adapt the arguments there to the case at hand. Alternatively, one can use the results in \cite{tataru:wave2}.} \cite{kman.selberg}, and construct a classical local wave map $(\phi_{x_0}, [-T,T])$ for each fixed $x_0$ with initial data $\phi_{x_0}[0] = (U_{x_0}(\phi_{0,x_0}), U_{x_0}(\phi_{1,x_0}))$, for some $T \sim 1$ independent of $x_0$.  Furthermore, by \eqref{phio-2} and persistence of regularity theory in \cite{tataru:wave2}, we will have
\begin{equation}\label{phixot}
 \| \phi_{x_0}[t] \|_{H^{s'}_x(\R^2) \times H^{s'-1}_x(\R^2)} \lesssim_{s'}   \| \eta(\cdot-x_0) U_{x_0}(\phi_0) \|_{H^{s'}_x(\R^2)} + \| \eta(\cdot-x_0) U_{x_0}(\phi_1) \|_{H^{s'-1}_x(\R^2)} 
\end{equation}
for all $t \in [-T,T]$.

By the Lorentz rotation symmetry \eqref{rotate-data}, the rotated solutions $U_{x_0}^{-1} \circ \phi_{x_0}$ are classical wave maps on $[-T,T]$ with initial data that agree with $\Phi$ on $B(x_0,1/2)$.  Using finite speed of propagation (and shrinking $T$ slightly if necessary), we may thus glue all these solutions together to obtain a classical wave map $\phi$ on $[-T,T]$ with initial data $\phi[0] = \Phi$.  The bounds \eqref{phit-grow-energy}, \eqref{phit-grow-energy-2} then follow from \eqref{phixot}, the stability of Sobolev spaces under multiplication by smooth cutoff functions, and the fact mentioned earlier that $U_{x_0}$ and $U_{x_1}$ differ only by bounded rotations when $|x_1-x_0| \lesssim 1$.
\end{proof}

Standard energy arguments show that classical wave maps are uniquely determined by their initial data.  Iterating the above local existence result in the usual fashion, and taking contrapositives, we thus conclude

\begin{corollary}[Subcritical blowup criterion]\label{class} If $\Phi = (\phi_0,\phi_1)$ is classical initial data and $t_0 \in \R$, then there exists a unique \emph{maximal classical lifespan} $I \subset \R$ which is an open interval containing $t_0$, and a unique classical wave map $\phi: I \to \S$ with $\phi[t_0] = \Phi$, (which we call the \emph{maximal classical development} from the  initial data $\Phi$ at time $t_0$) such that if $t_*$ is any finite endpoint of $I$ then $\| \phi[t] \|_{{\mathcal H}^s_\loc} \to \infty$ as $t \to t_*$ for every $s > 1$.
\end{corollary}

\begin{remark} \emph{A posteriori}, we will be able to show that the maximal classical lifespan in Corollary \ref{class} is identical to the maximal energy class lifespan in Theorem \ref{lwp-claim} (indeed, this follows immediately from parts (iii) and (v) of that theorem), and so the maximal classical development is essentially the same as the maximal Cauchy development.
\end{remark}

\section{Reduction to abstract a priori estimates}\label{reduce-sec}

For each compact interval $I$ and $E>0$, let $\WM(I,E)$ denote the space of classical wave maps $\phi: I \to \S$ on $I$ of energy less than W$E$, quotiented out by the action \eqref{rotate-data} of the Lorentz rotation group $SO(m,1)$.  Note that if $J \subset I$ then every element of $\WM(I,E)$ can also be viewed as an element of $\WM(J,E)$ (indeed, by uniqueness of classical wave maps, we can embed $\WM(I,E)$ as a subset of $\WM(J,E)$).  Also, if $\phi \in \WM(I,E)$ and $t \in I$ then $\phi[t]$ can be viewed as an element of the energy space $\Energy$.  We use $\const \in \WM(I,E)$ to denote the constant wave map (the exact choice of constant is irrelevant, thanks to Lorentz rotation invariance).

In this section we use standard continuity arguments to reduce Theorem \ref{lwp-claim} to that of establishing a collection of \emph{a priori} estimates for classical wave maps with respect to various ``function space norms'' (or more precisely, metrics) on $\WM(I,E)$.  More precisely, we have 

\begin{theorem}[A priori estimates]\label{apriori-thm} There exist metrics $d_{S^1_\mu,I}$ on $\WM(I,E)$ for each compact interval $I$, $E > 0$ and $0 < \mu \leq 1$ with the following properties, where we abbreviate $\| \phi \|_{S^1_\mu(I)}$ for $d_{S^1_\mu,I}(\phi,\const)$:
\begin{itemize}
\item[(i)] (Monotonicity)  If $I \subset J$, $0 < \mu \leq 1$, and $\phi, \phi' \in \WM(J,E)$ then $d_{S^1_\mu,I}(\phi,\phi') \leq d_{S^1_\mu,J}(\phi,\phi')$.
\item[(ii)] (Continuity)  If $\phi, \phi' \in \WM([t_-,t_+])$ and $0 < \mu \leq 1$, then $d_{S^1_\mu,[a,b]}(\phi,\phi')$ is a continuous function of $a,b$ in the region $t_- \leq a < b \leq t_+$.
\item[(iii)] (Vanishing) If $I_n$ is a decreasing sequence of compact intervals with $\bigcap_n I_n = \{t_0\}$, $0 < \mu \leq 1$, and $\phi, \phi' \in \WM(I_1,E)$, then $\lim_{n \to \infty} d_{S^1_\mu,I_n}(\phi,\phi') \lesssim_E d_{\Energy}(\phi[t_0], \phi'[t_0])$.
\item[(iv)] ($S^1$ controls energy)  For any $I$, any $0 < \mu \leq 1$ and $M > 0$, any $\phi, \phi' \in \WM(I,E)$ with $\|\phi\|_{S^1_\mu(I)}, \|\phi'\|_{S^1_\mu(I)} \leq M$, and any $t \in I$, we have $d_{\Energy}(\phi[t], \phi'[t]) \lesssim_{M,E} d_{S^1_\mu,I}(\phi,\phi')$.
\item[(v)] (Persistence of regularity estimate) If $M > 0$, $0 < \mu \leq 1$ is sufficiently small depending on $M$ and $E$, $I$ is an interval, $t_0 \in I$, and $\phi \in \WM(I,E)$ is such that $\|\phi\|_{S^1_\mu(I)} \leq M$, then
\begin{equation}\label{apriori3}
\| \phi[t] \|_{{\mathcal H}^{1+\delta_0/2}_{\loc}} \lesssim_{M, E, \mu, \| \phi[t_0] \|_{{\mathcal H}^{100}_{\loc}}} 1
\end{equation}
for all $t \in I$.
\item[(vi)] (Stability estimate) If $M > 0$, $0 < \mu \leq 1$ is sufficiently small depending on $M$ and $E$, $I$ is an interval, $t_0 \in I$, and $\phi,\phi' \in \WM(I,E)$ are such that $\|\phi\|_{S^1_\mu(I)}, \|\phi'\|_{S^1_\mu(I)} \leq M$, then
\begin{equation}\label{apriori2}
d_{S^1_\mu,I}(\phi,\phi') \lesssim_{M,\mu,E} d_{\Energy}(\phi[t_0], \phi'[t_0]).
\end{equation}
\end{itemize}
\end{theorem}

\begin{remark}  Suppose, as a gross caricature, that wave maps $\phi$ behaved like scalar solutions $\phi: I \times \R^3 \to \R$ to the energy-critical nonlinear wave equation (NLW) $\Box \phi = \phi^5$. Then one possible choice of the metrics $d_{S^1_\mu,I}$ would be
\begin{align*}
d_{S^1_\mu,I}(\phi,\phi') &:=
\| \partial_{t,x}(\phi-\phi')\|_{L^\infty_t L^2_x(I \times \R^3)} + \frac{1}{\mu} \| \partial_{t,x}((\phi-\phi') \|_{L^5_t L^{30/13}_x(I \times \R^3)}.
\end{align*}
For the small energy theory, it is well known (using Strichartz estimates) that one can iterate in the $S^1_\mu$ space with $\mu=1$ in order to obtain (global) well-posedness for the energy-critical NLW.  For large energy, one can similarly iterate\footnote{Actually, this is a slight oversimplification.  The more precise statement is that boundedness in the $S^1_\mu$ space allows one to perform a contraction mapping argument in the $S^1_1$ space.} in the $S^1_\mu$ space, with $\mu$ now chosen sufficiently small depending on the energy, and with the interval $I$ chosen small enough so that the $S^1_\mu$ norm stays bounded, to obtain local existence. Thus one should view $\mu$ as a parameter designed to compensate for large energy.  In practice, the Strichartz spaces are insufficient to control the nonlinearities we will encounter, but readers who are familiar with the Strichartz theory for the NLW (or NLS) equations may find the above analogies to be helpful in what follows.  
\end{remark}

\begin{remark} The persistence of regularity estimate \eqref{apriori3} is quite crude.  One would expect in fact that the ${\mathcal H}^s_{\loc}$ norms should grow at most linearly whenever the $S^1_\mu$ norm is controlled for sufficiently small $\mu$, but this seems to require a more delicate analysis and is not actually needed for our applications, so we have elected not to establish this result.  Such a claim may however be useful if one eventually wants to obtain scattering for classical wave maps.
\end{remark}

We shall prove Theorem \ref{apriori-thm} in later sections. In the remainder of this section, we assume Theorem \ref{apriori-thm} and use it to prove Theorem \ref{lwp-claim}.  The basic lemma is the following perturbation theory result.

\begin{lemma}[Long-time perturbations]\label{ltb}  Let $I$ be a compact interval let $E>0$, and let $t_0 \in I$.  Let $\phi,\phi' \in \WM(I,E)$ be such that 
\begin{equation}\label{energy-close-bound}
d_{\Energy}(\phi[t_0],\phi'[t_0]) \leq \eps
\end{equation}
for some $\eps > 0$.  Suppose also that one can cover $I$ by intervals $I_1,\ldots,I_J$ such that
\begin{equation}\label{s1-bound}
\| \phi \|_{S^1_\mu(I_j)} \leq M
\end{equation}
for all $1 \leq j \leq J$ and some $M, \mu > 0$.  If $\mu$ is sufficiently small depending on $M, E$, and $\eps$ is sufficiently small depending on $M, J, \mu, E$, then we have
\begin{equation}\label{energy-distance}
d_{\Energy}(\phi[t], \phi'[t]) \lesssim_{M,E,J,\mu} \eps
\end{equation}
and
\begin{equation}\label{perturb-regularity}
\| \phi'[t] \|_{{\mathcal H}^{100}} \lesssim_{M,E,J,L,\mu, \| \phi'[t_0] \|_{{\mathcal H}^{100}}} 1.
\end{equation}
for all $t \in I$, where $L$ is any upper bound for the length of $I$.  
\end{lemma}

\begin{proof}  Fix $E$; we allow all constants to depend on $E$.

It suffices to verify the claim when $J=1$, since the case of higher $J$ then follows by iteration (and by energy conservation and Theorem \ref{apriori-thm}(i)).  By a further subdivision we may then assume that $t_0$ is either the upper endpoint or lower endpoint of $I$.  For sake of exposition we shall assume that $t_0 = \inf(I)$ is the lower endpoint of $I$, as the other case is of course exactly analogous.  Thus $I = [t_0,t_1]$ for some $t_1 > t_0$.  

Our main tool here will be the continuity method. Suppose that $t_0 < T \leq t_1$ is such that $\| \phi' \|_{S^1_\mu([t_0,T] \times \R^2)} \leq 3M$.  Then from \eqref{apriori2}, \eqref{energy-close-bound}, and Theorem \ref{apriori-thm}(i) we see that if $\mu$ is sufficiently small depending on $M$, then
$$ d_{S^1_\mu, [t_0,T]}( \phi, \phi' ) \lesssim_{M} \eps.$$
Since $\eps$ is assumed to be sufficiently small depending on $M$, we thus conclude from \eqref{s1-bound} and the triangle inequality that
$$ \| \phi' \|_{S^1_\mu([t_0,T] \times \R^2)} \leq 2M.$$
By Theorem \ref{apriori-thm}(ii), the set of $T \in (t_0,t_1]$ such that $\| \phi' \|_{S^1_\mu([t_0,T]\times \R^2)} \leq 2M$ is thus both closed and open in $(t_0,t_1]$.  By Theorem \ref{apriori-thm}(iii) (if $M$ is large enough depending on $E$), this set also contains all times sufficiently close to $t_0$.  As $(t_0,t_1]$ is connected, we conclude that this set is all of $(t_0,t_1]$.  In particular we have
$$ \| \phi' \|_{S^1_\mu(I)} \leq 2M.$$
Since $\mu$ is assumed to be sufficiently small depending on $M$, we can apply Theorem \ref{apriori-thm}(v) and \eqref{energy-close-bound} to conclude that
$$ d_{S^1,I}(\phi,\phi') \lesssim_{M,\mu} \eps.$$
The claim \eqref{energy-distance} then follows from Theorem \ref{apriori-thm}(iv).  To prove \eqref{perturb-regularity}, observe from Theorem \ref{apriori-thm}(v) that
$$ \sup_{t \in I} \| \phi'[t] \|_{{\mathcal H}^{1+\delta_0/2}_\loc} \lesssim_{M, \mu, \| \phi'[t_0] \|_{{\mathcal H}^{100}}} 1,$$
and the claim then follows from repeated application of Lemma \ref{Class}.
\end{proof}

\begin{corollary}[Lifespan stability]\label{ltb-cor}  Let $I$ be a compact interval, let $t_0 \in I$ and $E>0$, and let $\phi\in \WM(I,E)$.  Suppose that one can cover $I$ by intervals $I_1,\ldots,I_J$ such that $\|\phi\|_{S^1_\mu(I_j)} \leq M$ for all $1 \leq j \leq J$ and some $M, \mu > 0$.  Let $\Phi' \in \S$ be such that $d_{\Energy}(\phi[t_0],\Phi') \leq \eps$ for some $\eps > 0$.  If $\mu$ is sufficiently small depending on $M$ and $E$, and $\eps$ is sufficiently small depending on $M, J, \mu$, then there exists $\phi' \in \WM(I,E)$ such that $\phi'[t_0] = \Phi$.
\end{corollary}

\begin{proof} If the conclusion failed, then by Corollary \ref{class} there would exist an open subinterval $I' \subset I$ containing $t_0$ and a classical solution $\phi': I' \to \S$ with $\phi'[t_0] = \Phi'$ such that $\| \phi'[t]\|_{{\mathcal H}^{1+\delta_0/2}_\loc}$ (and thus $\| \phi'[t]\|_{{\mathcal H}^{100}}$) was unbounded on $I'$.  But this contradicts Lemma \ref{ltb} (restricted to $I'$).
\end{proof}

We are now ready to prove Theorem \ref{lwp-claim}, which will be a completely soft argument relying on Lemma \ref{ltb}, Corollary \ref{ltb-cor}, and gluing of intervals.

\begin{proof}[Proof of Theorem \ref{lwp-claim}]  We first observe that it suffices to establish the claim under the additional assumption that all solutions involved have energy less than an arbitrary parameter $E > 0$, since the general case then follows by letting $E \to \infty$ and using the various claims in Theorem \ref{lwp-claim} to ensure that the solutions constructed for different choices of $E$ are compatible with each other.

We now fix this energy threshold $E > 0$, choose a sufficiently large parameter $M > 0$ depending on $E$, and also choose a sufficiently small parameter $\mu$ depending on $E, M$.  We allow all implied constants to depend on $E$.

Given $\Phi_0 \in \Energy$ with $\E(\Phi_0) < E$ and $t_0 \in \R$, let us say that a compact interval $I$ containing $t_0$ is \emph{good} for this choice of initial data if there exists a sequence $\phi^{(n)} \in \WM(I,E)$ such that $\phi^{(n)}(t_0)$ converges to $\Phi_0$ in $\Energy$ (and in particular has energy less than $E$ for sufficiently large $n$, by Theorem \ref{energy-claim}(iii)), and such that $I$ can be covered by finitely many compact intervals $I_1,\ldots,I_J$, such that $\limsup_{n \to \infty} \| \phi^{(n)} \|_{S^1_\mu(I_j)} < M$ for each $1 \leq j \leq J$.  From Lemma \ref{ltb}, we see (for appropriate choices of $M, \mu$) that $\phi^{(n)}$ forms a Cauchy sequence in the uniform topology in $\Energy$ on $I$, and so (by completeness of $\Energy$) converges uniformly in $\Energy$ to a limit $\phi: I \to \Energy$, which is then continuous in $\Energy$ since each of the $\phi^{(n)}$, being classical, are automatically continuous in $\Energy$ by Theorem \ref{energy-claim}.  Another application of Lemma \ref{ltb} shows that this limit is independent of the choice of sequence $\phi^{(n)}$.  Similarly, if there are two good compact intervals $I, I'$ containing $t_0$, and $\phi: I \to \Energy$ and $\phi': I' \to \Energy$ are constructed as above from the same initial data, then another application of Lemma \ref{ltb} (applied to the common intersection $I \cap I'$ and using monotonicity) shows that $\phi$ and $\phi'$ agree on $I \cap I'$.  Finally, another application of Lemma \ref{ltb} shows that the union of finitely many good intervals is still good.  Thus, if we define the \emph{maximal lifespan} $I$ of $\Phi_0$ from $t_0$ to be the union of all the good compact intervals, we can define a unique maximal Cauchy development $\phi: I \to \Energy$ by gluing together all the partial developments $\phi: J \to \Energy$ associated to good intervals.

We now establish claim (i), namely that $I$ is an open interval containing $t_0$.  We first show that $I$ contains a neighbourhood of $t_0$.  By Theorem \ref{energy-claim}(i), there exists a sequence $\Phi^{(n)} \in \S$ that converges in $\Energy$ to $\Phi$, and thus (by Theorem \ref{energy-claim}(iii)) we have $\E(\Phi^{(n)}) < E$ after discarding at most finitely many $n$.  Let $\eps > 0$ be sufficiently small depending on $E, M, \mu$.  We find an $n$ such that $d_{\Energy}(\Phi^{(n)},\Phi) \leq \eps/2$, and then let $\phi^{(n)}: I^{(n)} \to \S$ be the maximal classical development of $\Phi^{(n)}$ from $t_0$ given by Lemma \ref{class}, thus $I^{(n)}$ is an open interval containing $t_0$.  Applying Theorem \ref{apriori-thm}(iii), we can find a compact interval $I' \subset I^{(n)}$ containing $t_0$ in its interior such that $\limsup_{n \to \infty} \| \phi^{(n)} \|_{S^1_\mu(I')} < M/2$, if $M$ is large enough depending on $E$.  Now for all sufficiently large $m$, we have $d_{\Energy}(\Phi^{(n)},\Phi^{(m)}) \leq \eps$ by the triangle inequality, so by Corollary \eqref{ltb-cor} we have a classical wave map $\phi^{(m)}: I' \to \S$ with $\phi^{(m)}[t_0] = \Phi^{(m)}$.  Inspecting the proof of Lemma \ref{ltb} we see that $\limsup_{m \to \infty} \| \phi^{(m)} \|_{S^1_\mu(I')} < M$, and so $I'$ is good.  Thus $I$ contains an open neighbourhood of $t_0$.

It remains to show that $I$ itself is open.  Let $t_1 \in I$, then by construction there exists a good interval $I'$ containing both $t_0$ and $t_1$.  Thus we can find a sequence $\phi^{(n)} \in \WM(J,E)$ which converges uniformly to $\phi$ in $\Energy$ on $I'$, in particular $\phi^{(n)}[t_1]$ converges to $\phi[t_1]$ in $\Energy$.  Applying the previous argument at time $t_1$ rather than $t_0$, we conclude that there exists a compact interval $K$ containing $t_1$ in its interior such that the maximal classical developments of $\phi^{(n)}$ exist on $K$ and $\limsup_{n \to \infty} \| \phi^{(n)} \|_{S^1_\mu(K)} < M$.  This implies that $J \cup K$ is good, and so $t_1$ is in the interior of $I$.  Thus $I$ is open as claimed.

The claim (ii) of Theorem \ref{lwp-claim} is clear from construction.  Claim (iii) follows easily from Corollary \ref{ltb-cor} (and the uniqueness of classical wave maps), so we now turn to claim (iv).  From the proof of (i) we know that every element of the maximal lifespan of a solution is contained in the interior of a good interval, and so by compactness, we see that to prove (iv), it suffices to do so in the case when $K$ is a good interval.  Let $\phi_\infty: K \to \Energy$ be the restriction of a maximal Cauchy development to a good interval $K$, with energy strictly less than $E$.  Let $\phi_\infty^{(m)} \in \WM(K)$ be a sequence converging uniformly to $\phi_\infty$ as in the definition of a good interval, thus $\E(\phi_\infty^{(m)}) < E$ for sufficiently large $m$, and one can cover $K$ by a finite number of intervals $K_1,\ldots,K_k$ such that $\limsup_{m \to \infty} \| \phi_\infty^{(m)} \|_{S^1_\mu(K_j)} < M$ for all $1 \leq j \leq k$.  Now let $\eps > 0$ be a sufficiently small quantity depending on $E, M, k, \mu$.  For $n$ large enough, $\Phi_{0,n}$ lies within $\eps/3$ of $\Phi_{0,\infty}$ in the energy metric and has energy less than $E$, and so if one writes $\Phi_{0,n}$ as the limit of smooth data $\Phi_{0,n}^{(m)}$ of energy less than $E$, we see that $\Phi_{0,n}^{(m)}$ lies within $\eps$ of $\phi_\infty^{(m')}[t_0]$ if $m,m'$ are large enough.  Applying Corollary \ref{ltb-cor} and Lemma \ref{ltb} one can then (if $\eps$ is small enough) find $\phi_n^{(m)} \in \WM(K)$ with $\phi_n^{(m)}[t_0] = \Phi_{0,n}^{(m)}$, and furthermore (by inspection of the proof of Lemma \ref{ltb}) we have $\limsup_{m \to \infty} \| \phi_n^{(m)} \|_{S^1_\mu(I)} < M$ on $O_{M,k,K}(1)$ intervals $I$ (depending on $m'$, but independent of $m$) covering $K$.  From this we see that $K$ is good for $\Phi_n$, and from Lemma \ref{ltb} we then see that the resulting Cauchy development $\phi_n: K \to \Energy$ differs from $\phi_\infty: K \to \Energy$ by $O_{M,k,K}(\eps)$ in the uniform energy metric.  Sending $\eps \to 0$ we conclude that $\phi_n$ converges uniformly to $\phi_\infty$ on $K$, and Claim (iv) follows.

Finally, we establish Claim (v).  Suppose for contradiction that there was a sequence $t_n \in I$ converging to a finite endpoint $t_* \in \R$ of $I$ such that $\phi[t_n]$ converged in $\Energy$ to a limit $\Phi_*$.  Since $\phi$ had energy less than $E$, we see that $\Phi_*$ does too.  By the arguments used to prove (i), we can find a compact interval $J$ containing $t_*$ in its interior for which there exists a sequence $\phi_*^{(m)} \in \WM(J,E)$ with $\limsup_{m \to \infty} \|\phi_*^{(m)} \|_{S^1_\mu(J)} < M$ and with $\phi_*^{(m)}(t_*)$ converging in $\Energy$ to $\Phi_*$ (and thus has energy less than $E$ for sufficiently large $m$).  As before, $\phi_*^{(m)}$ converges uniformly in $\Energy$ to a continuous limit $\phi_*: J \to \Energy$.  Since $t_n \in J$ for sufficiently large $n$, we see from continuity that $\phi_*[t_n]$ converges in $\Energy$ to $\Phi_*$, and thus by the triangle inequality $d_{\Energy}(\phi_*[t_n],\phi[t_n])$ converges to zero.  

Let $\eps$ be sufficiently small depending on $M$, $E$ and $\mu$.  Then for $n$ large enough, we have
$d_{\Energy}(\phi_*[t_n],\phi[t_n]) \leq \eps/2$, and thus if $\Phi^{(m)}_n \in \S$ is a sequence of classical data converging in $\Energy$ to $\phi[t_n]$ (and thus with energy strictly less than $E$ for sufficiently large $n$) then $d_{\Energy}(\phi^{(m')}_*[t_n],\Phi^{(m)}_n) \leq \eps$ whenever $m, m'$ are sufficiently large.  Applying Corollary \ref{ltb-cor} and Lemma \ref{ltb}, we see (if $\eps$ is small enough) that there is $\phi^{(m)}_n \in \WM(J,E)$ with $\phi^{(m)}_n[t_n] = \Phi^{(m)}_n$ such that $\limsup_{m \to \infty} \| \phi^{(m)}_n \|_{S^1_\mu(J)} < M$.  From this (and the uniqueness of classical wave maps) we see that if $K$ is any good interval for $\phi$ that contains both $t_0$ and $t_n$, then $K \cup J$ is also good for $\phi$.  Thus $J \subset I$, and therefore $t_*$ is in fact an interior point of $I$, a contradiction.  This establishes Claim (v), and the proof of Theorem \ref{lwp-claim} is complete.
\end{proof}

It remains to prove Theorem \ref{apriori-thm}.  This is the purpose of the remaining sections of this paper.

\section{The caloric gauge and the energy space}\label{energy-sec}

Our arguments thus far have been extremely abstract; the exact form of the wave maps equation, or even of the energy space $\Energy$, has not played any real role in previous sections, beyond such basic properties as energy conservation and classical well-posedness.  To proceed further, we will have to use the specific structure of the equation and the energy space.  In this section we recall from \cite{tao:heatwave2} the construction of the energy space, and in particular the \emph{caloric gauge} that is used in that construction.  We refer the reader to \cite{tao:heatwave2} for a more detailed treatment of the material here.

\subsection{The caloric gauge}

We first recall a global existence theorem for the harmonic map heat flow, essentially due to Eells and Sampson \cite{eells}.

\begin{theorem}[Global existence for heat flow]\label{ghp}  Let $I$ be a compact interval, and let $\phi: I \times \R^2 \to \H$ be a smooth map which differs from a constant $\phi(\infty)$ by a Schwartz function.  Then there exists a unique extension $\phi: \R^+ \times I \times \R^2 \to H$ of $\phi$ (thus $\phi(0,t,x) = \phi(t,x)$ for all $(t,x) \in I \times \R^2$) which is smooth with all derivatives bounded, and obeys the harmonic map heat flow equation
$$ \partial_s \phi = (\phi^* \nabla)_i \partial_i \phi$$
on $\R^+ \times I \times \R^2 := \{ (s,t,x): s \in \R^+, t \in I, x \in \R^2\}$, and converges in $C^\infty_{\loc}(I \times \R^2)$ to $\phi(\infty)$ as $s \to \infty$.
\end{theorem}

\begin{proof} See \cite[Theorem 3.16]{tao:heatwave2}.
\end{proof}

This result allows us to extend any classical wave map $\phi \in WM(I)$ by harmonic map heat flow into the region $\R^+ \times I \times \R^2$.  In order to analyse this heat flow we will place an orthonormal frame on this map (or more precisely on the pullback tangent bundle $\phi^* T\H$).

Given any point $p \in \H$, define an \emph{orthonormal frame} at $p$ to be any orthogonal orientation-preserving map $e: \R^m \to T_p \H$ from $\R^m$ to the tangent space at $p$ (with the metric $h(p)$, of course), and let $\Frame(T_p \H)$ denote the space of such frames; note that this space has an obvious transitive action of the special orthogonal group $SO(m)$.  We then define the \emph{orthonormal frame bundle} $\Frame( \phi^* T\H)$ of $\phi$ to be the space of all pairs $((s,t,x), e)$ where $(s,x) \in \R^+ \times I \times \times \R^2$ and $e \in \Frame(T_{\phi(s,x)}\H)$; this is a smooth vector bundle over $\R^+ \times I \times \R^2$.  We then define an \emph{orthonormal frame} $e \in \Gamma(\Frame(\phi^* T\H))$ for $\phi$ to be a section of this bundle, i.e. a smooth assignment $e(s,x) \in \Frame(T_{\phi(s,x)}\H)$ of an orthonormal frame at $\phi(s,x)$ to every point $(s,x) \in \R^+ \times I \times \R^2$.

Each orthonormal frame $e \in \Gamma(\Frame(\phi^* T\H))$ provides an orthogonal, orientation-preserving identification between the vector bundle $\phi^* T \H$ (with the metric $\phi^* h$) and the trivial bundle $(\R^+ \times I \times \R^2) \times \R^m$ (with the Euclidean metric on $\R^m$), thus sections $\Psi \in \Gamma(\phi^* T\H)$ can be pulled back to functions $e^* \Psi: \R^+ \times I \times \R^2 \to \R^m$ by the formula $e^* \Psi := e^{-1} \circ \Psi$.  The connection $\phi^* \nabla$ on $\phi^* T\H$ can similarly be pulled back to a connection $D$ on the trivial bundle $(\R^+ \times \R^2) \times \R^m$, defined by
\begin{equation}\label{D-def}
 D_i := \partial_i + A_i
 \end{equation}
where $A_i \in \mathfrak{so}(m)$ is the skew-adjoint $m \times m$ matrix field is given by the formula
\begin{equation}\label{Adef}
(A_i)_{ab} = \langle (\phi^* \nabla)_i e_a, e_b \rangle_{\phi^* h}
\end{equation}
where $e_1,\ldots,e_m$ are the images of the standard orthonormal basis for $\R^m$ under $e$.  Of course one similarly has covariant derivatives $D_t = \partial_t + A_t$,  $D_s = \partial_s + A_s$ in the $t$ and $s$ directions.  

Given such a frame, we define the \emph{derivative fields}
$\psi_j: \R^+ \times I \times \R^2 \to \R^m$ by the formula 
\begin{equation}\label{psij-def}
\psi_j := e^* \partial_j \phi,
\end{equation}
and similarly define 
\begin{equation}\label{psis-def}
\psi_s := e^* \partial_s \phi; \quad \psi_t := e^* \partial_t \phi.
\end{equation}
We record the zero-torsion property
\begin{equation}\label{zerotor-frame}
D_i \psi_j = D_j \psi_i
\end{equation}
and the constant negative curvature property
\begin{equation}\label{curv-frame}
[D_i, D_j] = \partial_i A_j - \partial_j A_i + [A_i,A_j] = - \psi_i \wedge \psi_j
\end{equation}
where $\psi_i \wedge \psi_j$ is the anti-symmetric matrix field
$$ \psi_i \wedge \psi_j := \psi_i \psi_j^* - \psi_j \psi_i^*.$$
The harmonic map heat flow equation becomes
\begin{equation}\label{heatflow}
\psi_s = D_i \psi_i = \partial_i \psi_i + A_i \psi_i.
\end{equation}
The wave maps equation \eqref{cov} becomes
\begin{equation}\label{w-vanish}
w|_{s=0} = 0
\end{equation}
where $w: \R^+ \times I \times \R^2 \to \R^m$ is the \emph{wave-tension field}
\begin{equation}\label{w-def}
w := D^\alpha \psi_\alpha = \partial^\alpha \psi_\alpha + A^\alpha \psi_\alpha.
\end{equation}

Following \cite{tao:heatwave2}, we say that an orthonormal frame $e$ is a \emph{caloric gauge} if one has 
\begin{equation}\label{ass}
A_s = 0
\end{equation}
throughout $\R^+ \times I \times \R^2$, and if we have 
\begin{equation}\label{esx}
\lim_{s \to \infty} e(s,t,x) = e(\infty)
\end{equation}
for all $x \in \R^2$ and some constant $e(\infty) \in \Frame(T_{\phi(\infty)}\H)$.  

We have the following existence theorem for this gauge:

\begin{theorem}[Existence of caloric gauge]\label{dynamic-caloric}  Let $I$ be a time interval with non-empty interior, let $\phi: I \times \R^2 \to \H$ be a smooth map differing from a constant $\phi(\infty)$ by a Schwartz function, and let $e(\infty) \in \Frame(T_{\phi(\infty)}\H)$ be a frame for $\phi(\infty)$.  Let $\phi: \R^+ \times I \times \R^2 \to \H$ be the heat flow extension from Theorem \ref{ghp}.  Then there exists a unique smooth frame $e \in \Gamma(\Frame(\phi^* T\H))$ such that $e(t)$ is a caloric gauge for $\phi(t)$ which equals $e(\infty)$ at infinity for each $t \in I$.   All derivatives of $\phi - \phi(\infty)$ in the variables $t,x,s$ are Schwartz in $x$ for each fixed $t,s$.  In particular, these derivatives are uniformly Schwartz for $t, s$ in a compact range.
\end{theorem}

\begin{proof} See \cite[Theorem 3.16]{tao:heatwave2}.
\end{proof}

In the caloric gauge we have the derivative fields $\psi_t, \psi_x, \psi_s$ and the connection fields $A_t$, $A_x$ (recall in the caloric gauge that the field $A_s$ is trivial).  It will be convenient to introduce the combined vector or tensor fields
\begin{align*}
\Psi_x &:= (\psi_x, A_x) \\
\psi_{t,x} &:= (\psi_t, \psi_x) \\
A_{t,x} &:= (A_t, A_x) \\
\Psi_{t,x} &:= (\psi_{t,x}, A_{t,x}) \\
\Psi_{s,t,x} &:= (\psi_s, \psi_{t,x}, A_{t,x}).
\end{align*}
We refer to $\Psi_{s,t,x}$ as the \emph{differentiated fields} of $\phi$ in the caloric gauge associated with $e(\infty)$.  We also introduce the \emph{wave-tension field} $w := D^\alpha \psi_\alpha$; the wave map equation \eqref{cov} is then equivalent to the vanishing of this quantity on the $s=0$ boundary of $\R^+ \times I \times \R^2$.  We recall the basic equations of motion for these fields:

\begin{proposition}[Equations of motion]\label{abound}  Let $I$ be an interval, let $\phi \in \WM(I,E)$ with energy $\E(\phi) \leq E$, let $e$ be a caloric gauge for $\phi$, and let $\phi_t, \psi_x, \psi_s, A_t, A_x, w$ be the differentiated fields, connection fields, and wave-tension field.  Then we have the equations of motion
\begin{align}
A_{t,x}(s,t,x) &= \int_s^\infty \psi_s \wedge \psi_{t,x}(s',t,x)\ ds' \label{a-eq}\\
\psi_{t,x}(s,t,x) &= -\int_s^\infty D_{t,x} \psi_s(s',t,x)\ ds' \label{psi-eq}\\
\partial_s \psi_s &= D_i D_i \psi_s - (\psi_s \wedge \psi_i) \psi_i \label{psis-eq}\\
\partial_s \psi_{t,x} &= D_i D_i \psi_{t,x} - (\psi_{t,x} \wedge \psi_i) \psi_i \label{psit-eq}\\
\partial_s w &= D_i D_i w - (w \wedge \psi_i) \psi_i + 2 (\psi_\alpha \wedge \psi_i) D_i \psi^\alpha	 \label{w-eq} \\
D^\alpha D_\alpha \psi_s &= \partial_s w - (\psi_\alpha \wedge \psi_s) \psi^\alpha.   \label{psis-box}
\end{align}
\end{proposition}

\begin{proof} The equations \eqref{a-eq}, \eqref{psi-eq}, \eqref{psis-eq}, \eqref{psit-eq} follow from \cite[Lemma 3.11]{tao:heatwave2} and the qualitative decay properties at $s=\infty$ from \cite[Theorem 3.16]{tao:heatwave2}.  The equation \eqref{w-eq} follows from \cite[Lemma 7.1]{tao:heatwave2} (noting from \eqref{zerotor-frame} that $(\psi_\alpha \wedge \psi_i) D_i \psi^\alpha = - (\psi_t \wedge \psi_i) D_t \psi_i$).  Finally, \eqref{psis-box} follows from \eqref{zerotor-frame}, \eqref{curv-frame}, and the definition of $w$.
\end{proof}

\begin{remark}  As a caricature, in which one pretends $\phi$ is a scalar field rather than taking values in $\H$, the reader is encouraged to use the heuristics $\psi_\alpha(s,t) \approx \partial_\alpha e^{s\Delta} \phi(t)$, and $\psi_s(s,t) \approx \Delta e^{s\Delta} \phi(t)$.  Combining these heuristics with \eqref{a-eq}, we see that $A_\alpha$ should behave like a paraproduct of $\phi$ with itself; see \cite{tao:forges} or \cite[Chapter 6]{tao:cbms} for further discussion. 
\end{remark}

\begin{remark} The six evolution equations \eqref{a-eq}-\eqref{psis-box} will be used in different ways.  The equations \eqref{a-eq}, \eqref{psi-eq} effectively express the fields $A_\alpha, \psi_\alpha$ in terms of $\psi_s$, thus in principle converting the wave maps equation into a ``scalar'' equation involving $\psi_s$ as the only ``dynamic'' field (cf. the ``dynamic separation'' used in \cite{krieger:2d}).   We will evolve this field by the wave equation \eqref{psis-box}, using the parabolic equation \eqref{w-eq} to solve for the $\partial_s w$ forcing term appearing in that wave equation.  Finally, the remaining equations \eqref{psis-eq}, \eqref{psit-eq} will be used to compute the initial data $\psi_s(s,0,x)$, $D_t \psi_s(s,0,t) = \partial_s \psi_t(s,0,t)$ of the wave equation \eqref{psis-box}.
\end{remark}

\begin{remark} It will be convenient to unify \eqref{a-eq}, \eqref{psi-eq} in the schematic form
\begin{equation}\label{Psi-tx-eq}
\Psi_{t,x}(s) = \int_s^\infty \bigO( \partial_{t,x} \psi_s(s') ) + \bigO( \psi_s(s') \Psi_{t,x}(s') )\ ds'.
\end{equation}
Similarly we have
\begin{equation}\label{Psi-x-eq}
\Psi_x(s) = \int_s^\infty \bigO( \partial_x \psi_s(s') ) + \bigO( \psi_s(s') \Psi_x(s') )\ ds'.
\end{equation}
These integral ODE thus, in principle, recover $\Psi_x$ or $\Psi_{t,x}$ from $\psi_s$ (assuming sufficient decay at $s=\infty$).  In a similar spirit, we have the identities
\begin{align}
\partial_s \Psi_{t,x} &= \bigO( \partial_{t,x} \psi_s ) + \bigO( \psi_s \Psi_{t,x} )\label{inter1} \\
\partial_{t,x} \psi_s &= \bigO( \partial_s \Psi_{t,x} ) + \bigO( \psi_s \Psi_{t,x} )\label{inter2} \\
\partial_{t,x} \Psi_x &= \partial_x \Psi_{t,x} + \bigO( \Psi_x \Psi_{t,x} )\label{inter3}\\
\psi_s &= \bigO( \partial_x \Psi_x ) + \bigO( \Psi_x^2 )\label{inter4}
\end{align}
arising from \eqref{zerotor-frame}, \eqref{curv-frame}, \eqref{heatflow}, and similarly with the $t$ subscripts dropped.
\end{remark}

\subsection{Working in the caloric gauge}

Let $\phi \in \WM(I,E)$ be a classical wave map on a compact time interval $I$, quotiented out by Lorentz rotations $SO(m,1)$.  By Theorem \ref{dynamic-caloric}, given any such wave map, and given any frame $e(\infty)$ at $e(\infty) \in \Frame(T_{\phi(\infty)}\H)$, one can construct a caloric gauge for $\phi$, which then generates the differentiated fields $\Psi_{s,t,x}$.  We let $\WMC(\phi,I)$ denote the collection of all the possible fields $\Psi_{s,t,x}$ that can arise for a fixed $\phi$ by choosing $e(\infty)$, and let $\WMC(I,E) = \bigcup_{\phi \in \WM(I,E)} \WMC(\phi,I)$ denote all the differentiated fields that can arise from some classical wave map $\phi \in \WM(I,E)$.  

Observe that if $U \in SO(m)$ is a rotation, then replacing $e(\infty)$ by $e(\infty) \circ U^{-1}$ corresponds to rotating the fields $\psi_{t,x}$ to $U \psi_{t,x}$, and similarly rotating $A_{t,x}$ to $U A_{t,x} U^{-1}$.  This gives an action of $SO(m)$ on $\WMC(I,E)$, whose orbits are precisely the sets $\WMC(\phi,I)$.  Thus each classical wave map $\phi \in \WM(I,E)$ gives rise to an $SO(m)$-orbit $\WMC(\phi,I)$ of differentiated fields.  Conversely, it is not hard to see (from the Picard uniqueness theorem) that each differentiated field $\Psi_{s,t,x} \in \WMC(I,E)$ arises from exactly one wave map $\phi \in \WM(I,E)$ (quotienting out by $SO(m,1)$ as usual).  So we have an identification 
$$\WM(I,E) \equiv SO(m)\backslash \WMC(I,E).$$

We can play similar games with the initial data space $\S$: if $\Phi = (\phi_0,\phi_1) \in \S$ is classical initial data, we can construct associated caloric gauges $\Psi_{s,t,x}$ on $\R^+ \times \R^2$ for each choice of frame $\e(\infty) \in \Frame(T_{\phi_0(\infty)} \H)$, for instance by extending $\Phi$ for a short amount of time and then using Theorem \ref{dynamic-caloric}, or by using \cite[Theorem 3.12]{tao:heatwave2}, \cite[Lemma 4.8]{tao:heatwave2}.  We let $\SC(\Phi)$ denote the collection of all such differentiated fields for a fixed $\Phi$, and $\SC := \bigcup_{\Phi \in \S} \SC(\Phi)$ denote the total collection of such fields. Once again, the orthogonal group $SO(m)$ acts on $\SC$, with orbits $\SC(\Phi)$, with the orbit determining $\Phi$ up to the action of $SO(m,1)$, so that we have an identification 
$$SO(m,1) \backslash \S \equiv SO(m) \backslash \SC.$$
Also observe that if $\Psi_{s,t,x} \in \WMC(\phi,I)$ for some $\phi \in \WM(I,E)$, then $\Psi_{s,t,x}(t_0) \in \SC(\phi[t_0])$ for all $t_0 \in I$.  Finally, the constant data $\const \in \S$ corresponds to the zero differentiated field $0 \in \SC$, thus $\SC(\const) = \{0\}$ (and similarly $\WMC(\const) = \{0\}$).

\subsection{The energy space}\label{energydef-sec}

We introduce the Littlewood-Paley space
\begin{equation}\label{ldef}
 {\mathcal L} := L^2( \R^+ \times \R^2 \to \R^m, dx ds ) \times L^2( \R^2 \to \R^m, \frac{1}{2} dx ).
\end{equation}
We can define an energy metric $d_{\Energy}$ on $\SC$ by the formula
\begin{equation}\label{psipsi}
\begin{split}
d_\Energy( \Psi_{s,t,x}, \Psi'_{s,t,x} ) &:= \| ( \psi_s - \psi'_s, \psi_t(0) - \psi'_t(0) ) \|_{\mathcal L} \\
&= \left(\int_0^\infty \| \delta \psi_s(s) \|_{L^2_x(\R^2)}^2\ ds + \frac{1}{2} \| \delta \psi_t(0) \|_{L^2_x(\R^2)}^2\right)^{1/2}.
\end{split}
\end{equation}
We also abbreviate $d_{\Energy}(\Psi_{s,t,x},0)^2$ as $\E(\Psi_{s,t,x})$, and refer to this as the \emph{energy} of the differentiated field $\Psi_{s,t,x}$.

Observe that this metric is preserved by the action of $SO(m)$.  It thus descends to a metric on $SO(m,1) \backslash \S \equiv SO(m) \backslash \SC$ in the usual manner; by abuse of notation we also denote this metric as $d_\Energy$, thus
$$
d_\Energy(\Phi, \Phi') := \inf_{\Psi_{s,t,x} \in \SC(\Phi)} \inf_{\Psi'_{s,t,x} \in \SC(\Phi')}
d_\Energy(\Psi_{s,t,x}, \Psi'_{s,t,x} ).
$$
We define the energy space\footnote{This definition is equivalent to that in \cite{tao:heatwave2}, though arranged slightly differently.} $\Energy$ to be the completion of $SO(m,1) \backslash \S$ using the metric $\Energy$, and $\iota$ to be the quotient map from $\S$ to $SO(m,1) \backslash \S$.    With these definitions, Theorem \ref{energy-claim} was established in \cite{tao:heatwave2}.  Observe from Theorem \ref{energy-claim}(iv) that the energy of a map $\Phi \in \S$ is equal to the energy of any of its differentiated fields $\Psi_{s,t,x} \in \SC(\Phi)$.

\subsection{Lifting Theorem \ref{apriori-thm} to the caloric gauge}

We will be working exclusively in the caloric gauge in order to prove Theorem \ref{apriori-thm}.  Because of this, it will be convenient to deduce Theorem \ref{apriori-thm} from a variant involving differentiated wave maps $\WMC(I,E)$ in the caloric gauge, rather than classical wave maps $\WM(I,E)$.  More precisely we deduce Theorem \ref{apriori-thm} from

\begin{theorem}[A priori estimates in the caloric gauge]\label{apriori-thm2} There exist metrics $d_{S^1_\mu,I}$ on $\WMC(I,E)$ for each compact interval $I$, $E>0$, and $0 < \mu \leq 1$ with the following properties, where we abbreviate $\| \Psi_{s,t,x} \|_{S^1_\mu(I)}$ for $d_{S^1_\mu,I}(\Psi_{s,t,x},0)$:
\begin{itemize}
\item[(i)] (Monotonicity)  If $I \subset J$, $0 < \mu \leq 1$, and $\Psi_{s,t,x}, \Psi'_{s,t,x} \in \WM(J,E)$ then $d_{S^1_\mu,I}(\Psi_{s,t,x},\Psi'_{s,t,x}) \leq d_{S^1_\mu,J}(\Psi_{s,t,x},\Psi'_{s,t,x})$.
\item[(ii)] (Continuity)  If $\Psi_{s,t,x}, \Psi'_{s,t,x} \in \WMC([t_-,t_+])$ and $0 < \mu \leq 1$, then $d_{S^1_\mu,[a,b]}(\Psi_{s,t,x},\Psi'_{s,t,x})$ is a continuous function of $a,b$ in the region $t_- \leq a < b \leq t_+$.
\item[(iii)] (Vanishing) If $I_n$ is a decreasing sequence of compact intervals with $\bigcap_n I_n = \{t_0\}$, $0 < \mu \leq 1$, and $\Psi_{s,t,x}, \Psi'_{s,t,x} \in \WMC(I_1)$, then $\lim_{n \to \infty} d_{S^1_\mu,I_n}(\Psi_{s,t,x},\Psi'_{s,t,x}) \lesssim_E d_{\Energy}(\Psi_{s,t,x}(t_0), \Psi'_{s,t,x}(t_0))$.
\item[(iv)] ($S^1$ controls energy)  For any $I$, any $0 < \mu \leq 1$ and $M>0$, any $\Psi_{s,t,x}, \Psi'_{s,t,x} \in \WMC(I,E)$ with $\|\Psi_{s,t,x}\|_{S^1_\mu(I)}, \|\Psi'_{s,t,x}\|_{S^1_\mu(I)} \leq M$, and any $t \in I$, we have $d_{\Energy}(\Psi_{s,t,x}(t), \Psi'_{s,t,x}(t)) \lesssim_{M,E} d_{S^1_\mu,I}(\Psi_{s,t,x},\Psi'_{s,t,x})$.
\item[(v)] (Persistence of regularity estimate) If $M > 0$, $0 < \mu \leq 1$ is sufficiently small depending on $M$ and $E$, $I$ is an interval, $t_0 \in I$, and $\phi \in \WM(I,E)$ is such that $\|\Psi_{s,t,x}\|_{S^1_\mu(I)} \leq M$ for some $\Psi_{s,t,x} \in \WMC(\phi,I)$, then
\begin{equation}\label{apriori3a}
\| \phi[t] \|_{{\mathcal H}^{1+\delta_0/2}_{\loc}} \lesssim_{M, E, \mu, \| \phi[t_0] \|_{{\mathcal H}^{100}_{\loc}}} 1
\end{equation}
for all $t \in I$.
\item[(vi)] (Stability estimate) If $M > 0$, $0 < \mu \leq 1$ is sufficiently small depending on $M$ and $E$, $I$ is an interval, $t_0 \in I$, and $\Psi_{s,t,x}, \Psi'_{s,t,x} \in \WM(I,E)$ are such that $\|\Psi_{s,t,x}\|_{S^1_\mu(I)}, \|\Psi'_{s,t,x}\|_{S^1_\mu(I)} \leq M$, then
\begin{equation}\label{apriori2a}
d_{S^1_\mu,I}(\Psi_{s,t,x},\Psi'_{s,t,x}) \lesssim_{M,\mu,E} d_{\Energy}(\Psi_{s,t,x}(t_0), \Psi'_{s,t,x}(t_0)).
\end{equation}
\item[(vii)] (Quasi-rotation invariance)  If $U \in SO(m)$ and $\Psi_{s,t,x}, \Psi'_{s,t,x} \in \WM(I,E)$, then
\begin{equation}\label{quasi-eq}
d_{S^1_\mu,I}(U \Psi_{s,t,x},U \Psi'_{s,t,x}) \sim d_{S^1_\mu,I}(\Psi_{s,t,x},\Psi'_{s,t,x}).
\end{equation}
\end{itemize}
\end{theorem}

Indeed, if we have metrics $d_{S^1_\mu,I}$ on $\WMC(I,E)$ with the above properties, we can define metrics $d_{S^1_\mu,I}$ on $\WM(I,E)$ by the usual Hausdorff distance construction
\begin{align*}
 d_{S^1_\mu,I}(\phi, \phi') &:= \max( \sup_{\Psi_{s,t,x} \in \WMC(\phi,I)} \inf_{\Psi' \in \WMC(\phi',I)} d_{S^1_\mu,I}(\Psi_{s,t,x}, \Psi'_{s,t,x}), \\
 &\quad \sup_{\Psi' \in \WMC(\phi',I)} \inf_{\Psi \in \WMC(\phi,I)} d_{S^1_\mu,I}(\Psi_{s,t,x}, \Psi'_{s,t,x}) );
\end{align*}
the verification that this construction allows us to deduce Theorem \ref{apriori-thm} from Theorem \ref{apriori-thm2} is routine but somewhat tedious and is omitted.

It remains to establish Theorem \ref{apriori-thm2}.  This is the purpose of the remaining sections of the paper.

\section{Parabolic regularity}\label{parab-sec}

The equations of motion \eqref{psis-eq}, \eqref{psit-eq} show that $\Psi_{s,t,x}$ obeys a parabolic equation in the $s$ variable.  Parabolic regularity theory then suggests that these fields should become increasingly smooth in $s$ as $s \to \infty$.  The purpose of this section is to present an abstract formalisation of this principle, which we will use twice, once to control initial data, and again to control the solution itself.  In order to establish the stability estimate \eqref{apriori2a}, we will also need to adapt this regularity theory to differences between two solutions to the same parabolic equation.

Recall that $\Sch(I \times \R^2)$ is the space of smooth functions $\phi: I \times \R^2 \to \C$ such that all derivatives are rapidly decreasing in space (with the usual topology).  For each $s>0$, recall that the integer $k(s)$ is defined by \eqref{ks-def}.

\begin{definition}[Algebra family]\label{algebra-fam}  Let $I$ be a compact interval.  An \emph{algebra family} is a collection of continuous seminorms $S_k(I \times \R^2)$ on $\Sch(I \times \R^2)$ that obeys the product estimate that obeys the following estimates:
\begin{itemize}
\item The product estimate
\begin{equation}\label{prod1-abstract}
 \| \phi^{(1)} \phi^{(2)} \|_{S_{\max(k_1,k_2)}(I \times \R^2)} \lesssim \| \phi^{(1)} \|_{S_{k_1}(I \times \R^2)} \| \psi^{(2)} \|_{S_{k_2}(I \times \R^2)}
\end{equation}
for all $k_1,k_2 \in \Z$ and $\phi^{(1)}, \phi^{(2)} \in \Sch(I \times \R^2)$;
\item The parabolic regularity estimate
\begin{equation}\label{parreg-abstract}
 \| \partial_x^2 \phi(s) \|_{S_k(I \times \R^2)} \lesssim s^{-1} \| \phi(s/2) \|_{S_k(I \times \R^2)} + \sup_{s/2 \leq s' \leq s} \| (\partial_s - \Delta) \phi(s') \|_{S_k(I \times \R^2)}
\end{equation}
for all smooth $\phi: \R^+ \to \Sch(I \times \R^2)$ and $s > 0$, as well as the variant
\begin{equation}\label{parreg-abstract-alt}
 \| \partial_x^2 \phi(s) \|_{S_k(I \times \R^2)} \lesssim \| \partial_x^2 \phi(s'') \|_{S_k(I \times \R^2)} + \sup_{s'' \leq s' \leq s} \| (\partial_s - \Delta) \phi(s') \|_{S_k(I \times \R^2)}
\end{equation}
for all smooth $\phi: \R^+ \to \Sch(I \times \R^2)$ and $s/2 \leq s'' \leq s$;
\item The additional parabolic regularity estimate
\begin{equation}\label{parreg-abstract2}
 \| \partial_x^j e^{(s-s')\Delta} \phi \|_{S_{k(s)}(I \times \R^2)} \lesssim_j (s'/s)^{\delta_1/10} (s-s')^{-j/2} \| \phi \|_{S_{k(s')}(I \times \R^2)}
\end{equation}
for all smooth $\phi \in \Sch(I \times \R^2)$, $s > s' \geq 0$, and $j \geq 0$;
\item The comparability estimate
\begin{equation}\label{physical-abstract}  \| \phi\|_{S_{k_1}(I \times \R^2)} \lesssim \chi_{k_1=k_2}^{-\delta_1} \|\phi\|_{S_{k_2}(I \times \R^2)}
\end{equation}
for all $k_1,k_2 \in \Z$ and $\phi \in \Sch(I \times \R^2)$. 
\item The uniform bound
\begin{equation}\label{uniform} \|\phi\|_{L^\infty_{t,x}(I \times \R^2)} \lesssim \| \phi \|_{S_k(I \times \R^2)}
\end{equation}
for all $k\in \Z$ and $\phi \in \Sch(I \times \R^2)$.
\end{itemize}
\end{definition}

\begin{example}  For a fixed time $t_0$, the norms
$$ \|f\|_{S_k(I \times \R^2)} := \sup_{k' \in \Z} \chi_{k=k'}^{\delta_1} \sum_{j=0}^{10} 2^{-k'j} \| \partial_x^{j+1} P_{k'} f(t_0) \|_{L^2_x(\R^2)} + 
2^{-k'j} \| \partial_x^{j} P_{k'} f(t_0) \|_{L^\infty_x(\R^2)}$$
form an algebra family.
\end{example}

\begin{theorem}[Abstract parabolic regularity]\label{parab-thm}  Let $I$ be an interval, $S_k(I \times \R^2)$ be an algebra family.
Let $\Psi_{s,t,x}, \Psi'_{s,t,x} \in \WMC(I,E)$ be differentiated fields, and let $c, \delta c$ be frequency envelopes, with $c$ having energy at most $E$ for some $E>0$.  Suppose we have the bounds
\begin{equation}\label{psisjk-abstract}
\| \partial_x^j \psi^*_s(s) \|_{S_{k(s)}(I \times \R^2)} \lesssim c(s) s^{-(j+2)/2}
\end{equation}
and
\begin{equation}\label{psisjk-star-abstract}
\| \partial_x^j \delta \psi_s(s) \|_{S_{k(s)}(I \times \R^2)} \lesssim \delta c(s) s^{-(j+2)/2}
\end{equation}
for all $0 \leq j \leq 10$ and $s > 0$, where $k(s)$ is defined in \eqref{ks-def}.  Suppose also that we have the qualitative decay properties
\begin{align}
\lim_{s \to \infty} s^{(j+1)/2} \| \partial_x^j \Psi^*_x(s) \|_{S_{k(s)}(I \times \R^2)} &= 0 \label{psi-decay-1}
\end{align}
for all $j > 0$.  Then we have
\begin{align}
\| \partial_x^j \Psi^*_x(s) \|_{S_{k(s)}(I \times \R^2)} &\lesssim_{E,j} c(s) s^{-(j+1)/2} \label{parab-1}\\
\| \partial_x^j A^*_x(s) \|_{S_{k(s)}(I \times \R^2)} &\lesssim_{E,j} c(s)^2 s^{-(j+1)/2}\label{parab-2}\\
\| \partial_x^j \psi^*_s(s) \|_{S_{k(s)}(I \times \R^2)} &\lesssim_{E,j} c(s) s^{-(j+2)/2}\label{parab-3}\\
\| \partial_x^j \delta \Psi_x(s) \|_{S_{k(s)}(I \times \R^2)} &\lesssim_{E,j} \delta c(s) s^{-(j+1)/2}\label{parab-4}\\
\| \partial_x^j \delta A_x(s) \|_{S_{k(s)}(I \times \R^2)} &\lesssim_{E,j} c(s) \delta c(s) s^{-(j+1)/2}\label{parab-5}\\
\| \partial_x^j \delta \psi_s(s) \|_{S_{k(s)}(I \times \R^2)} &\lesssim_{E,j} \delta c(s) s^{-(j+2)/2}\label{parab-6}
\end{align}
for all $j \geq 0$ and $s > 0$.  One can also obtain similar bounds with $\partial_x^2$ replaced by $\partial_s$ as one wishes (e.g. one could replace $\partial_x^j$ by $\partial_x^{j-4} \partial_s^2$ if $j \geq 4$).
\end{theorem}

\begin{remark} We have stated these bounds for arbitrary non-negative $j$, but in practice we will only need finitely many of these bounds (e.g. it would suffice to establish these bounds just in the range $0 \leq j \leq 100$).  Similarly for the remainder of the estimates in this section.  Thus, the reader may safely ignore the dependence of $j$ in the implied constants.  For applications to the subsequent paper \cite{tao:heatwave5}, it is worth remarking that the wave map equation $D^\alpha \psi_\alpha=0$ is not actually used at all in the argument here.  We also make the technical remark that the connection $A_x$ need not be Schwartz (see \cite[Remark 3.14]{tao:heatwave2}), but in that case one can interpret the above bounds in the completion of the Schwartz space with respect to the $S_k$ norm.
\end{remark}

The rest of this section is devoted to the proof of Theorem \ref{parab-thm}.  Let the notation and hypotheses be as in that theorem.  We allow implied constants to depend on $E$, thus from \eqref{cse-sup} we have
\begin{equation}\label{envelo-2}
\sup_{s>0} c(s) + \int_0^\infty c(s)^2 \frac{ds}{s} \lesssim 1.
\end{equation}
Note that as the semi-norm $S_{k}(I \times \R^2)$ is continuous in the $\Sch(I \times \R^2)$ topology, we can establish Minkowski's inequality
\begin{equation}\label{mink}
 \| \int_{s_1}^{s_2} \psi(s)\ ds\|_{S_k(I \times \R^2)} \leq \int_{s_1}^{s_2} \| \psi(s) \|_{S_k(I \times \R^2)} \ ds
\end{equation}
whenever the integral is absolutely convergent in the $\Sch(I \times \R^2)$ topology.

We begin with a bound on $\Psi_x^*$ and $\delta \Psi_x$.

\begin{lemma}\label{jreg}  We have
\begin{equation}\label{psix-star}
 \| \partial_x^j \Psi^*_x(s) \|_{S_{k(s)}(I \times \R^2)} \lesssim c(s) s^{-(j+1)/2}
 \end{equation}
and
\begin{equation}\label{psix-diff-eq} \| \partial_x^j \delta \Psi_x(s) \|_{S_{k(s)}(I \times \R^2)} \lesssim \delta c(s) s^{-(j+1)/2}
\end{equation}
for all $0 \leq j \leq 9$ and $s > 0$.
\end{lemma}

\begin{proof}   We begin with \eqref{psix-star}.  Let us write 
\begin{equation}\label{fdef}
f(s) := \sum_{j=0}^9 s^{(j+1)/2} \| \partial_x^j \Psi^*_x(s) \|_{S_{k(s)}(I \times \R^2)}
\end{equation}
for $s > 0$.  To show \eqref{psix-star}, it suffices to show that $f(s) \lesssim c(s)$ for all $s > 0$.

From \eqref{Psi-x-eq}, the Leibniz rule \eqref{leibnitz}, and Minkowski's inequality \eqref{mink} we have
\begin{align*}
f(s) &\lesssim \int_s^\infty \sum_{j=0}^{9} s^{(j+1)/2} \| \partial^{j+1}_x \psi^*_s(s') \|_{S_{k(s)}(I \times \R^2)} \\
&\quad + \sum_{j_1,j_2 \geq 0: j_1+j_2 \leq 9} s^{(j_1+j_2+1)/2} \| \partial_x^{j_1} \Psi^*_x(s') \partial_x^{j_2} \psi^*_s(s') \|_{S_{k(s)}(I \times \R^2)} \ ds'.
\end{align*}
Using \eqref{psisjk-abstract}, \eqref{physical-abstract}, \eqref{prod1-abstract} we thus have
\begin{align*}
f(s) &\lesssim \int_s^\infty \sum_{j=0}^{9} s^{(j+1)/2} c(s') (s')^{-(j+3)/2} (s'/s)^{\delta_1/2} \\
&\quad +
\sum_{j_1,j_2 \geq 0: j_1+j_2 \leq 9} s^{(j_1+j_2+1)/2} (s')^{-(j_1+1)/2} f(s') c(s') (s')^{-(j_2+2)/2} (s'/s)^{-\delta_1/2}\ ds'.
\end{align*}
From \eqref{sm}, \eqref{envelo-2} we conclude
$$ f(s) \lesssim c(s) + \int_s^\infty f(s') c(s') (s'/s)^{-\frac{3}{2} + \frac{\delta_1}{2} + \delta_0} \frac{ds'}{s'}.$$
From \eqref{psi-decay-1} we also see that $f(s) \to 0$ as $s \to \infty$.  By Lemma \ref{gron-lem} and \eqref{sm} we conclude that $f(s) \lesssim c(s)$ as required.

Now we establish \eqref{psix-diff-eq}.  We write
$$
\delta f(s) := \sum_{j=0}^9 s^{(j+1)/2} \| \partial_x^j \delta \Psi_x(s) \|_{S_{k(s)}(I \times \R^2)}
$$
for $s > 0$.  It suffices to show that $f(s) \lesssim \delta c(s)$ for all $s \gtrsim 0$.

From \eqref{Psi-x-eq} and \eqref{disc-leib-eq} we have
\begin{equation}\label{delta-1}
\delta \Psi_x(s) = \int_s^\infty \bigO( \partial_x \delta \psi_s(s') ) + \bigO( (\delta \psi_s(s')) \Psi^*_x(s') ) + \bigO( \psi^*_s(s') \delta \Psi_x(s') )\ ds'.
\end{equation}
Using the Leibniz rule \eqref{leibnitz} and Minkowski's inequality \eqref{mink}, \eqref{psisjk-abstract}, \eqref{physical-abstract}, \eqref{psix-star}, and \eqref{prod1-abstract} we thus have
\begin{align*}
\delta f(s) &\lesssim \int_s^\infty \sum_{j=0}^{9} s^{(j+1)/2} \delta c(s') (s')^{-(j+3)/2} (s'/s)^{\delta_1/2} \\
&\quad +
\sum_{j_1,j_2 \geq 0: j_1+j_2 \leq 9} s^{(j_1+j_2+1)/2} (s')^{-(j_1+1)/2} c(s') \delta c(s') (s')^{-(j_2+2)/2} (s'/s)^{-\delta_1/2}\ ds'\\
&\quad +
\sum_{j_1,j_2 \geq 0: j_1+j_2 \leq 9} s^{(j_1+j_2+1)/2} (s')^{-(j_1+1)/2} \delta f(s') c(s') (s')^{-(j_2+2)/2} (s'/s)^{-\delta_1/2}\ ds'.
\end{align*}
Using \eqref{sm}, \eqref{envelo-2} this simplifies to
$$ \delta f(s) \lesssim \delta c(s) + \int_s^\infty \delta f(s') c(s') (s'/s)^{-\frac{3}{2} + \frac{\delta_1}{2} + \delta_0} \frac{ds'}{s'}.$$
By \eqref{psi-decay-1} we also have $\delta f(s) \to 0$ as $s \to \infty$.  The claim then follows from Lemma \ref{gron-lem} as before.
\end{proof}

\begin{remark} Note that every term that occurs when bounding $\delta \Psi_x$ contains exactly one factor containing a $\delta$; this is ultimately because the Leibniz rule \eqref{disc-leib-eq} is ``linear in $\delta$''.  This is why the final estimate is linear in $\delta c$.  The same phenomenon will hold throughout the rest of this section.
\end{remark}

We will need to compensate for the above loss of derivative (and for derivative losses in later parts of the argument). Here our main tools will be the parabolic equations \eqref{psit-eq}, \eqref{psis-eq} and \eqref{parreg-abstract}.  We first recover one derivative for $\psi_x$:

\begin{lemma}\label{psix-recover}  We have
\begin{equation}\label{deltash}
 \| \partial_x^j \psi^*_x(s) \|_{S_{k(s)}(I \times \R^2)} \lesssim c(s) s^{-(j+1)/2}
\end{equation}
and
\begin{equation}\label{deltash-diff}
\| \partial_x^j \delta \psi_x(s) \|_{S_{k(s)}(I \times \R^2)} \lesssim \delta c(s) s^{-(j+1)/2}
\end{equation}
for all $0 \leq j \leq 10$ and $s > 0$.
\end{lemma}

\begin{proof}  From Lemma \ref{jreg} we only need to verify this when $j=10$.
We begin with \eqref{deltash}.  From \eqref{psit-eq} we have the schematic equation
\begin{equation}\label{psit-star}
 \partial_s \psi^*_x = \Delta \psi^*_x + \bigO( \Psi^*_x \partial^*_x \Psi^*_x ) + \bigO( (\Psi^*_x)^3 ).
\end{equation}
Applying the Leibniz rule \eqref{leibnitz} and \eqref{parreg-abstract} we conclude that
\begin{align*}
\| \partial_x^{10} \psi^*_x(s) \|_{S_k(I \times \R^2)} &\lesssim s^{-1} \| \partial_x^8 \psi^*_x(s/2) \|_{S_k(I \times \R^2)} \\
&\quad +
\sup_{s/2 \leq s' \leq s} \sum_{j_1,j_2 \geq 0: j_1+j_2 = 9} \| \partial_x^{j_1} \Psi^*_x \partial_x^{j_2} \Psi^*_x \|_{S_k(I \times \R^2)} \\
&\quad + \sum_{j_1,j_2,j_3 \geq 0: j_1+j_2+j_3=8}
\| \partial_x^{j_1} \Psi^*_x \partial_x^{j_2} \Psi^*_x \partial_x^{j_3} \Psi^*_x \|_{S_k(I \times \R^2)}.
\end{align*}
Applying \eqref{prod1-abstract} and Lemma \ref{jreg} we conclude
$$ \| \partial_x^{10} \psi^*_x(s) \|_{S_k(I \times \R^2)} \lesssim c(s) s^{-(10+1)/2} +
\sup_{s/2 \leq s' \leq s} c(s')^2 s^{-(10+1)/2} + c(s')^3 s^{-(10+1)/2}$$
and the claim \eqref{deltash} follows from \eqref{sm}, \eqref{envelo-2}.

Now we prove \eqref{deltash-diff}.  From \eqref{psit-eq}, \eqref{disc-leib-eq} we have the schematic equation
$$ \partial_s \delta \psi_x = \Delta \delta \psi_x + \bigO( (\delta \Psi_x) \partial_x \Psi^*_x ) + \bigO( \Psi^*_x \partial_x \delta \Psi_x ) + \bigO( (\Psi^*_x)^2 \delta \Psi_x).$$
Applying the Leibniz rule \eqref{leibnitz} and \eqref{parreg-abstract} we conclude that
\begin{align*}
\| \partial_x^{10} \delta \psi_x(s) \|_{S_k(I \times \R^2)} &\lesssim s^{-1} \| \partial_x^8 \delta \psi_x(s/2) \|_{S_k(I \times \R^2)} \\
&\quad +
\sup_{s/2 \leq s' \leq s} \sum_{j_1,j_2 \geq 0: j_1+j_2 = 9} \| \partial_x^{j_1} \Psi^*_x \partial_x^{j_2} \delta \Psi_x \|_{S_k(I \times \R^2)} \\
&\quad + \sum_{j_1,j_2,j_3 \geq 0: j_1+j_2+j_3=8}
\| \partial_x^{j_1} \Psi^*_x \partial_x^{j_2} \Psi^*_x \partial_x^{j_3} \delta \Psi_x \|_{S_k(I \times \R^2)}.
\end{align*}
Applying \eqref{prod1-abstract}, \eqref{deltash}, Lemma \ref{jreg}, \eqref{sm}, and \eqref{envelo-2} we obtain the claim.
\end{proof}

Then we recover a derivative for $A_x$ (gaining a useful factor of $c(s)$ in the process):

\begin{lemma}\label{ax-recover}  We have 
\begin{equation}\label{ax-eq} \| \partial_x^j A^*_x(s) \|_{S_{k(s)}(I \times \R^2)} \lesssim c(s)^2 s^{-(j+1)/2}
\end{equation}
and
\begin{equation}\label{axdiff-eq} \| \partial_x^j \delta A_x(s) \|_{S_{k(s)}(I \times \R^2)} \lesssim c(s) \delta c(s) s^{-(j+1)/2}
\end{equation}
for all $0 \leq j \leq 10$ and $s > 0$.
\end{lemma}

\begin{proof} We first prove \eqref{ax-eq}. From \eqref{a-eq}, the Leibniz rule \eqref{leibnitz}, and Minkowski's inequality \eqref{mink} we have
$$ \| \partial_x^j A^*_x \|_{S_{k(s)}(I \times \R^2)} \lesssim \int_s^\infty \sum_{j_1,j_2 \geq 0: j_1+j_2=j} \| \partial_x^{j_1} \psi^*_s(s') \partial_x^{j_2} \psi^*_x(s') \|_{S_{k(s)}(I \times \R^2)}\ ds'.$$
Applying \eqref{prod1-abstract}, Lemma \ref{psix-recover}, and \eqref{psisjk-abstract} we thus have
$$ \| \partial_x^j A^*_x \|_{S_{k(s)}(I \times \R^2)} \lesssim \int_s^\infty c(s')^2 (s'/s)^{\delta_1/2} (s')^{-(j+3)/2}\ ds'.$$
Applying \eqref{sm}, \eqref{envelo-2}, we obtain the claim \eqref{ax-eq}.

Now we prove \eqref{axdiff-eq}.  
From \eqref{a-eq}, \eqref{disc-leib-eq}, the Leibniz rule \eqref{leibnitz}, and Minkowski's inequality \eqref{mink} we have
\begin{align*}
\| \partial_x^j \delta A_x \|_{S_{k(s)}(I \times \R^2)} &\lesssim \int_s^\infty \sum_{j_1,j_2 \geq 0: j_1+j_2=j} \| \partial_x^{j_1} \psi^*_s(s') \partial_x^{j_2} \delta \psi_x(s') \|_{S_{k(s)}(I \times \R^2)} \\
&\quad + \| \partial_x^{j_1} (\delta \psi^*_s)(s') \partial_x^{j_2} \psi^*_x(s') \|_{S_{k(s)}(I \times \R^2)}\ ds'.
\end{align*}
Arguing as in the proof of \eqref{ax-eq} we obtain the claim.
\end{proof}

Now we recover a derivative for $\psi_s$:

\begin{lemma}\label{psis-recover}  We have
\begin{equation}\label{psistar-eq} \| \partial_x^j \psi^*_s(s) \|_{S_{k(s)}(I \times \R^2)} \lesssim c(s) s^{-(j+2)/2}
\end{equation}
and
\begin{equation}\label{psistar-diff-eq} \| \partial_x^j \delta \psi_s(s) \|_{S_{k(s)}(I \times \R^2)} \lesssim \delta c(s) s^{-(j+2)/2}
\end{equation}
for all $0 \leq j \leq 11$ and $s > 0$.
\end{lemma}

\begin{proof} From Proposition \ref{jreg} we only need to verify this when $j=11$.  We begin with \eqref{psistar-eq}.
From \eqref{psis-eq} we have the schematic equation
\begin{equation}\label{psis-eq-star}
 \partial_s \psi^*_s = \Delta \psi^*_s + \bigO( \Psi^*_x \partial_x \psi^*_s ) + \bigO( \psi^*_s \partial_x \Psi^*_x ) + \bigO( (\Psi^*_x)^2 \psi^*_s ).
\end{equation}
Applying the Leibniz rule \eqref{leibnitz} and \eqref{parreg-abstract} we conclude that
\begin{align*}
\| \partial_x^{11} \psi^*_s(s) \|_{S_{k(s)}(I \times \R^2)} &\lesssim s^{-1} \| \partial_x^9 \psi^*_s(s/2) \|_{S_{k(s)}(I \times \R^2)} \\
&\quad + \sup_{s/2 \leq s'\leq s} \sum_{j_1,j_2 \geq 0: j_1+j_2=10} \| \partial_x^{j_1} \Psi^*_x \partial_x^{j_2} \psi_s \|_{S_{k(s)}(I \times \R^2)} \\
&\quad\quad  + \sum_{j_1,j_2,j_3 \geq 0: j_1+j_2=9} \| \partial_x^{j_1} \Psi^*_x \partial_x^{j_2} \Psi^*_x  \partial_x^{j_3} \psi^*_s \|_{S_{k(s)}(I \times \R^2)}.
\end{align*}
Applying \eqref{prod1-abstract}, Lemma \ref{psix-recover}, Lemma \ref{ax-recover}, and \eqref{psisjk-abstract} we thus have
$$ \| \partial_x^{10} \psi^*_x(s) \|_{S_{k(s)}(I \times \R^2)} \lesssim c(s) s^{-(11+2)/2} +
\sup_{s/2 \leq s' \leq s} c(s')^2 s^{-(11+2)/2} + c(s')^3 s^{-(11+2)/2}.$$
Applying \eqref{sm}, \eqref{envelo-2} we obtain the claim \eqref{psistar-eq}.

Now we prove \eqref{psistar-diff-eq}.  From \eqref{psis-eq} and \eqref{disc-leib-eq} we have the schematic equation
\begin{equation}\label{psisdelta}
\begin{split}
\partial_s \delta \psi_s &= \Delta \delta \psi_s + \bigO( (\delta \Psi_x) \partial_x \psi^*_s ) +  \bigO( \Psi^*_x \partial_x \delta \psi_x ) \\
&\quad + \bigO( (\delta \psi_s) \partial_x \Psi^*_x ) + 
\bigO( \psi_s^* \partial_x \delta \Psi^*_x ) \\
&\quad +\bigO( \Psi^*_x (\delta \Psi_x) \psi^*_s ) + \bigO( (\Psi^*_x)^2 \delta \psi^*_s ).
\end{split}
\end{equation}
As before, observe that every term on the right contains one $\delta$ expression.  The claim now follows by repeating the arguments used to prove \eqref{psistar-eq}.
\end{proof}

Note that the conclusion of this lemma is just like \eqref{psisjk-abstract}, but with one additional derivative of regularity on $\psi_s$.  We can iterate these arguments indefinitely to conclude the claim of Theorem \ref{parab-thm} with $\partial_x$.  The analogous claims involving $\partial_s$ can then be proven by further applications of estimates such as \eqref{psis-eq-star}; we omit the (rather tedious) details.

\section{Initial data estimates}\label{initial-sec}

In \cite{tao:heatwave2}, various parabolic regularity estimates on the differentiated fields $\Psi_{s,t,x} \in \WMC(I,E)$ in the caloric gauge were established, which were either qualitative (with constants depending on $\Psi_{s,t,x}$) or quantitative (with constants depending only on the energy $\E(\Psi_{s,t,x})$.  In this section we develop the key estimates from that theory that we need.  We begin with some qualitative estimates.

\begin{proposition}[Qualitative estimates]\label{qual-prop} If $I$ is a compact interval and $\Psi_{s,t,x} \in \WMC(I,E)$ for some $E>0$, then
\begin{equation}\label{l1-eq}
 \| \partial_x^j \psi_{t,x}(s,t) \|_{L^1_x(\R^2)} \lesssim_{\phi,j} \langle s \rangle^{-j/2}
\end{equation}
and
\begin{equation}\label{qual-eq}
 \| \partial_x^j \partial_{t,x} \psi_s(s,t) \|_{L^1_x(\R^2)} \lesssim_{\Psi_{s,t,x},j} \langle s \rangle^{-(j+2)/2}
\end{equation}
for all $t \in I$, $j \geq 0$, and $s \geq 0$.
\end{proposition}

\begin{proof} In \cite[Theorem 3.12]{tao:heatwave2} and \cite[Theorem 3.16]{tao:heatwave2}, it was shown that $\Psi_{s,t,x}$ was uniformly Schwartz in $t,x$ for $s$ in any bounded interval (here we use the compactness of $I$), and so \eqref{qual-eq} already holds true when $s$ is bounded.

It remains to handle the contribution of when $s$ is large.  Here we use the estimates
$$ |\partial_t^j \partial_x^m \phi(s,t,x)| \lesssim_{m,j,\phi} \langle s \rangle^{-(m+1)/2}$$ 
for all $j, m \geq 0$ just before \cite[Equation (87)]{tao:heatwave2}, which imply that
\begin{equation}\label{psreg-eq} |\partial_t^j \partial_x^m \Psi_{t,x}(s,t,x)| \lesssim_{m,j,\phi} \langle s \rangle^{-(m+2)/2}.
\end{equation} 
Next, since $\psi_{t,x}$ is Schwartz for bounded $s$, we have
\begin{equation}\label{psits}
 \| \psi_{t,x}(s,t) \|_{L^1_x(\R^2)} \lesssim_\phi 1
\end{equation}
for all $s = O(1)$.  Applying \cite[Lemma 4.8]{tao:heatwave} (noting that \eqref{psit-eq} that $\psi_{t,x}$ obeys a covariant heat equation) we see that \eqref{psits} in fact holds for all $s$. On the other hand, from \eqref{psit-eq} we have the heat equation
$$ \partial_s \psi_{t,x} = \Delta \psi_{t,x} + \bigO( \partial_x ( \Psi_x \psi_{t,x} ) ) + \bigO( (\partial_x\Psi_x) \psi_{t,x} ) + \bigO( \Psi_x^2 \psi_{t,x} ).$$
Repeatedly applying \eqref{psreg-eq}, Duhamel's formula and standard parabolic regularity estimates we then obtain \eqref{l1-eq} for all $j \geq 0$.  Further application of \eqref{psit-eq} then gives
$$
 \| \partial_x^j \partial_s \psi_{t,x}(s,t) \|_{L^1_x(\R^2)} \lesssim_{\phi,j} \langle s \rangle^{-(j+2)/2};
$$
using \eqref{inter2} we conclude \eqref{qual-eq}.
\end{proof}

Now we turn to quantitative estimates.

\begin{theorem}[Initial data envelope bounds]\label{envelop} Let $I$ be an interval, $E>0$, and let $t_0 \in I$.  Let $\Psi_{s,t,x}, \Psi'_{s,t,x} \in \WMC(I,E)$.  Then there exists frequency envelopes $c_0, \delta c_0: \R^+ \to \R^+$ of energy $O_{E}(1)$ and
$O_E(d_\Energy( \Psi_{s,t,x}, \Psi'_{s,t,x} ))^2$ respectively,
such that
\begin{equation}\label{l2s-fixed-special}
\begin{split}
\| \partial_x^{j+1} \Psi^*_{t,x}(s,t) \|_{L^2_x(\R^2)} 
+ \| \partial_x^{j} \Psi^*_{t,x}(s,t) \|_{L^\infty_x(\R^2)} \quad &\\
+ \| \partial_x^j \partial_{t,x} \Psi^*_x(s,t) \|_{L^2_x(\R^2)} +
\| \partial_x^j \psi^*_s(s,t) \|_{L^2_x(\R^2)} \quad &\\
+ \| \partial_x^{j-1} \partial_{t,x} \psi^*_s(s,t) \|_{L^2_x(\R^2)}
&\lesssim_{j} c_0(s) s^{-(j+1)/2} 
\end{split}
\end{equation}
and
\begin{equation}\label{l2s-fixed-special-diff}
\begin{split}
\| \partial_x^{j+1} \delta\Psi_{t,x}(s,t) \|_{L^2_x(\R^2)} 
+ \| \partial_x^{j} \delta\Psi_{t,x}(s,t) \|_{L^\infty_x(\R^2)} \quad &\\
+ \| \partial_x^j \partial_{t,x} \delta\Psi_x(s,t) \|_{L^2_x(\R^2)} +
\| \partial_x^j \delta\psi_s(s,t) \|_{L^2_x(\R^2)} \quad & \\
+ \| \partial_x^{j-1} \partial_{t,x} \delta \psi_s(s,t) \|_{L^2_x(\R^2)}
&\lesssim_{j} \delta c_0(s) s^{-(j+1)/2} 
\end{split}
\end{equation}
for all $j \geq 0$, with the convention that we drop all terms involving $\partial_x^{-1}$.
\end{theorem}

The rest of the section is devoted to proving Theorem \ref{envelop}.  We allow all implied constants to depend on $E$.  We also work exclusively at time $t$, and omit explicit dependence on this time parameter.

We introduce the seminorm
\begin{equation}\label{seminorm-eq} \| f \|_{S_k(t)} := \sup_{k'} \chi_{k=k'}^{-\delta_1} \|P_{k'} \partial_{t,x} f(t) \|_{L^2_x(\R^2)}.
\end{equation}
on $\Sch(I \times \R^2)$.  Note that this seminorm is continuous in the $\Sch(I \times \R^2)$ topology, and obeys \eqref{physical-abstract}.

We will define $c_0$ by the formula
$$ c_0(s) := \sum_{j=0}^{10} \sup_{s' > 0} \min( (s'/s)^{\delta_0}, (s/s')^{\delta_0} )
(s')^{(j+2)/2} \| \partial_x^j \psi^*_s(s',t) \|_{S_{k(s)}(t)}.$$

\begin{lemma}\label{c0lem} $c_0$ is a frequency envelope of energy $O(1)$.  In particular we have the crude estimate
\begin{equation}\label{c0-crude}
c_0(s)= O(1).
\end{equation}
\end{lemma}

\begin{proof} The property \eqref{sm} is clear from construction, so it suffices to show that
$$ \int_0^\infty c_0(s)^2 \frac{ds}{s} \lesssim 1.$$
If we set $f_{k}(s) := \sum_{j=0}^{10} s^{(j+2)/2} \| \partial_x^j \psi^*_s(s,t) \|_{S_{k}(t)}$, we see from the multiplicative form of Young's inequality (and \eqref{physical-abstract}) that
$$ \int_0^\infty c_0(s)^2 \frac{ds}{s} \lesssim \sum_k \sup_{s \sim 2^{-2k}} f_k(s)^2$$
and so by Poincar\'e's inequality, it suffices to show that
\begin{equation}\label{fks-eq} \int_0^\infty f_{k(s)}(s)^2 \frac{ds}{s} \lesssim 1.
\end{equation}
and
\begin{equation}\label{fks-2-eq} \int_0^\infty (s f'_{k(s)}(s))^2 \frac{ds}{s} \lesssim 1.
\end{equation}

From \eqref{inter4} we have
$$\psi^*_s = \bigO(\partial_x \psi^*_x ) + \bigO( (\Psi^*_x)^2 )$$
and hence
$$ \partial_{t,x} \partial_x^j \psi^*_s = \bigO( \partial_x^{j+1} \partial_{t,x} \psi^*_x ) + \sum_{j_1+j_2=j} \bigO( (\partial_x^{j_1} \Psi^*_x) (\partial_x^{j_2} \partial_{t,x} \Psi^*_x ) )$$
for $0 \leq j \leq 10$.  From \eqref{inter3} we thus have
$$ \partial_{t,x} \partial_x^j \psi^*_s = \bigO( \partial_x^{j+2} \psi^*_{t,x} ) + \sum_{j_1+j_2=j} \bigO( (\partial_x^{j_1} \Psi^*_x) (\partial_x^{j_2} \partial_{t,x} \Psi^*_x ) )
+ \bigO( (\partial_x^{j_1} \Psi^*_x) (\partial_x^{j_2} \partial_x \Psi^*_{t,x} ) ).
$$
From \eqref{seminorm-eq}, Littlewood-Paley theory and Bernstein's inequality we thus have
\begin{align*}
\| \partial_x^j \psi^*_s(s) \|_{S_k(I \times \R^2)}
&\lesssim 2^{k} \| \partial_x^{j+1} \psi^*_{t,x}(s)\|_{L^2_x(\R^2)}
+ 2^{-k} \| \partial_x^{j+3} \psi^*_{t,x}(s)\|_{L^2_x(\R^2)} \\
&\quad + \sum_{j_1+j_2=j,j+2} 2^{(j+1-j_1-j-1)k} 
\| \partial_x^{j_1} \Psi^*_x(s) \|_{L^2_x(\R^2)} \\
&\quad\quad \times  
( \| \partial_x^{j_2} \partial_x \Psi^*_{t,x}(s) \|_{L^\infty_x(\R^2)} + \| \partial_x^{j_2} \partial_{t,x} \Psi^*_x(s) \|_{L^\infty_x(\R^2)} ).
\end{align*}
From \cite[Lemma 7.2]{tao:heatwave2} we have
$$ \| \partial_x^j \Psi^*_{t,x} \|_{L^\infty_x(\R^2)} \lesssim_j s^{-(j+1)/2}$$
for all $j \geq 0$; applying \eqref{inter3} we conclude that
$$ \| \partial_x^{j+1} \Psi^*_{t,x} \|_{L^\infty_x(\R^2)} 
+ \| \partial_x^{j} \partial_{t,x} \Psi^*_{x} \|_{L^\infty_x(\R^2)} 
\lesssim_j s^{-(j+1)/2}.$$
Inserting this into the previous estimates, we conclude that
$$ f_{k(s)}(s) \lesssim \sum_{j=0}^{13} s^{j/2} \| \partial_x^j \Psi^*_{t,x}(s)\|_{L^2_x(\R^2)};$$
a similar (but more involved) computation (using \eqref{inter1}, \eqref{inter2}, \eqref{inter4} to handle the additional $s$ derivative) gives
$$ s |f'_{k(s)}(s)| \lesssim \sum_{j=0}^{15} s^{j/2} \| \partial_x^j \Psi^*_{t,x}(s)\|_{L^2_x(\R^2)}.$$
The claims \eqref{fks-eq}, \eqref{fks-2-eq} then follow from \cite[Corollary 4.6]{tao:heatwave2} and \cite[Lemma 4.8]{tao:heatwave2}.
\end{proof}

We similarly define
$$ \delta c_0(s) := \sum_{j=0}^{10} \sup_{s' > 0} \min( (s'/s)^{\delta_0}, (s/s')^{\delta_0} )
(s')^{(j+2)/2} \| \partial_x^j \delta \psi_s(s',t) \|_{S_{k(s)}(t)},$$
then $\delta c_0$ is also a frequency envelope of energy $O(1)$.  We will obtain a better estimate on the energy of $\delta c_0$ later.

We now plan to apply Theorem \ref{parab-thm}.  From Littlewood-Paley theory and Bernstein's inequality we see that the $S_{k(s)}$ norm obeys \eqref{prod1-abstract}.  Now we establish \eqref{parreg-abstract}. We use Duhamel's formula \eqref{duh} to write
$$\partial_x^2 \phi(s) = \partial_x^2 e^{s\Delta/2} \phi(s/2) + \int_{s/2}^s \partial_x^2 e^{(s-s')\Delta} (\partial_s - \Delta) \phi(s')\ ds'.$$
Direct calculation shows that the convolution kernel of $\partial_x^2 e^{s\Delta/2}$ has an $L^1$ norm of $O( s^{-1} )$, and so the contribution of the first term is acceptable by Minkowski's inequality and the translation invariance of $S_k(t)$.  So it suffices to show that
$$ \| \int_{s/2}^s \partial_x^2 e^{(s-s')\Delta} (\partial_s - \Delta) \phi(s')\ ds' \|_{S_k(I \times \R^2)} \lesssim \sup_{s/2 \leq s' \leq s} \| (\partial_s - \Delta) \phi(s')\ ds' \|_{S_k(I \times \R^2)}.$$
A direct application of Minkowski's inequality will give a logarithmic divergence, but we can take advantage of the Besov-space structure of $S_k(t)$ to evade this.  Indeed, from \eqref{seminorm-eq}
(and the fact that Littlewood-Paley operators commute with partial derivatives) it suffices to show that
$$ \| \int_{s/2}^s \partial_x^2 e^{(s-s')\Delta} (\partial_s - \Delta) P_{k'} \nabla_{t,x} \phi(s')\ ds' \|_{L^2_x(\R^2)} \lesssim \sup_{s/2 \leq s' \leq s} \| (\partial_s - \Delta) P_{k'} \nabla_{t,x} \phi(s')\ ds' \|_{L^2_x(\R^2)}$$
for each $k'$. We may freely insert a projection $P_{k'-5 < \cdot < k'+5}$ in front of the integrand.  A calculation using the Fourier transform shows that the convolution kernel of $P_{k'-5 < \cdot < k'+5} \partial_x^2 e^{(s-s')\Delta}$ has an $L^1$ norm of $O( 2^{2k'} \langle 2^{-2k'} (s-s') \rangle^{-100} )$, and so the claim \eqref{parreg-abstract} follows from translation invariance and Minkowski's inequality.   Finally, from Proposition \ref{qual-prop} and Bernstein's inequality we obtain the vanishing property \eqref{psi-decay-1}.

We now have all the hypotheses of Theorem \ref{parab-thm} verified; invoking this theorem, we obtain the estimates \eqref{parab-1}-\eqref{parab-6}.  From Littlewood-Paley theory and Bernstein's inequality we have
$$ \|f(t)\|_{L^\infty_x(\R^2)} + \| \partial_{t,x} f(t) \|_{L^2_x(\R^2)} \lesssim \|f\|_{S_k(t)}$$
for any $k$, and thus
\begin{align*}
\| \partial_x^j \Psi^*_x(s) \|_{L^\infty_x(\R^2)} +
\| \partial_x^j \partial_{t,x} \Psi^*_x(s) \|_{L^2_x(\R^2)} &\lesssim_{j} c_0(s) s^{-(j+1)/2}\\
\| \partial_x^j A^*_x(s) \|_{L^\infty_x(\R^2)} +
\| \partial_x^j \partial_{t,x} A^*_x(s) \|_{L^2_x(\R^2)} &\lesssim_{j} c_0(s)^2 s^{-(j+1)/2}\\
\| \partial_x^j \psi^*_s(s) \|_{L^\infty_x(\R^2)} +
\| \partial_x^j \partial_{t,x} \psi^*_s(s) \|_{L^2_x(\R^2)} &\lesssim_{j} c_0(s) s^{-(j+2)/2}\\
\| \partial_x^j \delta \Psi_x(s) \|_{L^\infty_x(\R^2)} +
\| \partial_x^j \partial_{t,x} \delta \Psi_x(s) \|_{L^2_x(\R^2)} &\lesssim_{j} \delta c_0(s) s^{-(j+1)/2}\\
\| \partial_x^j \delta A_x(s) \|_{L^\infty_x(\R^2)} +
\| \partial_x^j \partial_{t,x} \delta A_x(s) \|_{L^2_x(\R^2)} &\lesssim_{j} c_0(s) \delta c_0(s) s^{-(j+1)/2}\\
\| \partial_x^j \delta \psi_s(s) \|_{L^\infty_x(\R^2)} +
\| \partial_x^j \partial_{t,x} \delta \psi_s(s) \|_{L^2_x(\R^2)} &\lesssim_{j} \delta c_0(s) s^{-(j+2)/2}
\end{align*}
for all $j \geq 0$ and $s>0$.
The estimates \eqref{l2s-fixed-special}, \eqref{l2s-fixed-special-diff} now follow from these estimates (and \eqref{inter1}-\eqref{inter4}, the Leibniz rule, H\"older's inequality, and \eqref{c0-crude} as necessary).

It remains to establish the energy bound on $\delta c_0$.

\begin{lemma}[Energy bound]\label{engen}  We have
$$ \int_0^\infty (\delta c_0(s))^2 \frac{ds}{s} \lesssim d_\Energy(\Psi_{s,t,x}, \Psi'_{s,t,x})^2.$$
\end{lemma}

\begin{proof} Write $\mu_0 := d_\Energy(\Psi_{s,t,x}, \Psi'_{s,t,x})$ and $\mu :=
(\int_0^\infty (\delta c_0(s))^2 \frac{ds}{s})^{1/2}$.  We may assume that $\mu_0 \lesssim \mu$ as
the claim is trivial otherwise.   It will suffice to establish a bound of the form
\begin{equation}\label{mumu-eq}
\mu^2 \lesssim_\eps \mu_0^{2-\eps} \mu^\eps 
\end{equation}
for some $0 < \eps < 2$.

By repeating the arguments used to prove Lemma \ref{c0lem}, we have
$$\mu^2 \lesssim \int_0^\infty \delta f_{k(s)}(s)^2 \frac{ds}{s} + \int_0^\infty \delta (s f'_{k(s)}(s))^2 \frac{ds}{s}$$
where 
$$ \delta f_k(s) := \sum_{j=0}^{10} s^{(j+2)/2} \| \partial_x^j \delta \psi_s(s,t) \|_{S_{k}(t)}.$$
Continuing the arguments of Lemma \ref{c0lem} (adapted to differences, of course), we have
$$ \delta f_{k(s)}(s) + s |\delta f'_{k(s)}(s)| \lesssim \sum_{j=0}^{15} s^{j/2} \| \partial_x^j \delta \Psi_{t,x}(s)\|_{L^2_x(\R^2)}$$
and thus
$$ \mu^2 \lesssim \sum_{j=0}^{15} \int_0^\infty s^{j-1} \| \partial_x^j \delta \Psi_{t,x}(s)\|_{L^2_x(\R^2)}^2\ ds.$$

On the other hand, by \eqref{psipsi} we have
\begin{equation}\label{mu0d-eq} \int_0^\infty \| \delta \psi_s(s) \|_{L^2_x(\R^2)}^2\ ds \lesssim \mu_0^2
\end{equation}
and
\begin{equation}\label{deltat-eq} \| \delta \psi_t(0) \|_{L^2_x(\R^2)} \lesssim \mu_0.
\end{equation}
Meanwhile, from \eqref{l2s-fixed-special-diff} we have
$$ \| \partial_x^j \delta \psi_s(s) \|_{L^2_x(\R^2)} \lesssim_j s^{-j/2} \delta c_0(s)$$
and
$$ \| \partial_x^j \delta \psi_s(s) \|_{L^\infty_x(\R^2)} \lesssim_j s^{-(j+1)/2} \delta c_0(s)$$
for all $j \geq 0$.  Applying repeated $s$ derivatives using the differenced versions of \eqref{psis-eq}, \eqref{psit-eq}, \eqref{inter1} we conclude
$$ \|  \partial_s^k \partial_x^j \delta \psi_s(s) \|_{L^2_x(\R^2)} \lesssim_{j,k} s^{-j/2-k} \delta c_0(s)$$
for all $j,k > 0$, and thus
\begin{equation}\label{psp}
 \int_0^\infty s^{j+2k} \|\partial_s^k \partial_x^j \delta \psi_s(s) \|_{L^2_x(\R^2)}^2\ ds \lesssim_{j,k} \mu^2.
 \end{equation}
Interpolating with \eqref{mu0d-eq}, we conclude that
$$ \int_0^\infty s^{j+2k} \| \partial_s^k \partial_x^j \delta \psi_s(s) \|_{L^2_x(\R^2)}^2\ ds \lesssim_{j,k,\eps} \mu_0^{2-\eps} \mu^\eps$$
for all $j,k \geq 0$ and $\eps > 0$.  If we write
$$ f(s) := \sum_{j=0}^{100} s^{(j+1)/2} \| \partial_x^j \delta \psi_s(s) \|_{L^2_x(\R^2)}$$
we thus have
$$ \int_0^\infty f(s)^2 \frac{ds}{s} +  \int_0^\infty (sf'(s))^2 \frac{ds}{s} \lesssim_\eps \mu_0^{2-\eps} \mu^\eps$$
for all $\eps > 0$.  Arguing as in Lemma \ref{c0lem}, we may thus find a frequency envelope $\delta c'$ of energy $O_\eps( \mu_0^{2-\eps} \mu^\eps)$ for any $\eps > 0$, such that $f(s) \lesssim \delta c'(s)$ for all $s > 0$, thus
$$ \| \partial_x^j \delta \psi_s(s) \|_{L^2_x(\R^2)} \lesssim \delta c'(s) s^{-(j+1)/2}$$
for all $s > 0$ and $0 \leq j \leq 100$.  If we introduce the variant
\begin{equation}\label{seminorm-alt} \| f \|_{\tilde S_k(t)} := \sup_{k'} \chi_{k=k'}^{-\delta_1} \|P_{k'} \partial_{x} f(t) \|_{L^2_x(\R^2)}.
\end{equation}
of \eqref{seminorm-eq}, we conclude from Littlewood-Paley theory that
$$ \| \partial_x^j \delta \psi_s(s) \|_{\tilde S_{k(s)}(t)} \lesssim \delta c'(s) s^{-(j+2)/2}$$
for all $s > 0$ and $0 \leq j \leq 10$.  Also, from \eqref{l2s-fixed-special} we have
$$ \| \partial_x^j \psi^*_s(s) \|_{\tilde S_{k(s)}(t)} \lesssim c_0(s) s^{-(j+2)/2}$$
for the same range of $s, j$. We can then apply Theorem \ref{parab-thm} as before and conclude that
\eqref{parab-1}-\eqref{parab-6} hold for this norm.  In particular, since
$$ \|f(t)\|_{L^\infty_x(\R^2)} + \| \partial_{x} f(t) \|_{L^2_x(\R^2)} \lesssim \|f\|_{\tilde S_k(t)},$$
we have
\begin{align}
\| \partial_x^j \delta \Psi_x(s) \|_{L^\infty_x(\R^2)} +
\| \partial_x^{j+1} \delta \Psi_x(s) \|_{L^2_x(\R^2)} &\lesssim_{j} \delta c'(s) s^{-(j+1)/2}\label{dell-1}\\
\| \partial_x^j \delta A_x(s) \|_{L^\infty_x(\R^2)} +
\| \partial_x^{j+1} \delta A_x(s) \|_{L^2_x(\R^2)} &\lesssim_{j} c_0(s) \delta c'(s) s^{-(j+1)/2}\label{dell-2}\\
\| \partial_x^j \delta \psi_s(s) \|_{L^\infty_x(\R^2)} +
\| \partial_x^{j+1} \delta \psi_s(s) \|_{L^2_x(\R^2)} &\lesssim_{j} \delta c'(s) s^{-(j+2)/2}\label{dell-3}
\end{align}
for all $j \geq 0$ and $s > 0$.

These estimates will be adequate for controlling spatial derivatives.  To control the time derivative, we must of course use \eqref{deltat-eq}.  From \eqref{psit-eq} and \eqref{disc-leib-eq} we know that $\delta \psi_t$ obeys a heat equation,
\begin{equation}\label{dst-eq} \partial_s \delta \psi_t = \Delta \delta \psi_t + \bigO( A_x \partial_x \delta \psi_t ) + \bigO( \partial_x A_x \delta \psi_t ) + \bigO( \Psi_x^2 \delta \psi_t ) + F
\end{equation}
where the forcing term $F$ has the form
$$ F = \bigO( (\delta A_x) \partial_x \psi_t ) + \bigO( \partial_x \delta \Psi_x \psi_t ) + \bigO( \Psi_x \delta \Psi_x \psi_t ).$$
Integrating this heat equation against $\delta \psi_t$ and using H\"older's inequality, we obtain the energy inequality
\begin{align*}
\partial_s \| \delta \psi_t \|_{L^2_x(\R^2)}^2 &\leq - 2 \| \partial_x \delta \psi_t \|_{L^2_x(\R^2)}^2 + O( \| A_x \|_{L^\infty_x(\R^2)} \| \delta \psi_t \|_{L^2_x(\R^2)}
\| \partial_x \delta \psi_t \|_{L^2_x(\R^2)} ) \\
&\quad + O( (\| \partial_x A_x \|_{L^\infty_x(\R^2)} 
+ \| \Psi_x \|_{L^\infty_x(\R^2)}^2) \| \delta \psi_t \|_{L^2_x(\R^2)}^2 ) + O( \|F\|_{L^2_x(\R^2)} \|\delta \psi_t \|_{L^2_x(\R^2)} )
\end{align*}
for $s > 0$.  From \eqref{l2s-fixed-special} we have
\begin{align*}
 \| A_x \|_{L^\infty_x(\R^2)} &\lesssim c_0(s)^2 s^{-1/2}\\
\| \partial_x A_x \|_{L^\infty_x(\R^2)} + \| \Psi_x \|_{L^\infty_x(\R^2)}^2 &\lesssim c_0(s)^2 s^{-1}
\end{align*}
while from \eqref{l2s-fixed-special}, \eqref{dell-1}-\eqref{dell-3}, and H\"older's inequality we have
$$ \|F\|_{L^2_x(\R^2)} \lesssim c_0(s) \delta c'(s) s^{-1}.$$
Using the elementary inequality $ab \leq \frac{1}{2} a^2 + \frac{1}{2} b^2$ to split up some mixed terms, we obtain
$$ \partial_s \| \delta \psi_t \|_{L^2_x(\R^2)}^2 \leq - \| \partial_x \delta \psi_t \|_{L^2_x(\R^2)}^2 + O( c_0(s)^2 s^{-1} \| \delta \psi_t \|_{L^2_x(\R^2)}^2 )
+ O( (\delta c'(s))^2 s^{-1} ).$$
Since $c_0$ has energy $O(1)$ and $\delta c'(s)$ has energy $O_\eps( \mu_0^{2-\eps} \mu^\eps)$, we conclude from Gronwall's inequality and \eqref{deltat-eq} (discarding the negative
$\| \partial_x \delta \psi_t \|_{L^2_x(\R^2)}^2$ term) that
$$ \| \delta \psi_t(s) \|_{L^2_x(\R^2)}^2 \lesssim_\eps \mu_0^{2-\eps} \mu^\eps$$
for all $s > 0$ and $\eps > 0$.  Reinstating the discarded term, we then conclude that
$$ \int_0^\infty \| \partial_x \delta \psi_t(s) \|_{L^2_x(\R^2)}^2\ ds \lesssim_\eps \mu_0^{2-\eps} \mu^\eps.$$
From this, \eqref{dst-eq}, Cauchy-Schwarz, and the preceding bounds we conclude that
$$ \int_0^\infty \| (\partial_s - \Delta) \delta \psi_t \|_{L^2_x(\R^2)}\ ds \lesssim_\eps \mu_0^{2-\eps} \mu^\eps$$
and hence by Lemma \ref{ubang} we have
$$ \int_0^\infty \| \delta \psi_t(s) \|_{L^\infty_x(\R^2)}^2\ ds \lesssim_\eps \mu_0^{2-\eps} \mu^\eps.$$

On the other hand, by repeating the arguments used to establish \eqref{psp}, we have
$$
 \int_0^\infty s^{j+2k} \|\partial_s^k \partial_x^{j+1} \delta \psi_t(s) \|_{L^2_x(\R^2)}^2\ ds \lesssim_{j,k} \mu^2
$$
and
$$
 \int_0^\infty s^{j+2k} \|\partial_s^k \partial_x^j \delta \psi_t(s) \|_{L^\infty_x(\R^2)}^2\ ds \lesssim_{j,k} \mu^2
$$
for all $j,k \geq 0$.  Interpolating, we conclude that
$$
 \int_0^\infty s^{j+2k} \|\partial_s^k \partial_x^{j+1} \delta \psi_t(s) \|_{L^2_x(\R^2)}^2\ ds \lesssim_{j,k,\eps} \mu_0^{2-\eps} \mu^\eps
$$
and
$$
 \int_0^\infty s^{j+2k} \|\partial_s^k \partial_x^j \delta \psi_t(s) \|_{L^\infty_x(\R^2)}^2\ ds \lesssim_{j,k,\eps} \mu_0^{2-\eps} \mu^\eps.
$$
for all $j,k \geq 0$ and $\eps > 0$.
Arguing as with the construction of $\delta c'$, we may thus find a frequency envelope $\delta c''$ of energy $O_\eps( \mu_0^{2-\eps} \mu^\eps )$ for any $\eps > 0$ such that  
\begin{equation}\label{dell-eq} \| \partial_x^{j+1} \delta \psi_t(s) \|_{L^2_x(\R^2)} + \| \partial_x^j \delta \psi_t(s) \|_{L^\infty_x(\R^2)} \lesssim \delta c''(s) s^{-(j+1)/2}
\end{equation}
for all $0 \leq j \leq 100$.  By increasing $\delta c''$ if necessary we may assume that $\delta c'' \geq \delta c'$.  

From \eqref{a-eq} we have
$$ \delta A_t(s) = \int_s^\infty \bigO( \delta \psi_t(s') \psi_s(s') ) + \bigO( \psi_t(s') \delta \psi_s(s') )\ ds' $$
and so by \eqref{dell-eq}, \eqref{dell-3}, the Leibniz rule, and the Minkowski and H\"older inequalities we have
$$ \| \partial_x^{j+1} \delta A_t(s) \|_{L^2_x(\R^2)} + \| \partial_x^j \delta A_t(s) \|_{L^\infty_x(\R^2)} \lesssim \delta c''(s) s^{-(j+1)/2} $$
for all $0 \leq j \leq 100$.  Combining this with \eqref{dell-eq}, \eqref{dell-1} we have
$$ \| \partial_x^j \delta \Psi_{t,x}(s) \|_{L^\infty_x(\R^2)} +
\| \partial_x^{j+1} \delta \Psi_{t,x}(s) \|_{L^2_x(\R^2)} \lesssim \delta c''(s) s^{-(j+1)/2}$$
and
$$ \| \partial_x^j \delta \psi_s(s) \|_{L^\infty_x(\R^2)} +
\| \partial_x^{j+1} \delta \psi_s(s) \|_{L^2_x(\R^2)} \lesssim \delta c''(s) s^{-(j+2)/2}$$
for $0 \leq j \leq 100$.  If we repeat the proof of Lemma \ref{c0lem}, using the fact that $\delta c''$ has energy $O_\eps( \mu_0^{2-\eps} \mu^\eps )$ we conclude that
$$ \int_0^\infty \delta c_0(s)^2 \frac{ds}{s} \lesssim_\eps \mu_0^{2-\eps} \mu^\eps .$$
Since the left-hand side is $\mu^2$, we obtain \eqref{mumu-eq} as desired.
\end{proof}

The proof of Theorem \ref{envelop} is now complete.

\subsection{Frequency localisation of $\psi_s$}

Heuristically, $\psi_s(s)$ is concentrated at frequencies $\sim s^{-1/2} \sim 2^{k(s)}$ with an energy of $\sim s^{-1} c_0(s)$.  We make this heuristic more precise in the following technical lemma.

\begin{lemma}\label{psioc-lemma} With the notation and assumptions as in Theorem \ref{envelop}, we have
\begin{equation}\label{psioc}
 \sum_{k'} \chi_{k' \leq k(s)}^{-10} \chi_{k'=k(s)}^{-\delta_1}
 \|  P_{k'} \partial_{t,x} \psi^*_s(s,t_0) \|_{L^2_x(\R^2)} \lesssim_E c_0(s) s^{-1}
 \end{equation}
and
\begin{equation}\label{psioc-diff}
 \sum_{k'} \chi_{k' \leq k(s)}^{-10} \chi_{k'=k(s)}^{-\delta_1} \|  P_{k'} \partial_{t,x} \delta \psi_s(s,t_0) \|_{L^2_x(\R^2)} \lesssim_E \delta c_0(s) s^{-1}.
\end{equation}
for all $s > 0$.
\end{lemma}

\begin{proof}  Again we allow implied constants to depend on $E$.
We use \eqref{inter3} to expand 
$$ \partial_{t,x} \psi^*_s = \bigO( \partial_x \partial_{t,x} \Psi^*_x ) + \bigO( \Psi^*_x \partial_{t,x} \Psi^*_x ).$$
Using \eqref{l2s-fixed-special} one already sees that
$$ \| P_{k'} \bigO( \partial_x \partial_{t,x} \Psi^*_x ) \|_{L^2_x(\R^2)} \lesssim 
\chi_{k' \geq k(s)} \chi_{k' \leq k(s)}^{20} c_0(s) s^{-1}$$
(say), and similarly that
$$ \| \partial_x^{20} \bigO( \Psi^*_x \partial_{t,x} \Psi^*_x ) \|_{L^2_x(\R^2)} \lesssim 
c_0(s) s^{-11}$$
and
$$ \| \bigO( \Psi^*_x \partial_{t,x} \Psi^*_x ) \|_{L^1_x(\R^2)} \lesssim 
c_0(s) s^{-1/2}$$
which by Bernstein's inequality gives
$$ \| P_{k'} \bigO( \Psi^*_x \partial_{t,x} \Psi^*_x ) \|_{L^2_x(\R^2)} \lesssim 
\chi_{k' \geq k(s)} \chi_{k' \leq k(s)}^{20} c_0(s) s^{-1}.$$
Putting this all together, we conclude \eqref{psioc}.  The claim \eqref{psioc-diff} is proven by adapting the above argument to differences.
\end{proof}

We also need the following variant of the above lemma, controlling the lower frequency components of $\psi_s$ and $\psi_{t,x}$ more adequately:

\begin{lemma}\label{psioc-lemma-2} With the notation and assumptions as in Theorem \ref{envelop}, we have
\begin{equation}\label{kpsis}
 \|P_k \nabla_{t,x}^i \psi^*_s(s,t_0)\|_{L^2_x(\R^2)} \lesssim_{E,j} c_0(s) s^{-(j+1)/2} \chi_{k \leq k(s)}^j \chi_{k=k(s)}^{0.1}
 \end{equation}
and
\begin{equation}\label{kpsit}
 \|P_k \psi^*_{t,x}(0,t_0)\|_{L^2_x(\R^2)} \lesssim_E c_0(2^{-2k}),
 \end{equation}
and similarly
\begin{equation}\label{kpsis-diff}
 \|P_k \nabla_{t,x}^j \delta \psi_s(s,t_0)\|_{L^2_x(\R^2)} \lesssim_{E,j} \delta c_0(s) s^{-(j+1)/2} \chi_{k \leq k(s)}^j \chi_{k=k(s)}^{0.1}
 \end{equation}
and
\begin{equation}\label{kpsit-diff}
 \|P_k \delta \psi_{t,x}(0,t_0)\|_{L^2_x(\R^2)} \lesssim_E \delta c_0(2^{-2k})
\end{equation}
for all $i=0,1$ and $j \geq 0$.
\end{lemma}

Note that these estimates are consistent with \eqref{psipsi} and Lemma \ref{engen}, but they give much more precise information about the frequency distribution of various components of the distance $d_\Energy(\Psi_{s,t,x}, \Psi'_{s,t,x})$.

\begin{proof}  We may rescale $k=0$, and will omit the $t_0$ subscript for brevity.  We allow all implied constants to depend on $E$. From \cite[Lemma 7.2]{tao:heatwave2} we record the basic estimates
\begin{equation}\label{pij}
\| \psi^*_{t,x}(s) \|_{L^2_x(\R^2)} \lesssim_j s^{-j/2}
\end{equation}
for all $s > 0$.

We first show \eqref{kpsis}.  To simplify the notation we shall only handle the $i=0$ case, though the $i=1$ case follows (and is in fact somewhat easier, due to the increased decay in $s$) using the same arguments.
When $s \geq 1$ this claim follows from Theorem \ref{envelop}, so assume $s<1$; our task is now to show that
$$ \|P_0 \psi^*_s(s)\|_{L^2_x(\R^2)} \lesssim c_0(s) s^{-1/2+0.1}.$$
From \eqref{psis-eq-star} and the fundamental theorem of calculus we have
\begin{align*}
\|P_0 \psi^*_s(s)\|_{L^2_x(\R^2)} &\lesssim \|P_0 \psi^*_s(1)\|_{L^2_x(\R^2)} + \int_s^1 \| P_0 \Delta \psi^*_s(s') \|_{L^2_x(\R^2)} + 
\| P_0(\Psi^*_x \partial_x \psi^*_s)(s') \|_{L^2_x(\R^2)} \\
&\quad + \| P_0(\psi^*_s \partial_x \Psi^*_x)(s') \|_{L^2_x(\R^2)} \\
&\quad + \| P_0((\Psi^*_x)^2 \psi^*_s) (s') \|_{L^2_x(\R^2)}\ ds.
\end{align*}
From \eqref{psioc} the first term is $O( c_0(1) )$, which is acceptable by \eqref{sm}.  By Theorem \ref{envelop}, the term $\| P_0 \Delta \psi^*_s(s') \|_{L^2_x(\R^2)}$ is $O( c_0(s') (s')^{-1/2} )$, which is similarly acceptable by \eqref{sm}.  For the term $\| P_0(\Psi^*_x \partial_x \psi^*_s)(s') \|_{L^2_x(\R^2)}$, we see from Theorem \ref{envelop} and \eqref{pij} that
$$ \| \Psi^*_x \partial_x \psi^*_s(s') \|_{L^1_x(\R^2)} \lesssim c_0(s') (s')^{-1}$$
so by Bernstein's inequality
$$ \| P_0(\Psi^*_x \partial_x \psi^*_s(s')) \|_{L^2_x(\R^2)} \lesssim c_0(s') (s')^{-1}$$
which is acceptable.  Similar arguments dispose of the remaining terms in the integrand.

Now we show \eqref{kpsit}.  From \eqref{psit-eq} we have
$$ \partial_s \psi^*_{t,x} = \Delta \psi^*_{t,x} + \bigO( \Psi^*_{t,x} \partial^*_{t,x} \Psi^*_{t,x} ) + \bigO( (\Psi^*_{t,x})^3 ).$$
so the fundamental theorem of calculus we have
\begin{align*}
\|P_0 \psi^*_{t,x}(s)\|_{L^2_x(\R^2)}  &\lesssim \|P_0 \psi^*_{t,x}(1)\|_{L^2_x(\R^2)}  + \int_s^1 \| P_0 \Delta \psi^*_{t,x}(s') \|_{L^2_x(\R^2)} \\
&\quad + \| P_0 ( \Psi^*_{t,x} \partial^*_x \Psi^*_{t,x} )(s') \|_{L^2_x(\R^2)} + \| P_0( (\Psi^*_{t,x})^3 )(s') \|_{L^2_x(\R^2)}\ ds'.
\end{align*}
From Theorem \ref{envelop}, the first term is $O(c_0(1))$, and the first integrand is $O( c_0(s') )$, which are both acceptable by \eqref{sm}.  Repeating the previous Bernstein and interpolation arguments, one sees that all the remaining integrands are $O( c_0(s') (s')^{-1.1})$, which is also acceptable.

The remaining two estimates are obtained by adapting the above scheme to differences; we omit the details.
\end{proof}

\section{Spacetime function spaces}\label{func0-sec}

The metrics $d_{S^1_\mu,I}$ needed for Theorem \ref{apriori-thm} will be constructed by applying certain spacetime function space norms (mostly from \cite{tao:wavemap2}) to the field $\psi_s$.   Fortunately, we do not need to know the explicit construction of these spaces from \cite{tao:wavemap2} (which are rather complicated), but instead just need a certain abstract list of properties to be satisfied by these norms.  In this section we record the properties we will need for these spaces (cf. \cite[Theorem 3]{tao:wavemap2}).  More precisely, in Section \ref{func-sec} we will establish

\begin{theorem}[Function space norms]\label{func}  For every interval $I$, every integer $k$, and every $\mu > 0$, there exist translation-invariant norms $S_k(I \times \R^2)$, $S_{\mu,k}(I \times \R^2)$, $N_k(I \times \R^2)$ on $\Sch(I \times \R^2)$, with the following properties for all integers $k,k_1,k_2,k_3$ and $\phi, \phi^{(1)}, \phi^{(2)}, \phi^{(3)}, F \in \Sch(I \times \R^2)$:
\begin{itemize}
\item (Continuity and monotonicity) If $I = [t_-,t_+]$, then $\|\phi\|_{S_{\mu,k}([a,b])}$ is a continuous function of $a,b$ for $t_- \leq a < b \leq t_+$, and is decreasing in $a$ and increasing in $b$.
\item ($S_k$ and $S_{\mu,k}$ are comparable) We have
\begin{equation}\label{sksk-star}
 \| \phi \|_{S_{k}(I \times \R^2)} \lesssim \| \phi \|_{S_{\mu,k}(I \times \R^2)} \lesssim \mu^{-1} \| \phi \|_{S_k(I \times \R^2)}.
\end{equation}
\item (Vanishing)  If $\phi \in \Sch(I \times \R^2)$ and $t_0 \in I$, then there exists an interval $J \subset I$ containing $t_0$ such that 
\begin{equation}\label{shrinko}
\| \phi \|_{S_{\mu,k}(I \times \R^2)} \lesssim \sum_{k'} \chi_{k=k'}^{-\delta_1} \| \partial_{t,x} P_{k'} \phi(t_0) \|_{L^2_x(\R^2)}.
\end{equation}
\item (First product estimate) We have
\begin{equation}\label{prod1}
 \| \phi^{(1)} \phi^{(2)} \|_{S_{\max(k_1,k_2)}(I \times \R^2)} \lesssim \| \phi^{(1)} \|_{S_{k_1}(I \times \R^2)} \| \phi^{(2)} \|_{S_{k_2}(I \times \R^2)}.
 \end{equation} 
\item ($N$ contains $L^1_t L^2_x$)  If $F \in \Sch(I \times \R^2)$ has Fourier support in the region $\{ \xi: |\xi| \sim 2^k \}$, then
\begin{equation}\label{fl1l2}
 \| F \|_{N_k(I \times \R^2)} \lesssim \|F\|_{L^1_t L^2_x(I \times \R^2)}.
 \end{equation}
\item (Adjacent $N_k$ or $S_k$ are equivalent) If $\phi \in \Sch(I \times \R^2)$, then
\begin{equation}\label{physical}
 \| \phi\|_{S_{k_1}(I \times \R^2)} \lesssim \chi_{k_1=k_2}^{-\delta_1} \|\phi\|_{S_{k_2}(I \times \R^2)}
\end{equation}
and
\begin{equation}\label{physical-star}
 \| \phi\|_{S_{\mu,k_1}(I \times \R^2)} \lesssim \chi_{k_1=k_2}^{-\delta_1} \|\phi\|_{S_{\mu,k_2}(I \times \R^2)}.
\end{equation}
Similarly, if $F \in \Sch(I \times \R^2)$ and $k_1=k_2+O(1)$, then
\begin{equation}\label{physicaln-eq} \| F\|_{N_{k_1}(I \times \R^2)} \sim \|F\|_{N_{k_2}(I \times \R^2)}.
\end{equation}
\item (Energy estimate) If $\phi \in \Sch(I \times \R^2)$ has Fourier transform supported in the region $\{ |\xi| \sim 2^k\}$, and $t_0 \in I$, then
\begin{equation}\label{energy-est}
\| \phi \|_{S_k(I \times \R^2)} \lesssim 
\| \phi[t_0] \|_{\dot H^1(\R^2) \times L^2(\R^2)} + \| \Box \phi \|_{N_k(I \times \R^2)}.
\end{equation}
\item (Parabolic regularity estimate)  If $\phi: \R^+ \to \Sch(I \times \R^2)$ is smooth and $s > 0$, then
\begin{equation}\label{parreg}
 \| \partial_x^2 \phi(s) \|_{S_k(I \times \R^2)} \lesssim s^{-1} \| \phi(s/2) \|_{S_k(I \times \R^2)} + \sup_{s/2 \leq s' \leq s} \| (\partial_s - \Delta) \phi(s') \|_{S_k(I \times \R^2)}
\end{equation}
and similarly for $S_{\mu,k}(I \times \R^2)$ or $N_k(I \times \R^2)$.
\item (Second product estimate) We have
\begin{equation}\label{second-prod} \| P_k( \phi F ) \|_{N_k(I \times \R^2)} \lesssim \chi_{k \geq k_2}^{\delta_2} 
\chi_{k=\max(k_1,k_2)}^{\delta_2} \| \phi \|_{S_{k_1}(I \times \R^2)} \|F\|_{N_{k_2}(I \times \R^2)}.
\end{equation}
\item (Improved trilinear estimate) We have
\begin{equation}\label{trilinear-improv}
\begin{split}
 \| P_k( \phi^{(1)} \partial_\alpha \phi^{(2)} \partial^\alpha \phi^{(3)}) \|_{N_k(I \times \R^2)} &\lesssim \mu^{1-\eps} \chi_{k=\max(k_1,k_2,k_3)}^{\eps \delta_1} \chi_{k_1 \leq \min(k_2,k_3)}^{\eps \delta_1} \\
&\quad \times  \|\phi^{(1)} \|_{S_{\mu_1, k_1}(I \times \R^2)}
 \|\phi^{(2)}\|_{S_{\mu_2,k_2}(I \times \R^2)}
 \|\phi^{(3)} \|_{S_{\mu_3,k_3}(I \times \R^2)}
\end{split}
\end{equation}
and
\begin{equation}\label{trilinear-improv2}
\begin{split}
 \| P_k( \phi^{(1)} \partial_\alpha \phi^{(2)} \partial^\alpha \phi^{(3)}) \|_{N_k(I \times \R^2)} &\lesssim \mu^{2-2\eps} \chi_{k=\max(k_1,k_2,k_3)}^{\eps \delta_1} \chi_{k_1 \leq \min(k_2,k_3)}^{\eps \delta_1} \\
&\quad  \|\phi^{(1)} \|_{S_{\mu,k_1}(I \times \R^2)}
 \|\phi^{(2)}\|_{S_{\mu,k_2}(I \times \R^2)}
 \|\phi^{(3)} \|_{S_{\mu,k_3}(I \times \R^2)}
\end{split}
\end{equation}
for every $0 \leq \eps \leq 1$ and if $k_1 \geq \min(k_2,k_3)-O(1)$, whenever two of the $\mu_1,\mu_2,\mu_3$ are equal to $\mu$ and the third is equal to $1$ (with the convention that $S_{1,k} = S_k$).
\item (Strichartz estimates)  If $\phi$ has Fourier support in the region $\{ \xi: |\xi| \lesssim 2^k\}$, then we have
\begin{equation}\label{outgo}
 \sup_{t \in I} \| \phi[t] \|_{H^1(\R^2) \times L^2_x(\R^2)} \lesssim \|\phi\|_{S_k(I \times \R^2)}
 \end{equation}
and
\begin{equation}\label{lstrich-4}
\| \partial_{t,x} \phi\|_{L^5_t L^\infty_x(I \times \R^2)} \lesssim 2^{4k/5} \mu \|\phi\|_{S_{\mu,k}(I \times \R^2)}
\end{equation}
and
\begin{equation}\label{lstrich}
\| \partial_{t,x} \phi\|_{L^q_t L^\infty_x(I \times \R^2)} \lesssim 2^{(1-1/q) k} \|\phi\|_{S_{k}(I \times \R^2)}.
\end{equation}
for $5 \leq q \leq \infty$.
\end{itemize}
\end{theorem}

\begin{remark} One could form the metric completion of the Schwartz space $\Sch(I \times \R^2)$ under the above norms to obtain Banach spaces instead of mere normed vector spaces, but we will not need to do so here as our analysis will remain purely in the Schwartz category (note that Theorem \ref{apriori-thm2} deals exclusively with classical solutions).
\end{remark}

\begin{remark} We make the trivial but very useful remark that all the above estimates for scalar-valued functions automatically extend to vector or tensor-valued functions of bounded dimension (with a slight degradation in the implied constants).  We will frequently use this remark in the rest of the paper (especially when dealing with expressions in schematic form) without future comment.
\end{remark}

\begin{remark} The trilinear estimates \eqref{trilinear-improv}, \eqref{trilinear-improv2} are variants of the estimate
\begin{align*}
\| P_k( \phi^{(1)} \partial_\alpha \phi^{(2)} \partial^\alpha \phi^{(3)}) \|_{N_k(I \times \R^2)} &\lesssim \chi_{k=\max(k_1,k_2,k_3)}^{\delta_1} \chi_{k_1 \leq \min(k_2,k_3)}^{\delta_1} \\
&\quad  \|\phi^{(1)}\|_{S_{k_1}(I \times \R^2)} \|\phi^{(2)} \|_{S_{k_2}(I \times \R^2)} \|\phi^{(3)} \|_{S_{k_3}(I \times \R^2)}
\end{align*}
that was established (with some difficulty) in \cite[Section 18]{tao:wavemap2}, and indeed we will use that estimate in the proof of \eqref{trilinear-improv}, \eqref{trilinear-improv2}.  The key improvement in \eqref{trilinear-improv}, \eqref{trilinear-improv2}, which is crucial for the large data theory, is that we can gain an additional factor of the parameter $\mu$.
\end{remark}

\begin{remark} The $L^5_t L^\infty_x$ Strichartz estimate in \eqref{lstrich-4} is a little short of the endpoint estimate $L^4_t L^\infty_x$.  We were not able to establish this estimate for our function spaces (the energy estimate with null frame atom forcing term was problematic); fortunately, as observed in \cite{krieger:3d}, \cite{krieger:2d}, non-endpoint Strichartz estimates such as $L^5_t L^\infty_t$ estimates are still available, and suffice for the purpose of estimating higher order terms.  Fortunately, the Strichartz estimates are only needed for higher order terms anyway, and in fact any non-trivial $L^q_t L^\infty_x$ Strichartz estimate (where by ``non-trivial'' we mean that $q$ is finite) would suffice for our purposes.
\end{remark}

We close this section with some basic corollaries of Theorem \ref{func}.  Firstly, we make the technical observation that the $S_k(I \times \R^2)$, $S_{\mu,k}(I \times \R^2)$, $N_k(I \times \R^2)$ norms are continuous with respect to the $\Sch(I \times \R^2)$ topology (which, in particular, will allow us to use Minkowski's inequality \eqref{mink} in those norms when the integrands are uniformly Schwartz).  For the $N_k(I \times \R^2)$ norm this follows from \eqref{fl1l2}.  For the $S_k(I \times \R^2)$ norm, we use Littlewood-Paley decomposition, the triangle inequality, \eqref{energy-est}, and \eqref{physical}, \eqref{fl1l2} to obtain the useful estimate
\begin{equation}\label{split-energy-est} 
\| \phi \|_{S_k(I \times \R^2)} \lesssim \sum_{k'} \chi_{k=k'}^{-\delta_1} 
[\| P_{k'} \phi[t_0] \|_{\dot H^1(\R^2) \times L^2(\R^2)} + \| \Box P_{k'} \phi \|_{L^1_t L^2_x(I \times \R^2)}].
\end{equation}
It is not difficult to show that the right-hand side goes to zero when $\phi$ goes to zero in $\Sch(I \times \R^2)$, thus establishing continuity of the $S_k(I \times \R^2)$ norm.  The claim for $S_{\mu,k}(I \times \R^2)$ then follows from \eqref{sksk-star}.

Next, we observe the interpolation estimate
\begin{equation}\label{interp-eq}
\| \partial_x^{j'} \phi \|_{S_k(I \times \R^2)} \lesssim_{j',j} \| \phi \|_{S_k(I \times \R^2)}^{1-j'/j} \| \partial_x^{j} \phi \|_{S_k(I \times \R^2)}^{j'/j}
\end{equation}
for all $0 < j' < j$ and $\phi \in \Sch(I \times \R^2)$, and similarly for the $S_{\mu,k}(I \times \R^2)$ and $N_k$ norms.  Indeed, given any frequency parameter $k_0$, we can use Littlewood-Paley theory to write
$$ \partial_x^{j'} \phi = \partial_x^{j'} P_{\leq k_0} \phi + \partial_x^{j'-j} P_{> k_0} \partial_x^j \phi,$$
where $\partial_x^{j'-j}$ is a suitable Fourier multiplier of order $j'-j$.  It is not hard to see that $\partial_x^{j'} P_{\leq k_0}$ and $\partial_x^{j'-j} P_{>k_0}$ are convolution operators whose kernel has $L^1$ norm $O_{j,j'}( 2^{j' k_0} )$ and $O_{j,j'}( 2^{(j'-j) k_0} )$ respectively, so from Minkowski's inequality \eqref{mink}, the triangle inequality, and the translation invariance of the $S_k(I \times \R^2)$ norms we conclude
$$
\| \partial_x^{j'} \phi \|_{S_k(I \times \R^2)} \lesssim_{j',j} 2^{j' k_0} \| \phi \|_{S_k(I \times \R^2)} + 2^{(j'-j)k_0} \| \partial_x^{j} \phi \|_{S_k(I \times \R^2)}$$
for any $k_0$.  Optimising in $k_0$ we obtain the claim.

\section{Proof of Theorem \ref{apriori-thm2} assuming Theorem \ref{func}}\label{ap-sec}

In this (lengthy) section we assume Theorem \ref{func} and use it to prove Theorem \ref{apriori-thm2}.  Fix $E > 0$; all constants are allowed to depend on $E$.

\subsection{Definition of the metrics}\label{metric-def}

For each integer $k$ and $\Psi_{s,t,x}, \Psi'_{s,t,x} \in \WMC(I,E)$, define the quantity
\begin{equation}\label{dmc}
d_{\mu,k,I}(\Psi_{s,t,x},\Psi'_{s,t,x}) := \sum_{j=0}^{10} \sup_{2^{-2k-2} \leq s \leq 2^{-2k}} s^{1+\frac{j}{2}} \| \partial_x^j (\psi_s(s) - \psi'_s(s)) \|_{S_{\mu,k}(I \times \R^2)}
\end{equation}
and then define the metric $d_{S^1_\mu,I}$ by the formula
\begin{equation}\label{ds1-eq}
d_{S^1_\mu,I}(\Psi_{s,t,x},\Psi'_{s,t,x}) := (\sum_k d_{\mu,k,I}(\Psi_{s,t,x},\Psi'_{s,t,x})^2)^{1/2} 
\end{equation}

To ensure that these are actually metrics, we have to check finiteness and non-degeneracy:

\begin{lemma}[Finiteness]\label{lfin} For any $\Psi_{s,t,x}, \Psi'_{s,t,x} \in \WMC(I,E)$, $\tilde d_{S^1_\mu,I}(\Psi_{s,t,x},\Psi'_{s,t,x})$ is finite.  
\end{lemma}

\begin{proof}  By the triangle inequality it suffices to show that the sequence
$$ \sup_{2^{-2k-2} \leq s \leq 2^{-2k}} s^{1+\frac{j}{2}} \| \partial_x^j \psi_s(s) \|_{S_{\mu,k}(I \times \R^2)}$$
is square-summable in $k$ for $0 \leq j \leq 10$ and $\Psi_{s,t,x} \in \WMC(I,E)$.  

By \eqref{sksk-star} we may replace $S_{\mu,k}$ by $S_k$ here.  Applying \eqref{split-energy-est} we have
\begin{equation}\label{psisk-en}
\begin{split}
 \| \partial_x^j \psi_s(s) \|_{S_{k}(I \times \R^2)} &\lesssim 
\sum_{k'} \chi_{k=k'}^{-\delta_1}
[ \| \partial_x^j \partial_{t,x} P_{k'} \psi_s(s,t_0) \|_{L^2_x(\R^2)} \\
&\quad +
\| \partial_x^j \Box P_{k'} \psi_s(s) \|_{L^1_t L^2_x(I \times \R^2)}]
\end{split}
\end{equation}
where $t_0 \in I$ is arbitrary.

From Proposition \ref{qual-prop} and Bernstein's inequality we have
$$
 \| \partial_x^j \partial_{t,x} P_{k'} \psi_s(s,t) \|_{L^2_x(\R^2)} \lesssim_{\phi,j,m} 
\min( 2^{k'}, (\langle s \rangle 2^{2k'})^{-m} ) 
 \langle s \rangle^{-(j+2)/2}
$$
for any $m$.  This ensures that the contribution of the first term in \eqref{psisk-en} is acceptable.

As for the second term, we first recall from \cite[Lemma 7.5]{tao:heatwave2} that
\begin{equation}\label{wdamp-eq} \| \partial_x^j \partial_s w(s,t) \|_{L^1_x(\R^2)} \lesssim_{j,\phi} s^{-(j+2)/2}
\end{equation}
for $s \gtrsim 1.$
From \eqref{psis-box} we have
$$ \Box \psi_s = \bigO( \partial_s w ) + \bigO( \partial_{t,x} \Psi_{t,x} \psi_s ) + \bigO( \Psi_{t,x} \partial_{t,x} \psi_s ) + \bigO( \Psi_{t,x}^2 \psi_s ).$$
Using \eqref{wdamp-eq}, Proposition \ref{qual-prop}, \eqref{psreg-eq} we conclude that
$$ \| \partial_x^j \Box \psi_s \|_{L^1_x(\R^2)} \lesssim_{\phi,j} s^{-(j+2)/2}$$
for all $j \geq 0$ and $s \gtrsim 1$, which is enough to show that the second term in \eqref{psisk-en} is acceptable.
\end{proof}

\begin{lemma}[Non-degeneracy]
If $\Psi_{s,t,x}, \Psi'_{s,t,x} \in \WMC(I,E)$ are such that $\tilde d_{S^1_\mu,I}(\Psi_{s,t,x},\Psi'_{s,t,x})=0$, then $\Psi_{s,t,x} = \Psi'_{s,t,x}$.  
\end{lemma}

\begin{proof}  This is clear from construction (note from (say) \eqref{lstrich-4} that the only functions with vanishing $S_{\mu,k}(I \times \R^2)$ norm are identically zero).  
\end{proof}

\subsection{Easy verifications}

We now verify some of the easier components of Theorem \ref{apriori-thm2}.  The monotonicity of $\|\Psi_{s,t,x}\|_{S^1_\mu(I \times \R^2)}$ in $I$ follows from the monotonicity of the $S_{\mu,k}$ norms from Theorem \ref{func}.  The continuity of $\|\Psi_{s,t,x}\|_{S^1_\mu(I \times \R^2)}$ in $I$ similarly follows from the continuity of the $S_{\mu,k}$ norms and the dominated convergence theorem (using Lemma \ref{lfin} to provide the domination).  The quasi-isometry property \eqref{quasi-eq} follows easily by breaking everything up into components and using the triangle inequality.

Now we show the vanishing property.  Fix $\Psi_{s,t,x}, \Psi'_{s,t,x}$.  By Theorem \ref{envelop} we can find frequency envelopes $c_0$, $\delta c_0$ with the stated properties.

By construction of $d_{S^1_\mu,I_n}$, it will suffice to show that
$$ \lim_{n \to \infty} \sum_k d_{\mu,k,I}(\Psi_{s,t,x},\Psi'_{s,t,x})^2 \lesssim d_\Energy( \Psi_{s,t,x}, \Psi'_{s,t,x} )^2.$$
From \eqref{dmc}, \eqref{shrinko} (and continuity of the $S_{\mu,k,I}$ norm in the Schwartz topology), we know that
$$ \lim_{n \to \infty} d_{\mu,k,I}(\Psi_{s,t,x},\Psi'_{s,t,x}) \lesssim
\sum_{j=0}^{10}  \sup_{2^{-2k-2} \leq s \leq 2^{-2k}} s^{1+\frac{j}{2}} 
\sum_{k'} \chi_{k=k'}^{-\delta_1} \| \partial_x^j \partial_{t,x} P_{k'} (\psi_s-\psi'_s)(s,t_0) \|_{L^2_x(\R^2)}.
$$
Thus by the monotone convergence theorem, it suffices to show that
$$ \sum_k [\sup_{2^{-2k-2} \leq s \leq 2^{-2k}} s^{1+\frac{j}{2}} 
\sum_{k'} \chi_{k=k'}^{-\delta_1} \| \partial_x^j \partial_{t,x} P_{k'} (\psi_s-\psi'_s)(s,t_0) \|_{L^2_x(\R^2)}]^2 \lesssim d_\Energy( \Psi_{s,t,x}, \Psi'_{s,t,x} )^2.$$
But by \eqref{psioc-diff}, the expression in brackets is $O( \delta c_0(2^{-k}) )$, and the claim follows from \eqref{sm} and the fact that $\delta c_0$ has energy $O(d_\Energy( \Psi_{s,t,x}, \Psi'_{s,t,x} )^2)$.

\subsection{The stability priori estimate}\label{s1-aproiri}

We now establish \eqref{apriori2a}.  Fix $M$, $\mu$, $I$, $t_0$, $\Psi_{s,t,x}$, $\Psi'_{s,t,x}$, where we assume $\mu$ sufficiently small depending on $M$.  We allow all implied constants to depend on $M$, thus
$$ \|\Psi_{s,t,x}\|_{S^1_\mu(I \times \R^2)}, \|\Psi_{s,t,x}'\|_{S^1_\mu(I \times \R^2)} \lesssim 1.$$
Our task is to show that
\begin{equation}\label{dsmu}
d_{S^1_\mu,I}(\Psi_{s,t,x},\Psi'_{s,t,x}) \lesssim_{\mu} d_{\Energy}(\phi_{s,t,x}(t_0), \phi'_{s,t,x}(t_0)).
\end{equation}

By \eqref{ds1-eq}, we have
\begin{equation}\label{sumd}
 \sum_k d_{\mu,k,I}(\Psi_{s,t,x},0)^2, \sum_k d_{\mu,k,I}(\Psi'_{s,t,x},0)^2 \lesssim 1.
 \end{equation}
and our task is to show that
$$ \sum_k d_{\mu,k,I}(\Psi_{s,t,x},\Psi'_{s,t,x})^2 \lesssim_\mu d_{\Energy}(\Psi_{s,t,x}(t_0), \Psi'_{s,t,x}(t_0))^2.$$
In view of \eqref{sksk-star}, it suffices to show that
\begin{equation}\label{dkd}
 \sum_k d_{k}(\Psi_{s,t,x},\Psi'_{s,t,x})^2 \lesssim d_{\Energy}(\Psi_{s,t,x}(t_0), \Psi'_{s,t,x}(t_0))^2
 \end{equation}
where
$$ d_{k}(\Psi_{s,t,x},\Psi'_{s,t,x}) := \sum_{j=0}^{10} \sup_{2^{-2k-2} \leq s \leq 2^{-2k}} s^{1+\frac{j}{2}} \| \partial_x^j (\psi_s(s) - \psi'_s(s)) \|_{S_{k}}.$$

Now define
\begin{equation}\label{csdef}
 c(s) := \sum_k [d_{\mu,k,I}(\Psi_{s,t,x},0) + d_{\mu,k,I}(\Psi'_{s,t,x},0)] \min( (2^{2k} s)^{\delta_0}, (2^{2k} s)^{-\delta_0})
\end{equation}
and
\begin{equation}\label{csdef-delta}
 \delta c(s) := \sum_k d_k(\Psi_{s,t,x},\Psi'_{s,t,x}) \min( (2^{2k} s)^{\delta_0}, (2^{2k} s)^{-\delta_0}).
\end{equation}
Then $c, \delta c$ are frequency envelopes, and by \eqref{sumd} and Young's inequality we see that $c$ has energy $O(1)$. In particular from \eqref{cse}, \eqref{cse-sup} we have
\begin{equation}\label{envelo}
\sup_s c(s) + \int_0^\infty c(s)^2 \frac{ds}{s} \lesssim 1.
\end{equation}
From \eqref{csdef-delta}, we see that to prove \eqref{dkd}, it will suffice to show that
\begin{equation}\label{envelo-targ}
\int_0^\infty \delta c(s)^2 \frac{ds}{s} \lesssim d_{\Energy}(\phi[t_0], \phi'[t_0])^2.
\end{equation}

From \eqref{csdef}, \eqref{csdef-delta}, \eqref{physical}, \eqref{sumd}, we see that\footnote{There will be many powers of $s$ on the right-hand side of the estimates in this section, but one does not need to pay too much attention to the exponents here, as they are always equal to the exponent predicted by the dimensional analysis heuristics $\psi_{t,x}, \partial_{t,x}, A_{t,x}, 2^k \sim s^{-1/2}; \psi_s \sim s^{-1}; L^q_t \sim s^{-1/(2q)}; L^r_x \sim s^{-1/r}$, with $P_k$, $c(s)$, $\mu$ being dimensionless.}
\begin{equation}\label{psisjk}
\| \partial_x^j \psi^*_s(s) \|_{S_{\mu, k(s)}(I \times \R^2)} \lesssim c(s) s^{-(j+2)/2}
\end{equation}
and
\begin{equation}\label{psisjk-star}
\| \partial_x^j \delta \psi_s(s) \|_{S_{k(s)}(I \times \R^2)} \lesssim \delta c(s) s^{-(j+2)/2}
\end{equation}
for all $0 \leq j \leq 10$ and $s > 0$, where $k(s)$ was defined in \eqref{ks-def}.  From \eqref{psisjk}, \eqref{sksk-star} we of course have
\begin{equation}\label{psisjk-nomu}
\| \partial_x^j \psi^*_s(s) \|_{S_{k(s)}(I \times \R^2)} \lesssim c(s) s^{-(j+2)/2}
\end{equation}
for the same range of $j, s$.

The arguments used to prove Lemma \ref{lfin} also establish the vanishing property \eqref{psi-decay-1}.  We can now invoke Theorem \ref{parab-thm} to obtain the estimates \eqref{parab-1}-\eqref{parab-6} for all $j \geq 0$. We also have analogous bounds for $S_{\mu,k}$, for $\psi^*_s$ at least:

\begin{corollary}[Infinite gain of regularity in $S_{\mu,k}$]\label{infi-star} 
\begin{align}
\| \partial_x^j \psi^*_s(s) \|_{S_{\mu,k(s)}(I \times \R^2)} &\lesssim_{j,\eps} \mu^{-\eps} c(s) s^{-(j+2)/2}\label{s-star}
\end{align}
for all $j \geq 0$, $\eps > 0$ and $s > 0$.
\end{corollary}

\begin{proof}  From \eqref{psisjk-star} we have \eqref{s-star} for $j \leq 10$ (with no loss of $\mu^{-\eps}$), while from \eqref{parab-3} and \eqref{sksk-star} we have \eqref{s-star} for all $j$ with a loss of $\mu^{-O(1)}$.  Interpolating using \eqref{interp-eq} we conclude \eqref{s-star} for all $j$ and all $\eps > 0$.
\end{proof}

This, together with \eqref{psisjk}, leads to the following Strichartz estimates which will be useful for disposing of higher order terms in the nonlinearity.

\begin{lemma}[Strichartz estimates]\label{strichlem}  We have
\begin{align}
\| \partial_x^j \partial_{t,x} \Psi^*_x(s) \|_{L^q_t L^\infty_x(I \times \R^2)} &\lesssim_{j,\eps} \mu^{(5-\eps)/q} c(s) s^{-(j+2)/2+1/(2q)}  \label{strich1}\\
\| \partial_x^j \psi^*_s(s) \|_{L^q_t L^\infty_x(I \times \R^2)} &\lesssim_{j,\eps} \mu^{(5-\eps)/q} c(s) s^{-(j+2)/2+1/(2q)} \label{strich2}\\
\| \partial_x^j \partial_{t,x} \psi^*_s(s) \|_{L^q_t L^\infty_x(I \times \R^2)} &\lesssim_{j,\eps} \mu^{(5-\eps)/q} c(s) s^{-(j+4)/2+1/(2q)}  \label{strich3}\\
\| \partial_x^j \Psi^*_{t,x}(s) \|_{L^q_t L^\infty_x(I \times \R^2)} &\lesssim_{j,\eps} \mu^{(5-\eps)/q} c(s) s^{-(j+1)/2+1/(2q)}  \label{strich4}\\
\| \partial_x^j A^*_{t,x}(s) \|_{L^r_t L^\infty_x(I \times \R^2)} &\lesssim_{j,\eps} \mu^{(5-\eps)/2r} c(s) s^{-(j+1)/2+1/(2r)}  \label{strich5}\\
\| \partial_x^j \partial_{t,x} \Psi^*_x(s) \|_{L^\infty_t L^2_x(I \times \R^2)} &\lesssim_j c(s) s^{-(j+1)/2}\label{strich6} \\
\| \partial_x^j \psi^*_s(s) \|_{L^\infty_t L^2_x(I \times \R^2)} &\lesssim_j c(s) s^{-(j+1)/2}\label{strich7} \\
\| \partial_x^j \partial_{t,x} \psi^*_s(s) \|_{L^\infty_t L^2_x(I \times \R^2)} &\lesssim_j c(s) s^{-(j+2)/2}\label{strich8} 
\end{align}
for all $j \geq 0$, $s > 0$, $5 \leq q \leq \infty$, $5/2 \leq r \leq \infty$, and $\eps > 0$.  One also has analogues of all the above estimates in which $\Psi^*$ is replaced by $\delta \Psi$, $c$ is replaced by $\delta c$, etc., and all powers of $\mu$ are discarded.  Thus for instance, the analogue of \eqref{strich1} is
\begin{equation}\label{strich1-diff}
\| \partial_x^j \partial_{t,x} \delta \Psi_x(s) \|_{L^q_t L^\infty_x(I \times \R^2)} \lesssim_{j}  \delta c(s) s^{-(j+2)/2+1/(2q)}.
\end{equation}
\end{lemma}

\begin{proof}  We first establish the estimates for $\Psi^*_{s,t,x}$, and return to $\delta \Psi_{s,t,x}$ later.

From Corollary \ref{infi-star}, \eqref{physical} we have
$$ \| P_k \Psi^*_x(s) \|_{S_{\mu,k}(I \times \R^2)} \lesssim_{j',\eps} \mu^{-\eps} \chi_{k=k(s)}^{-\delta_1} c(s) s^{-1/2} \langle 2^k s^{-1/2} \rangle^{-j'}$$
for all $k, j', \eps$, and similarly for the $S_k(I \times \R^2)$ norm without the $\mu^{-\eps}$ loss (thanks to \eqref{parab-1}).  By \eqref{lstrich-4}, \eqref{lstrich} we thus have
$$ \| \partial_x^j \partial_{t,x} P_k \Psi^*_x \|_{L^5_t L^\infty_x(I \times \R^2)} \lesssim_{j,j'}
\mu^{1-\eps} 2^{kj} 2^{(1-1/5) k} c(s) s^{-1/2} \langle 2^k s^{-1/2} \rangle^{-j'}$$
and
$$ \| \partial_x^j \partial_{t,x} P_k \Psi^*_x \|_{L^\infty_t L^\infty_x(I \times \R^2)} \lesssim_{j,j'} 2^{kj} 2^{k} c(s) s^{-1/2} \langle 2^k s^{-1/2} \rangle^{-j'}$$
for any $j,j'$.  Interpolating these estimates, taking $j'=j+5$ (say) and summing in $k$, we conclude \eqref{strich1}.

From \eqref{strich1}, \eqref{heatflow}, the Leibniz rule \eqref{leibnitz}, and H\"older's inequality we obtain \eqref{strich2}, \eqref{strich3}.

Now we turn to \eqref{strich4}.  If we set
$$ f_{j,q}(s) := \sum_{j'=0}^j c(s)^{-1} s^{(j+1)/2-1/2q} \| \partial_x^j \Psi^*_{t,x}(s) \|_{L^q_t L^\infty_x(I \times \R^2)}$$
then from \eqref{Psi-tx-eq}, Minkowski's inequality, the Leibniz rule, H\"older's inequality, \eqref{strich2}, \eqref{strich3}, and \eqref{sm} we conclude that
$$ f_{j,q}(s) \lesssim_{j,\eps} \mu^{(5-\eps)/q} + \int_{s'}^\infty (s'/s)^{-\sigma} f_{j,q}(s') c(s') ds'/s'$$
for some absolute constant $\sigma > 0$ (independent of $\delta_0$) and all $\eps > 0$.  Noting from Proposition \ref{qual-prop}, Bernstein's inequality, and H\"older's inequality that $f_{j,q}(s) \to 0$ as $s \to \infty$.  Applying Lemma \ref{gron-lem} we conclude that $f_{j,q}(s) \lesssim_{j,\eps} \mu^{(5-\eps)/q}$ for all $s > 0$ and $\eps > 0$, and \eqref{strich4} follows.

The estimate \eqref{strich5} follows from \eqref{a-eq}, the Leibniz rule, Minkowski's inequality, H\"older's inequality, \eqref{strich2}, \eqref{strich4}, and \eqref{sm}, \eqref{envelo}.  

The estimate \eqref{strich6} follows from \eqref{parab-1} and \eqref{outgo}.  The estimates \eqref{strich7}, \eqref{strich8} then follows from \eqref{strich6}, \eqref{strich4}, and \eqref{heatflow}.  

The analogous estimates for $\delta \Psi$ are obtained by applying \eqref{disc-leib-eq} to all the equations of motion used above, and modifying the arguments appropriately (without trying to gain any factors of $\mu$).
\end{proof}

Having obtained adequate control on $\Psi^*_{s,t,x}$, $\Psi'_{s,t,x}$, we now turn attention to the wave-tension fields $w = D^\alpha \psi_\alpha$, $w' = (D')^\alpha \psi'_\alpha$, $\delta w = w' - w$.  We begin with the nonlinear forcing term 
\begin{equation}\label{Fdef}
F := (\psi_\alpha \wedge \psi_i) D_i \psi^\alpha
\end{equation}
appearing in \eqref{w-eq}, with $F'$ (and hence $F^*$ and $\delta F$) defined accordingly.  Here we can obtain estimates which (crucially) gains more than one power of $\mu$, as well as some decay at low frequencies.  

\begin{lemma}[Forcing term estimate]\label{force}  We have 
\begin{equation}\label{force-eq} \| P_k F^*(s) \|_{N_k(I \times \R^2)} \lesssim_{j,\eps} \mu^{2-\eps} c(s) s^{-1} \chi_{k \geq k(s)}^{\delta_2} \chi_{k \leq k(s)}^j
\end{equation}
and
\begin{equation}\label{force-diff-eq} \| P_k \delta F(s) \|_{N_k(I \times \R^2)} \lesssim_{j,\eps} \mu^{1-\eps} \kappa c(s) s^{-1} \chi_{k \geq k(s)}^{\delta_2} \chi_{k \leq k(s)}^j
\end{equation}
for all $j \geq 0$, $s > 0$, and integers $k$, and $\eps > 0$.
\end{lemma}

\begin{proof}  
We will need to extract the implicit null structure from the expression $F$ by exploiting ``dynamic separation'' as in \cite{krieger:2d}, \cite{krieger:3d}.  In the context of the caloric gauge, dynamic separation entails using \eqref{psi-eq}, \eqref{a-eq} to rewrite the terms involving $\alpha$ in \eqref{Fdef} in terms of $\psi_s$, modulo higher order terms.  Indeed, from \eqref{psi-eq} we can express $F(s)$ as the sum of the cubic term
\begin{equation}\label{F1-eq}
\int_s^\infty \int_s^\infty \bigO( \partial_\alpha \psi_s(s') \psi_x(s) \partial_x \partial^\alpha \psi_s(s'') )\ ds' ds'' 
\end{equation}
the quintic\footnote{Here we count $A$ as a quadratic term (cf. \eqref{a-eq}, \eqref{ax-eq}; also note the range of exponents in \eqref{strich5} is twice as large as \eqref{strich1}-\eqref{strich3}).} terms
\begin{align}
\int_s^\infty \int_s^\infty \bigO( \partial_\alpha \psi_s(s') \psi_x(s) A_x(s) \partial^\alpha \psi_s(s'') )&\ ds' ds''\label{F2-a}\\
\int_s^\infty \int_s^\infty \bigO( A_\alpha(s') \psi_s(s') \psi_x(s) \partial_x \partial^\alpha \psi_s(s''))&\ ds' ds'' \label{F2-b}\\
\int_s^\infty \int_s^\infty \bigO( \partial_\alpha \psi_s(s') \psi_x(s) \partial_x (A^\alpha(s'') \psi_s(s'')))&\ ds' ds'' \label{F2-c}
\end{align}
the septic terms
\begin{align}
\int_s^\infty \int_s^\infty \bigO( A_\alpha(s') \psi_s(s') \psi_x(s) A_x(s) \partial^\alpha \psi_s(s'') )&\ ds' ds''\label{F3-a}\\
\int_s^\infty \int_s^\infty \bigO( A_\alpha(s') \psi_s(s') \psi_x(s) \partial_x(A^\alpha(s'') \psi_s(s'')))&\ ds' ds'' \label{F3-b}\\
\int_s^\infty \int_s^\infty \bigO( \partial_\alpha \psi_s(s') \psi_x(s) A_x(s) A^\alpha(s'') \psi_s(s''))&\ ds' ds'' \label{F3-c}
\end{align}
and the nonic term
\begin{equation}\label{F4}
\int_s^\infty \int_s^\infty \bigO( A_\alpha(s') \psi_s(s') \psi_x(s) A_x(s) A^\alpha(s'') \psi_s(s''))\ ds' ds''.
\end{equation}
Note in each of these expressions, the derivatives such as $\partial_\alpha$ or $\partial_x$ are falling on ``low frequency'' terms (terms arising from large values $s',s''$ of the heat-temporal parameter, rather than from small values such as $s$).  This phenomenon, which also occurs in  the integral expressions appearing later in this proof, will be crucial in ensuring that the integrals are convergent in the required function space norms.

We now turn to the proof of \eqref{force-eq}.  It suffices to show that
$$
\| P_k \partial_x^j F^*(s) \|_{N_k(I \times \R^2)} \lesssim_{j,\eps} \mu^{2-O(\eps)} c(s) s^{-(j+2)/2} \chi_{k \geq k(s)}^{\delta_2}
$$
for all $j \geq 0$, $s > 0$, $k \in \Z$, $\eps > 0$.  We split $F^*$ into terms of the form \eqref{F1-eq}-\eqref{F4} (but with $\Psi_{s,t,x}$ replaced by $\Psi^*_{s,t,x}$, etc.).

We first deal with the cubic term \eqref{F1-eq}.  By Minkowski's inequality, the contribution of this term is bounded by
$$ \lesssim \int_s^\infty \int_s^\infty \| P_k \partial_x^j (\partial_\alpha \psi^*_s(s') \psi^*_x(s) \partial_x \partial^\alpha \psi^*_s(s'')) \|_{N_k(I \times \R^2)}\ ds' ds''.$$
Applying the Leibniz rule \eqref{leibnitz}, \eqref{trilinear-improv2}, \eqref{parab-1}, Corollary \ref{infi-star}, and \eqref{sm}, \eqref{envelo}, we can bound this by
$$ \lesssim_{j,\eps} 
\int_s^\infty \int_s^\infty s^{-j/2} \mu^{2-O(\eps)} c(s) \chi_{k = k(s)}^{\eps \delta_1} (\max(s',s'')/s)^{-\eps \delta_1/2} (s')^{-1} s^{-1/2} (s'')^{-3/2}\ ds' ds'';$$
performing the integrals we see that these terms are acceptable.  Note that the $j>0$ cases are no harder than the $j=0$ case (and in some terms there is even a slight gain); the reader may in fact wish to set $j=0$ when following the discussion below, as the higher $j$ case never adds any substantial new difficulty.

The first quintic term \eqref{F2-a} can be handled similarly.  Indeed, from \eqref{parab-1}, \eqref{prod1}, \eqref{envelo}, and the Leibniz rule \eqref{leibnitz}, we have
$$ \| \partial_x^j (\psi_x(s) A_x(s)) \|_{S_{k(s)}(I \times \R^2)} \lesssim_j c(s) s^{-j/2} s^{-1}$$
for all $j \geq 0$, and then by repeating the previous arguments we can estimate the contribution of this case by
$$ \lesssim_{j,\eps} 
\int_s^\infty \int_s^\infty s^{-j/2} \mu^{2-O(\eps)} c(s) \chi_{k = k(s)}^{\eps \delta_1} (\max(s',s'')/s)^{-\eps \delta_1/2} (s')^{-1} s^{-1} (s'')^{-1}\ ds' ds''$$
which is also (barely) acceptable.

We set aside the other quintic terms for now and look at the nonic term \eqref{F4}.  By Minkowski's inequality, \eqref{fl1l2}, and Bernstein's inequality, and discarding the null structure, the contribution of this term is bounded by
$$ \lesssim 2^k \int_s^\infty \int_s^\infty \| P_k \partial_x^j ( A^*_{t,x}(s') \psi^*_s(s') \psi^*_x(s) A^*_x(s) A^*_{t,x}(s'') \psi^*_s(s'')) \|_{L^1_t L^1_x(I \times \R^2)} \ ds' ds''.$$
We apply Leibniz's rule \eqref{leibnitz} and H\"older's inequality, estimating the ``high frequency'' terms $\psi_x(s)$, $A^*_x(s)$ in $L^\infty_t L^2_x$, the term $A^*_{t,x}(s')$ (say) in $L^{5/2}_t L^\infty_x$, amd the other three terms $A^*_{t,x}(s'')$, $\psi_s(s')$, $\psi_s(s'')$ in $L^5_t L^\infty_x$.  Applying Lemma \ref{strichlem}, \eqref{envelo} we can estimate the above expression by
$$ \lesssim_{j,\eps} 2^k \int_s^\infty \int_s^\infty s^{-j/2} \mu^{2-O(\eps)} c(s) (s')^{-3/2+1/5}  (s'')^{-3/2+3/10}\ ds' ds''.$$
(in fact we have several powers of $\mu$ to spare).
Performing the integrals we see that this term is acceptable.

The septic terms can be treated by the same methods as the nonic term, but with a more delicate numerology.  To illustrate this, consider for instance \eqref{F3-c}.  If we do not apply Bernstein's inequality, but otherwise repeat the above argument, we can control the contribution of this term by
$$ \lesssim \int_s^\infty \int_s^\infty \| P_k \partial_x^j ( \partial_{t,x}(s') \psi^*_s(s') \psi^*_x(s) A^*_x(s) A^*_{t,x}(s'') \psi^*_s(s'')) \|_{L^1_t L^2_x(I \times \R^2)} \ ds' ds''.$$
If we place $\psi^*_x(s)$ in $L^\infty_t L^2_x$, $A^*_x(s)$ and $A^*_{t,x}(s'')$ in $L^{10/3}_t L^\infty_x$ (say), and $\partial_{t,x}(s') \psi^*_s(s')$ and $\psi^*_s(s'')$ in $L^5_t L^\infty_x$, we obtain the integral
$$ \lesssim_{j,\eps} \int_s^\infty \int_s^\infty s^{-j/2} \mu^{2-O(\eps)} c(s) (s')^{-3/2+1/10}  s^{-1/2+3/20} (s'')^{-3/2+1/4}\ ds' ds''$$
(again, we have powers of $\mu$ to spare).  This integral is almost acceptable, except that 
we did not gain the factor of $\chi_{k \geq k(s)}^{\delta_2}$.  For this, we must invoke Bernstein's inequality a little bit, moving $L^1_t L^2_x$ to $L^1_t L^{2-c}_x$ for some small $c$, and then moving some of the exponents in H\"older's inequality on, say, $A^*_x(s)$, to match (by interpolating the various estimates in Lemma \ref{strichlem}); note that scale-invariance assures us that we end up with the right exponent of $s$ in the end.  We leave the details to the reader.  The treatment of the terms \eqref{F3-a}, \eqref{F3-b} are similar (with some permutations in the exponents) and are also left to the reader.

The quintic terms \eqref{F2-b}, \eqref{F2-c} cannot be directly treated by Strichartz methods.  (Even with the endpoint $L^4_t L^\infty_x$ Strichartz estimate, one would barely be able to place the nonlinearity in $L^1_t L^2_x$, leaving no room for the Bernstein inequality to give the important factor of $\chi_{k \geq k(s)}^{\delta_2}$.  This issue also comes up in \cite{krieger:2d}.)  To deal with this, we need to perform more dynamic separation on these terms.  For instance, to deal with \eqref{F2-b}, we use \eqref{a-eq}, \eqref{psi-eq} to write
$$ A^*_\alpha(s') = \int_{s'}^\infty \int_{s'''}^\infty \bigO( \psi^*_s(s''') \partial_\alpha \psi^*_s(s'''') ) + \bigO( \psi^*_s(s''') A^*_\alpha(s'''') \psi^*_s(s'''') )\ ds'''' ds'''.$$
This splits \eqref{F2-b} into a quintic null form
$$
\int_s^\infty \int_s^\infty \int_{s'}^\infty \int_{s'''}^\infty \bigO( \psi^*_s(s''') \partial_\alpha \psi^*_s(s'''') \psi^*_s(s') \psi^*_x(s) \partial_x \partial^\alpha \psi^*_s(s''))\ ds'''' ds''' ds' ds''$$
and a septic expression
$$
\int_s^\infty \int_s^\infty \int_{s'}^\infty \int_{s'''}^\infty \bigO( \psi^*_s(s''') A^*_\alpha(s'''') \psi^*_s(s'''') \psi^*_s(s') \psi^*_x(s) \partial_x \partial^\alpha \psi^*_s(s''))\ ds'''' ds''' ds' ds''.$$
The septic expression can be dealt with by the same sort of Strichartz and Bernstein techniques as the previous septic expressions \eqref{F3-a}, \eqref{F3-b}, \eqref{F3-c}, while the quintic null form can be dealt with similarly to the quintic null form \eqref{F2-a}.  Similar arguments let one handle \eqref{F2-c}.  We omit the details as they are very similar to the expressions already dealt with, except for some minor permutations in the numerology (the key point again being that the derivatives are falling on low frequency terms rather than high frequency ones).

This ends our discussion of \eqref{force-eq}.  To prove \eqref{force-diff-eq}, one applies \eqref{disc-leib-eq} to all the expressions above, thus replacing one of the $\Psi^*$ factors appearing above with a $\delta \Psi$ factor (for various values of ``$\Psi$'').  One then repeats the above arguments, using the appropriate ``$\delta$ versions'' of estimates such as \eqref{parab-1} or Lemma \ref{strichlem}, and using \eqref{trilinear-improv} instead of \eqref{trilinear-improv2}.  Because the estimates on $\delta$ terms do not have any gains in $\mu$ (and in particular are estimated in $S_k$ rather than $S_{\mu,k}$), we only end up with a power of $\mu^{1-O(\eps)}$ rather than $\mu^{2-O(\eps)}$.  Instead of keeping a factor of $c(s)$ in the final estimate, one instead retains a factor such as $\delta c(s)$, $\delta c(s')$, $\delta c(s'')$, etc. (depending on where the $\delta$ term has fallen); however, by using \eqref{sm} one can convert this back to $\delta c(s)$ at a cost of an expression such as $(s'/s)^{\delta_0}$, which does not affect the convergence of any of the integrals (as they already have some room to spare in these exponents, typically of the order of $1/10$ or so).  We again leave the details to the reader, as they are somewhat tedious.
\end{proof}

We are now in a position to obtain control on the wave-tension field $w = D^\alpha \psi_\alpha$, that also gains the crucial $\mu$ factor:

\begin{lemma}[Control of wave-tension field]\label{pkow}  We have
\begin{equation}\label{pkw-1}
\| P_k w^* \|_{N_k(I \times \R^2)} \lesssim_{j,\eps} \mu^{2-\eps} c(s) \chi_{k \geq k(s)}^{\delta_2/10} \chi_{k \leq k(s)}^j
\end{equation}
and
\begin{equation}\label{pkw-2}
 \| P_k \partial_s w^* \|_{N_k(I \times \R^2)} \lesssim_{j,\eps} \mu^{2-\eps} c(s) s^{-1} \chi_{k \geq k(s)}^{\delta_2/10} \chi_{k \leq k(s)}^j
 \end{equation}
for all $j \geq 0$, $s > 0$, $\eps > 0$, and integers $k$.  Similarly we have
\begin{equation}\label{pkw-1-diff}
\| P_k \delta w \|_{N_k(I \times \R^2)} \lesssim_{j,\eps} \mu^{1-\eps} \delta c(s) \chi_{k \geq k(s)}^{\delta_2/10} \chi_{k \leq k(s)}^j
\end{equation}
and
\begin{equation}\label{pkw-2-diff}
 \| P_k \partial_s \delta w \|_{N_k(I \times \R^2)} \lesssim_{j,\eps} \mu^{1-\eps} \delta c(s) s^{-1} \chi_{k \geq k(s)}^{\delta_2/10} \chi_{k \leq k(s)}^j
 \end{equation}
for the same range of $j,s,\eps,k$.
\end{lemma}

\begin{proof}  
We begin with the proof of \eqref{pkw-2}.
Fix $j$.  Without loss of generality we may take $j \geq 10$ (say).

From \eqref{w-eq} we have the schematic heat equation
\begin{equation}\label{wsax}
\partial_s w^* = \Delta w^* + \bigO( A^*_x \partial_x w^* ) + \bigO( (\partial_x A^*_x) w^* ) + \bigO( (\Psi^*_x)^2 w^* ) + \bigO( F^* ).
\end{equation}
From Duhamel's formula \eqref{duh} and \eqref{w-vanish} we conclude that
\begin{equation}\label{w-duh}
w^*(s) = \int_0^s e^{(s-s')\Delta} (\bigO( A^*_x \partial_x w^* ) + \bigO( (\partial_x A^*_x) w^* ) + \bigO( (\Psi^*_x)^2 w^* ) + \bigO( F^* ))(s')\ ds'.
\end{equation}

Let
$$f_j(s) := \sup_k c(s)^{-1} \chi_{k \geq k(s)}^{-\delta_2/10} \chi_{k \leq k(s)}^{-j} \| P_k w^* \|_{N_k(I \times \R^2)},$$
thus
\begin{equation}\label{pkw}
\| P_k w^* \|_{N_k(I \times \R^2)} \lesssim f_j(s) c(s) \chi_{k \geq k(s)}^{\delta_2/10} \chi_{k \leq k(s)}^{j}
\end{equation}
for all $k$.  The claim \eqref{pkw-1} is then equivalent to showing that $f_j(s) \lesssim_{j,\eps} \mu^{2-\eps}$ for all $s>0$, $j \geq 0$, and $\eps > 0$.

Observe that the convolution kernel of $P_k e^{(s-s')\Delta}$ has total mass $O_j( \chi_{k \leq k(s-s')}^{100j} )$ for any $j$. From \eqref{w-duh} and Minkowski's inequality, we can thus bound $f_j(s)$ by
\begin{align*}
 \lesssim_j &\sup_k c(s)^{-1} \chi_{k \geq k(s)}^{-\delta_2} \chi_{k \leq k(s)}^{-j} 
\int_0^s \chi_{k \leq k(s-s')}^{100j} \times \\
&\quad \| P_k (\bigO( A^*_x \partial_x w^* ) + \bigO( (\partial_x A^*_x) w^* ) + \bigO( (\Psi^*_x)^2 w^* ) + \bigO( F^* ))(s') \|_{N_k(I \times \R^2)}\ ds'.
\end{align*}
For any $s' > 0$, we see from \eqref{parab-1}, \eqref{parab-2}, \eqref{prod1} that
$$ \| \partial_x A^*_x(s') \|_{S_{k(s')}(I \times \R^2)}, \| (\Psi^*_x)^2(s') \|_{S_{k(s')}(I \times \R^2)} \lesssim c(s')^2 (s')^{-1}$$
and thus by \eqref{pkw}, \eqref{second-prod} and dyadic decomposition we have
$$ \| P_k (\bigO( (\partial_x \Psi^*_x) w^* ) + \bigO( (\Psi^*_x)^2 w^* ))(s') \|_{N_k(I \times \R^2)}
\lesssim_{j} (s'/s)^{\delta_2/2} c(s')^3 f_j(s') (s')^{-1} \chi_{k \geq k(s')}^{\delta_2/10} \chi_{k \leq k(s')}^{j}.$$
A similar argument also gives
$$\| P_k (\bigO( \Psi^*_x \partial_x w^* ) \|_{N_k(I \times \R^2)}
\lesssim_{j} c(s')^3 (s'/s)^{\delta_2/2} f_j(s') (s')^{-1} \chi_{k \geq k(s')}^{\delta_2/10} \chi_{k \leq k(s')}^{j-1}.$$ 
Combining these estimates with Lemma \ref{force}, we conclude that
\begin{align*}
f_j(s) &\lesssim_{j,\eps} \sup_k c(s)^{-1} \chi_{k \geq k(s)}^{-\delta_2/10} \chi_{k \leq k(s)}^{-j} \\
&\quad \int_0^s \chi_{k \leq k(s-s')}^{100j} [(s'/s)^{\delta_2/2} c(s')^3 f_j(s') \chi_{k \geq k(s')}^{\delta_2/10} \chi_{k \leq k(s')}^{j-1} \\
&\quad\quad + \mu^{2-\eps} c(s') \chi_{k \geq k(s')}^{\delta_2} \chi_{k \leq k(s')}^{100j} ]\ \frac{ds'}{s'}
\end{align*}
for all $\eps > 0$.

A direct computation using \eqref{sm} and the elementary bounds
$$ \chi_{k \geq k(s)}^{-\delta_2/10} \leq \chi_{k \geq k(s')}^{-\delta_2/10}; \quad \chi_{k \leq k(s)}^{-j} \lesssim_j \chi_{k \leq k(s-s')}^{-j} \chi_{k \leq k(s')}^{-j}$$
shows that
$$ c(s)^{-1} \chi_{k \geq k(s)}^{-\delta_2/10} \chi_{k \leq k(s)}^{-j} \int_0^s \chi_{k \leq k(s-s')}^{100j} c(s') \chi_{k \geq k(s')}^{\delta_2} \chi_{k \leq k(s')}^{100j}\ \frac{ds'}{s'} \lesssim_j 1.$$
Using this, \eqref{sm}, and the elementary bounds
\begin{align*}
\chi_{k \geq k(s)}^{-\delta_2/10} \chi_{k \geq k(s')}^{\delta_2/10} &\lesssim (s'/s)^{-\delta_2/20} \\
\chi_{k \leq k(s)}^{-j} &\lesssim \chi_{k \leq k(s-s')}^{-j} \chi_{k \leq k(s')}^{-j} \\
\chi_{k \leq k(s-s')} \chi_{k \leq k(s')}^{-1} &\lesssim ((s-s')/s)^{-1/2}\\
\end{align*}
then gives the integral inequality
$$ f_j(s) \lesssim_{j,\eps} \mu^{2-\eps} + \int_0^s ((s-s')/s)^{-1/2} (s'/s)^{\delta_2/4} c(s')^2 f_j(s') \frac{ds'}{s'}$$
for all $\eps > 0$. Also, from \eqref{fl1l2}, Proposition \ref{qual-prop}, Bernstein's inequality, and H\"older's inequality we have $f_j(s) \to 0$ as $s \to 0$.  Applying Lemma \ref{gron-lem-2} we have $f_j(s) \lesssim_{j,\eps} \mu^{2-\eps}$, and the claim \eqref{pkw-1} follows.

The claim \eqref{pkw-2} then follows by using \eqref{wsax} to write $\partial_s w^*$ in terms of $\partial_x^2 w^*$ (which can be controlled by \eqref{pkw-1}), together with several additional terms which were already estimated in the required manner in the first part of the proof.

The proof of \eqref{pkw-1-diff} follows from \eqref{pkw-1} as in previous propositions.  A little more specifically, we define
$$\delta f_j(s) := \sup_k \delta c(s)^{-1} \chi_{k \geq k(s)}^{-\delta_2/10} \chi_{k \leq k(s)}^{-j} \| P_k \delta w \|_{N_k(I \times \R^2)}$$
and repeat the above arguments (and using the bound \eqref{pkw-1} just established) to eventually obtain the integral inequality
$$ \delta f_j(s) \lesssim_{j,\eps} \mu^{1-\eps} + \int_0^s ((s-s')/s)^{-1/2} (s'/s)^{\delta_2/4} c(s')^2 \delta f_j(s') \frac{ds'}{s'}$$
Using the continuity method as before we obtain $\delta f_j(s) \lesssim_{j,\eps} \mu^{1-\eps}$ for all $s>0$, $j \geq 0$, and $\eps > 0$, giving \eqref{pkw-1-diff}.  The proof of \eqref{pkw-2-diff} then follows by the differenced version of \eqref{wsax}.  We leave the details to the reader.
\end{proof}

Having earned the crucial gains in $\mu$, we now return to control $\psi_s$.

\begin{lemma}[Return to $\psi_s$]\label{psis-lem}  We have
\begin{equation}\label{psis-iter-eq} \| P_k \psi^*_s(s) \|_{S_k(I \times \R^2)} \lesssim_{j,\eps}  \| \partial_{t,x} P_k \psi^*_s(s,t_0) \|_{L^2_x(\R^2)} + \mu^{2-\eps} c(s) s^{-1} \chi_{k \geq k(s)}^{\delta_2/10} \chi_{k \leq k(s)}^j 
\end{equation}
and
\begin{equation}\label{psis-iter-diff} \| P_k \delta \psi_s(s) \|_{S_{k}(I \times \R^2)} \lesssim_{j,\eps} \| \partial_{t,x} P_k \delta \psi_s(s,t_0) \|_{L^2_x(\R^2)} + \mu^{1-\eps} \delta c(s) s^{-1} \chi_{k \geq k(s)}^{\delta_2/10} \chi_{k \leq k(s)}^j 
\end{equation}
for all $j \geq 0$, $\eps > 0$, $s > 0$, and integers $k$.
\end{lemma}

\begin{remark} The estimate \eqref{psis-iter-eq} will not be directly needed for the stability estimate \eqref{apriori3}, but is useful for the persistence of regularity estimate \eqref{apriori2} in the next section.
\end{remark}

\begin{proof}  We begin by proving \eqref{psis-iter-eq}.
Let us fix $j,k,s,\eps$, and suppress dependence of implied constants on $j,\eps$.  
From \eqref{psis-box} we have the schematic wave equation
$$ \Box \psi^*_s(s) = \bigO( G_1 ) + \bigO( G_2 ) + \bigO( G_3 ) + \bigO( G_4 ) + \bigO( G_5 )$$
where
\begin{align*}
G_1 &:= A^*_\alpha(s) \partial^\alpha \psi^*_s(s)  \\
G_2 &:=(\partial^\alpha A^*_\alpha(s)) \psi^*_s(s) \\
G_3 &:= (A^*)^\alpha(s) A^*_\alpha(s) \psi^*_s(s) \\
G_4 &:= (\psi^*)^\alpha(s) \psi^*_\alpha(s) \psi^*_s(s) \\
G_5 &:= \partial_s w^*(s).
\end{align*}
Applying \eqref{energy-est}, we have
$$ \| P_k \psi^*_s(s) \|_{S_k(I \times \R^2)} \lesssim \| \partial_{t,x} P_k \psi^*_s(s,t_0) \|_{L^2_x(\R^2)} + \sup_{1 \leq i \leq 5} \| P_{k} G_i \|_{N_{k'}(I \times \R^2)}.$$
It thus suffices to show that
\begin{equation}\label{pkgi-targ}
\| \partial_x^j P_{k} G_i \|_{N_{k}(I \times \R^2)} \lesssim \mu^{2-\eps} c(s) s^{-(j+2)/2} \chi_{k \geq k(s)}^{\delta_2/10} 
\end{equation}
for $i=1,2,3,4,5$.  

The term $G_5$ can be treated by Lemma \ref{pkow}.  To handle the other terms $G_1,G_2,G_3,G_4$, it will be necessary to use dynamic separation to make the null structure in these expressions more explicit, creating cubic and quintic expressions with explicit null structure, together with sextic and higher order terms which can be handled by Strichartz estimates.  Unfortunately, this process creates a large number\footnote{With the decomposition we use, there are 28 (!) terms, though symmetry allows one to identify a few of these terms together.  One could cut down the number of terms somewhat by establishing a certain quintilinear estimate from the function spaces $S_k, S_{*,k}$ to $N_k$, as remarked earlier, thus relieving the need to deal with septilinear terms, but this does not seem to result in a net gain in simplicity for the paper and we have not pursued this approach.} of terms to estimate, and so it will be convenient to introduce some notation to partially unify the treatment of these terms.

Given any integers $J \geq 0$ and $\tilde J, J' \geq 1$ and formal symbols $D, D' \in \{\partial, A^*\}$, we let $G^{D,D'}_{\tilde J, J,J'}$ denote the schematic quantity
\begin{equation}\label{gdd}
G^{D,D'}_{\tilde J,J,J'} := \int_\ast \psi^*_s(\tilde s_1) \ldots \psi^*_s(\tilde s_{\tilde J}) \psi^*_s(s_1) \ldots \psi^*_s(s_{J-1}) D_\alpha \psi_s(s_J) \psi^*_s(s'_1) \ldots \psi^*_s(s'_{J'-1}) (D')^\alpha \psi^*_s(s'_{J'})
\end{equation}
where $\tilde s_1 := s$ and the $\ast$ subscript denotes an integration of the variables $\tilde s_2,\ldots,\tilde s_{\tilde J} s_1,\ldots,s_J,s'_1,\ldots,s'_{J'}$ over the region
$$ s = \tilde s_1 \leq \ldots \leq \tilde s_{\tilde J} \leq s_1 \leq \ldots \leq s_J; \quad \tilde s_{\tilde J} \leq s'_1 \leq \ldots \leq s'_{J'}$$
and we adopt the convention
$$
\psi^*_s(\tilde s_{\tilde J}) \psi^*_s(s_1) \ldots \psi^*_s(s_{J-1}) D_\alpha \psi^*_s(s_{J}) := D_\alpha \psi^*_s(\tilde s_{\tilde J})
$$
if $J'=0$.  Thus for instance
\begin{align*}
G^{\partial,\partial}_{1,0,2} &= \int_{s'_2 \geq s'_1 \geq s} \partial_\alpha \psi^*_s(s) \psi^*_s(s'_1) \partial^\alpha \psi^*_s(s'_2)\ ds'_1 ds'_2 \\
G^{A^*,\partial}_{1,1,1} &= \int_{s_1, s'_1 \geq s} \psi^*_s(s) A^*_\alpha(s_1) \psi^*_s(s_1) \partial^\alpha \psi^*_s(s'_1)\ ds_1 ds'_1 \\
G^{A^*,A^*}_{2,0,2} &= \int_{s'_2 \geq s'_1 \geq \tilde s_2 \geq s} \psi^*_s(s) A^*_\alpha(\tilde s_2) \psi^*_s(\tilde s_2) \psi^*_s(s'_1) (A^*)^\alpha(s'_2) \psi^*_s(s'_2)\ d\tilde s_2 ds'_1 ds'_2
\end{align*}
and so forth.  As it will turn out, the $G^{\partial,\partial}_{\tilde J, J,J'}$ expressions will be handled using the trilinear null form estimates \eqref{trilinear-improv}, \eqref{trilinear-improv2}, whereas the other expressions can be handled using Strichartz estimates as soon as $\tilde J + J+J' \geq 5$\footnote{The terms $G^{\partial,\partial}_{\tilde J,J,J'}$ with $\tilde J + J+J' \geq 7$ or $G^{A^*,A^*}_{\tilde J, J, J'}$ with $\tilde J + J+J' \geq 3$ could also be handled by Strichartz estimates, but we will not need to use this fact.}.

Repeated use of \eqref{psi-eq}, \eqref{a-eq} gives the decompositions
\begin{align*}
G_1 &= \bigO( G^{\partial,\partial}_{1,0,2} ) + \bigO( G^{\partial,\partial}_{1,0,4} ) + \bigO( G^{\partial,A^*}_{1,0,4} ) \\
G_3 &:= \bigO( G^{\partial,\partial}_{1,2,2} ) + \bigO( G^{\partial,A^*}_{1,2,2} ) + \bigO( G^{A^*,\partial}_{1,2,2} ) + \bigO( G^{A^*,A^*}_{1,2,2} ) \\
G_4 &:= \bigO( G^{\partial,\partial}_{1,1,1} ) + \sum_{(i,j) = (1,3),(3,1),(3,3)} \bigO( G^{\partial,\partial}_{1,i,j} ) + \bigO( G^{\partial,A^*}_{1,i,j} ) + \bigO( G^{A^*,\partial}_{1,i,j} ) + \bigO( G^{A^*,A^*}_{1,i,j} ).
\end{align*}
The $G_2$ term is handled slightly differently, as both of the null form derivatives are falling on the same term (cf. \cite[Step 2(d)]{tao:wavemap2}).  From \eqref{a-eq} and the Leibniz rule \eqref{leibnitz} we have
$$ G_2 = \int_s^{\infty} \bigO( \psi^*_s(s) (\partial^\alpha \psi^*_\alpha(s')) \psi^*_s(s') ) + \bigO( \psi^*_s(s) \psi^*_\alpha(s') \partial^\alpha \psi^*_s(s') )\ ds'.$$
The second term can be expanded using \eqref{psi-eq}, \eqref{a-eq} to take the form
$$ \bigO( G^{\partial,\partial}_{2,0,1} ) + \bigO( G^{\partial,\partial}_{2,0,3} ) + \bigO( G^{\partial,A^*}_{2,0,3} ).$$
For the first term, we can use the wave-tension field \eqref{w-def} to write $\partial^\alpha \psi^*_\alpha = w^* + \bigO( (A^*)^\alpha \psi^*_\alpha )$.  Expanding the latter out using \eqref{psi-eq}, \eqref{a-eq}, we thus see that we can write this contribution to $G_2$ as
$$ \bigO( H ) + \bigO( G^{\partial,\partial}_{2,1,2} ) + \bigO( G^{A^*,\partial}_{2,1,2} ) + \bigO( G^{\partial,A^*}_{2,1,2} ) + \bigO( G^{A^*,A^*}_{2,1,2} )$$
where
\begin{equation}\label{Hdef}
 H := \int_s^{\infty} \psi^*_s(s) w^*(s') \psi^*_s(s')\ ds'.
\end{equation}

Putting this all together, we see that to prove \eqref{pkgi-targ}, it suffices to show the \emph{null form estimate}
\begin{equation}\label{nullform-est}
\| \partial_x^j P_k G^{\partial,\partial}_{\tilde J, J, J'} \|_{N_k(I \times \R^2)} \lesssim_{\tilde J, J, J,\eps} \mu^{2-\eps} c(s) s^{-(j+2)/2} \chi_{k \geq k(s)}^{\delta_2/10} 
\end{equation}
for $\tilde J, J' \geq 1$ and $J \geq 0$ with $\tilde J+J+J' \geq 3$, the \emph{Strichartz estimate}
\begin{equation}\label{strich-est}
\| \partial_x^j P_k G^{D,D'}_{\tilde J, J, J'} \|_{N_k(I \times \R^2)} \lesssim_{\tilde J, J, J,\eps} \mu^{2-\eps} c(s) s^{-(j+2)/2} \chi_{k \geq k(s)}^{\delta_2/10} 
\end{equation}
for $(D,D') = (\partial,A^*),(A^*,\partial),(A^*,A^*)$, $\tilde J, J' \geq 1$ and $J \geq 0$ with $\tilde J + J + J' \geq 5$, and the \emph{exceptional estimate}
\begin{equation}\label{excep}
 \| \partial_x^j P_k H \|_{N_k(I \times \R^2)} \lesssim_\eps \mu^{2-\eps} c(s) s^{-(j+2)/2} \chi_{k \geq k(s)}^{\delta_2/10}.
\end{equation}

We begin with \eqref{excep}.  By \eqref{Hdef}, the Leibniz rule \eqref{leibnitz}, and Minkowski's inequality, we can bound the left-hand side of \eqref{excep} by
$$ \lesssim \sum_{j_1,j_2,j_3 \geq 0: j_1+j_2+j_3=j} \int_s^\infty \| P_k( \partial_x^{j_1} \psi^*_s(s) \partial_x^{j_2} w^*(s') \partial_x^{j_3} \psi^*_s(s') \|_{N_k(I \times \R^2)} ds'.$$
Applying \eqref{second-prod} and \eqref{prod1}, we can bound this by
$$ \lesssim \sum_{j_1,j_2,j_3 \geq 0: j_1+j_2+j_3=j} \int_s^\infty \sum_{k_2} \chi_{k \geq k_2}^{\delta_2} 
\chi_{k=\max(k(s),k_2)}^{\delta_2} \| \partial_x^{j_1} \psi^*_s(s) \|_{S_{k(s)}(I \times \R^2)} \| \partial_x^{j_2} P_{k_2} w^*(s') \|_{N_{k_2}(I \times \R^2)} \| \partial_x^{j_3} \psi^*_s(s') \|_{S_{k(s')}(I \times \R^2)}\ ds'.$$
Applying \eqref{pkw-1}, \eqref{parab-3}, we can bound this
$$ \lesssim \mu^{2-\eps} s^{-(j+2)/2} \int_s^\infty \sum_{k_2} \chi_{k \geq k_2}^{\delta_2} \chi_{k=\max(k(s),k_2)}^{\delta_2} c(s) c(s')^2 \chi_{k_2 = k(s')}^{\delta_2/10} \frac{ds'}{s'}.$$
Performing the $k_2$ summation and then using \eqref{envelo}, we obtain the claim \eqref{excep}.

We remark that the hardest case here was $j=0$; the case of higher $j$ are at least as easy as those of $j=0$, as each derivative will eventually pull out a factor of $(s')^{-1/2}$ for some $s' \geq s$, which is acceptable.  Hence, to simplify the notation somewhat, when discussing the other two estimates \eqref{nullform-est}, \eqref{strich-est} we shall only discuss the $j=0$ case.

We begin with the $j=0$ case of \eqref{strich-est}.  We fix $\tilde J,J,J'$ and allow implied constants to depend on these quantities.  By \eqref{fl1l2} and Bernstein's inequality, it suffices to show that
$$ \| G^{D,D'}_{\tilde J, J, J'} \|_{L^1_t L^1_x(I \times \R^2)} \lesssim \mu^{2-\eps} c(s)$$
for $(D,D') = (\partial,A^*), (A^*,\partial), (A^*,A^*)$.

Suppose first that $(D,D')=(\partial,A^*)$ or $(D,D')=(A^*,\partial)$.  Then we see that the integrand in \eqref{gdd} consists of $J+J'+J''+1 \geq 6$ factors, most of which are of $\psi_s$ type, but with one factor of type $\partial_{t,x} \psi^*_s$ and one of type $A^*_{t,x}$.  We use H\"older's inequality and place $A^*_{t,x}$ in $L^{5/2}_t L^\infty$, the $\psi^*_s$ or $\partial_{t,x} \psi^*_s$ factor with the lowest value of $s$ in $L^\infty_t L^2_x$, the next three lowest in $L^4_t L^\infty_x$, and the remaining factors in $L^\infty_t L^\infty_x$. Using Lemma \ref{strichlem} and \eqref{envelo}, we can bound the $L^1_t L^1_x(I \times \R^2)$ norm of the integrand of \eqref{gdd} by
$$ \lesssim \mu^{2-\eps} c(s) (s/s_{\max})^\sigma \frac{1}{\tilde s_2} \ldots \frac{1}{\tilde s_{\tilde J}}
\frac{1}{s_1} \ldots \frac{1}{s_J} \frac{1}{s'_1} \ldots \frac{1}{s'_{J'}}$$
for some absolute constant $\sigma>0$, where $s_{\max} := \max( \tilde s_1,\ldots,\tilde s_{\tilde J}, s_1,\ldots,s_J,s'_1,\ldots,s'_{J'})$; the point here is that the $\partial_{t,x}$ and $A^*_{t,x}$ terms are at high values of $s$, leading to the crucial $(s/s_{\max})^\sigma$ gain.  Integrating this using Minkowski's inequality we obtain the claim \eqref{strich-est}.  The case $(D,D')=(A^*,A^*)$ is similar; the $\partial_{t,x}$ term has been replaced by an additional $A^*_{t,x}$ term, but this can be estimated in $L^\infty_t L^\infty_x$ using Lemma \ref{strichlem} to yield the same type of estimates as before.

Finally, we show the $j=0$ case of \eqref{nullform-est}.  We again fix $\tilde J,J,J'$ and allow implied constants to depend on these quantities.  

Let us first handle the case $(\tilde J,J) \neq (1,0)$.  Here, the integrand in \eqref{gdd} takes the form
\begin{equation}\label{Phiphi}
 \Phi \partial^\alpha \psi^*_s(s_J) \partial_\alpha \psi^*_s(s'_{J'}) 
 \end{equation}
where $\Phi$ is the product of many factors of $\psi^*_s$ evaluated at various $s' \geq s$, including one factor of $\psi^*_s$ itself.  By \eqref{parab-3}, \eqref{sm}, \eqref{envelo}, and many applications of \eqref{prod1}, we see that
$$ \| \Phi \|_{S_{k(s)}(I \times \R^2)} \lesssim c(s) s^{-1} \frac{1}{\tilde s_2} \ldots \frac{1}{\tilde s_{\tilde J}}
\frac{1}{s_1} \ldots \frac{1}{s_{J-1}} \frac{1}{s'_1} \ldots \frac{1}{s'_{J'-1}}.$$
Applying \eqref{trilinear-improv2} and Corollary \ref{infi-star}, \eqref{envelo} we thus conclude that the $N_k(I \times \R^2)$ norm of $P_k$ applied to \eqref{Phiphi} is bounded by
$$ 
\lesssim \chi_{k=k(s)}^{\delta_2} (s/s_{\max})^\sigma \mu^{2-\eps}
c(s) s^{-1} \frac{1}{\tilde s_2} \ldots \frac{1}{\tilde s_{\tilde J}}
\frac{1}{s_1} \ldots \frac{1}{s_{J}} \frac{1}{s'_1} \ldots \frac{1}{s'_{J'}}$$
for some $\sigma > 0$ (depending on $\delta_1$); integrating this we obtain the claim \eqref{nullform-est}.

The case when $(\tilde J,J)=(1,0)$ (so $J' \geq 2$) is slightly more delicate, due to the derivative falling on the roughest term.  Here, the integrand in \eqref{gdd} takes the form
\begin{equation}\label{Phiphi2}
\partial^\alpha \psi^*_s(s) \Phi \partial_\alpha \psi^*_s(s'_{J'})
\end{equation}
where $\Phi := \psi^*_s(s'_1) \ldots \psi^*_s(s'_{J'-1})$.  By arguing as before, we have
$$ \| \Phi \|_{S_{k(s'_1)}(I \times \R^2)} \lesssim c(s'_1) \frac{1}{s'_1} \ldots \frac{1}{s'_{J'-1}}.$$
Applying \eqref{trilinear-improv2} and Corollary \ref{infi-star} (but now refraining from using \eqref{envelo} to eliminate some key $c()$ factors), the $N_k(I \times \R^2)$ norm of $P_k$ applied to \eqref{Phiphi2} is bounded by
$$ \lesssim \chi_{k=k(s)}^{\delta_1/2} (s'_1/s'_{J'})^{\delta_1/4} \mu^{2-\eps}
c(s) c(s'_1) c(s'_{J'}) s^{-1} \frac{1}{s'_1} \ldots \frac{1}{s'_{J'}}.$$
We can use \eqref{sm} to bound $(s'_1/s'_{J'})^{\delta_1/4} c(s'_1) c(s'_{J'})$ by $c(s'_1)^2 (s'_1/s'_{J'})^{\delta_1/4}$.  Integrating out $s'_{J'}$, then $s'_{J'-1}$, and so forth down to $s'_1$ (using \eqref{envelo} to control the final integral) we obtain the claim \eqref{nullform-est}. 

The proof of \eqref{psis-iter-diff} follows by applying \eqref{disc-leib-eq} to all the expressions above and repeating the argument in the obvious manner; we omit the details.
\end{proof}

Now, at long last, we are ready to prove \eqref{envelo-targ}.  Fix $0 \leq j \leq 10$ and $s > 0$. Applying Lemma \ref{psis-lem} (with $j=11$, say), Littlewood-Paley decomposition, and the triangle inequality, we have
$$
s^{1+\frac{j}{2}} \| \partial_x^j \delta \psi_s(s) \|_{S_{k(s)}(I \times \R^2)}
\lesssim_\eps \sum_{k'} \chi_{k' \leq k(s)}^{-10} s \|  P_{k'} \partial_{t,x} \psi_s(s,t_0) \|_{L^2_x(\R^2)}
+ \mu^{1-\eps} \delta c(s).
$$
Applying Lemma \ref{psioc-lemma} we conclude
$$
s^{1+\frac{j}{2}} \| \partial_x^j \delta \psi_s(s) \|_{S_{k(s)}(I \times \R^2)}
\lesssim \delta c_0(s) + \mu^{1-\eps} \delta c(s)
$$
for some frequency envelope $\delta c_0$ of energy $O( d_\Energy( \Psi_{s,t,x}, \Psi'_{s,t,x} )^2 )$.

Applying \eqref{csdef-delta}, \eqref{dmc} we conclude
$$ \delta c(s) \lesssim \delta c_0(s) + \mu^{1-\eps} \delta c(s).$$
If $\mu$ is sufficiently small, we thus conclude
$$ \delta c(s) \lesssim \delta c_0(s)$$
and \eqref{envelo-targ} follows from the energy bound on $\delta c_0$.
The proof of \eqref{apriori2a} is now (finally!) complete.

\subsection{Return to the energy metric}

We now establish part (iv) of Theorem \ref{apriori-thm2}, in which we control the energy metric by the $S^1_\mu$ metric.

Fix $I, \mu, M, E$; we let all implied constants depend on $M, E$.  Suppose $\Psi_{s,t,x}, \Psi'_{s,t,x} \in \WMC(I,E)$ is such that $\|\Psi_{s,t,x}\|_{S^1_\mu(I)}, \|\Psi'_{s,t,x}\|_{S^1_\mu(I)} \leq M$.  Let us write
$$ \sigma := d_{S^1_\mu,I}(\Psi_{s,t,x},\Psi'_{s,t,x})$$
and let $t \in I$ be fixed; we omit the explicit mention of this parameter.  By \eqref{psipsi}, our task is to show that
$$ \int_0^\infty \| \delta \psi_s(s) \|_{L^2_x(\R^2)}^2\ ds \lesssim \sigma^2$$
and
$$ \| \delta \psi_t(0) \|_{L^2_x(\R^2)} \lesssim \sigma.$$

Let $c, \delta c$ be as in the preceding section, then $\delta c$ has energy $O( \sigma^2 )$ by construction, while $c$ has energy $O(1)$.
From  Lemma \ref{strichlem} (which does not require a smallness hypothesis on $\mu$), we have the bounds
\begin{align*}
\| \partial_x^j \psi^*_s(s) \|_{L^2_x(\R^2)} &\lesssim_j c(s) s^{-(j+1)/2} \\
\| \partial_x^j \psi^*_s(s) \|_{L^\infty_x(\R^2)} &\lesssim_j c(s) s^{-(j+2)/2} \\
\| \partial_x^{j+1} \psi^*_{t,x}(s) \|_{L^2_x(\R^2)} &\lesssim_j c(s) s^{-(j+1)/2} \\
\| \partial_x^j \psi^*_{t,x}(s) \|_{L^\infty_x(\R^2)} &\lesssim_j c(s) s^{-(j+1)/2} 
\end{align*}
for all $j \geq 0$, and similarly for $\delta \psi_s$ and $\delta \psi_{t,x}$ with $c$ replaced by $\delta c$.  Also, from \cite[Lemma 7.2]{tao:heatwave2} we have
$$ \| \psi^*_{t,x}(s) \|_{L^2_x(\R^2)} \lesssim 1.$$
Repeating the arguments used to prove Lemma \ref{psioc-lemma-2}, we see that
\begin{equation}\label{kpsis-diff-2}
 \|P_k \delta \psi_s(s,t_0)\|_{L^2_x(\R^2)} \lesssim_j \delta c_0(s) s^{-1/2} \chi_{k \leq k(s)}^j \chi_{k=k(s)}^{0.1}
 \end{equation}
and
\begin{equation}\label{kpsit-diff-2}
 \|P_k \delta \psi_t(0,t_0)\|_{L^2_x(\R^2)} \lesssim \delta c_0(2^{-2k}),
\end{equation}
for all $s > 0$, $j \geq 0$, and $k \in \Z$, and the claim follows.

\subsection{Persistence of regularity}

We now prove \eqref{apriori3a}.  We allow all implied constants to depend on $M$ and $\| \phi[t_0] \|_{{\mathcal H}^{100}}$, in particular we have $\| \phi \|_{S^1_\mu(I \times \R^2)} \lesssim 1$ and hence $\E(\phi) \lesssim 1$.  
Our task is now to show that if $\mu$ is sufficiently small, then
\begin{equation}\label{phit-smo} \| \phi[t] \|_{{\mathcal H}^{1+\delta_0/2}_\loc} \lesssim_\mu 1.
\end{equation}
for all $t \in I$.

The first step is to show that some of the ${\mathcal H}^{100}$ regularity persists under the heat flow.  Let $\Psi_{s,t,x} \in \WMC(\phi,I)$ be one of the differentiated fields for $\phi$ (the exact choice does not matter).

\begin{lemma}[Smooth control of $\psi_s$]\label{psis-smooth}  For any $0 < s \lesssim 1$ we have
$$
\| \psi_s(s,t_0) \|_{H^{90}_x(\R^2)} + \| \partial_{t,x} \psi_s(s,t_0) \|_{H^{90}_x(\R^2)}
\lesssim 1.$$
\end{lemma}

\begin{proof}  By hypothesis, 
$$ \| \phi[t_0] \|_{{\mathcal H}^{100}} \lesssim 1.$$
From \eqref{hkdef}, \eqref{hkdef-loc}, for each $x_0$ one can find $U_{x_0} \in SO(m,1)$ such that
$$\| \eta(\cdot-x_0) U_{x_0}(\phi_0) \|_{H^{100}_x(\R^2)} + \| \eta(\cdot-x_0) U_{x_0}(\phi_1) \|_{H^{99}_x(\R^2)} \lesssim 1$$
which by Sobolev embedding implies that $U_{x_0}(\phi_0) = O(1)$ on $B(x_0,1)$.  From this and many applications of the Leibniz rule \eqref{leibnitz}, we conclude that
\begin{align*}
&\| |(\phi^* \nabla_x)^j \partial_x \phi_0|_{\phi^* h} \|_{L^2(B(x_0,1))} +
\| |(\phi^* \nabla_x)^j \phi_1|_{\phi^* h} \|_{L^2(B(x_0,1))} \\
&\quad \lesssim
\inf_{U \in SO(m,1)} \| \eta(\cdot-x_0) U(\phi_0) \|_{H^s_x(\R^2)} + \| \eta(\cdot-x_0) U(\phi_1) \|_{H^{s-1}_x(\R^2)}
\end{align*}
for all $x_0 \in \R^2$ and $0 \leq j \leq 99$.  Integrating this over $x_0$ and using \eqref{hkdef-loc}, we conclude that
$$ \| |(\phi^* \nabla_x)^j \partial_x \phi_0|_{\phi^* h} \|_{L^2(\R^2)} +
\| |(\phi^* \nabla_x)^j \phi_1|_{\phi^* h} \|_{L^2(\R^2)} \lesssim 1$$
and thus in the caloric gauge
$$ \| D_x^j \psi_{t,x}(0,t_0) \|_{L^2(\R^2)} \lesssim 1$$
for $0 \leq j \leq 99$.

For $s \geq 0$, define
$$ E_{100}(s) := \sum_{j=0}^{99} \int_{\R^2} |D_x^j \psi_{t,x}(s,t_0)|^2\ dx$$
thus $E_{100}(0) = O(1)$.  On the other hand, from Theorem \ref{envelop} we can find an envelope $c_0$ of energy $O(1)$ such that $D_x^j \psi_{t,x}(s,t_0) = O( c_0(s) s^{-(j+1)/2} )$ for all $0 \leq j \leq 100$.  Using this, \eqref{psit-eq}, and many applications of covariant integration by parts, one obtains the differential inequality
$$ \partial_s E_{100}(s) = O( c_0(s)^2 s^{-1} E_{100}(s) )$$
which by Gronwall's inequality and the fact that $c_0$ has energy $O(1)$ implies that $E_{100}(s) = O(1)$ for all $s$, thus
$$ \| D_x^j \psi_{t,x}(s,t_0) \|_{L^2_x(\R^2)} \lesssim 1$$
for all $0 \leq j \leq 99$ and $s > 0$.  Using the covariant Gagliardo-Nirenberg inequality
$$ \| \varphi \|_{L^\infty_x(\R^2)} \lesssim \| \varphi \|_{L^2_x(\R^2)}^{1/3}
\| D_x \varphi \|_{L^2_x(\R^2)}^{1/3} \| D_x^2 \varphi \|_{L^2_x(\R^2)}^{1/3}$$
(see the proof of \cite[Lemma 4.2]{tao:heatwave2}) we conclude that
$$ \| D_x^j \psi_{t,x}(s,t_0) \|_{L^\infty_x(\R^2)} \lesssim 1$$
for all $0 \leq j \leq 97$ and $s > 0$.  Combining this with what one gets from Theorem \ref{envelop}, we conclude
$$ \| D_x^j \psi_{t,x}(s,t_0) \|_{L^\infty_x(\R^2)} \lesssim \langle s \rangle^{-(j+1)/2}$$
for all $0 \leq j \leq 97$ and $s > 0$.  In particular, by \eqref{heatflow} this implies
$$ \| D_x^j \psi_s(s,t_0) \|_{L^\infty_x(\R^2)} \lesssim \langle s \rangle^{-(j+2)/2}$$
for all $0 \leq j \leq 96$ and $s > 0$.  Similar reasoning gives
$$ \| D_x^j \psi_{t,x}(s,t_0) \|_{L^2_x(\R^2)} \lesssim \langle s \rangle^{-j/2}$$
and
$$ \| D_x^j \psi_s(s,t_0) \|_{L^2_x(\R^2)} \lesssim \langle s \rangle^{-(j+1)/2}$$
for this range.  Applying \eqref{a-eq}, the covariant Leibniz rule, and Minkowski's inequality, we conclude
$$ \| D_x^j A_{t,x}(s,t_0) \|_{L^\infty_x(\R^2)} \lesssim \langle s \rangle^{-(j+1)/2}$$
and
$$ \| D_x^j A_{t,x}(s,t_0) \|_{L^2_x(\R^2)} \lesssim \langle s \rangle^{-j/2}$$
for all $0 \leq j \leq 96$ and $s>0$, where $D_x B := \partial_x B + [A_x,B]$ for any matrix field $B$.  (The $j=0$ $L^2$ bound on $A_{t,x}$ was already established in \cite[Proposition 4.3]{tao:heatwave2}.)  Writing non-covariant derivatives in terms of covariant ones, we conclude that the above two estimates continue to hold if $D_x$ is replaced by $\partial_x$.  We can then similarly replace $D_x$ by $\partial_x$ in all previous estimates.  We conclude that
$$\| \Psi_{t,x}(s,t_0) \|_{H^{96}_x(\R^2)} + \| \psi_s(s,t_0) \|_{H^{96}_x(\R^2)} \lesssim 1$$
for all $0 \leq s \lesssim 1$.  The claim then easily follows from \eqref{heatflow} and \eqref{zerotor-frame}.
\end{proof}

Let $c(s)$ be the envelope from the proof of \eqref{apriori2a}.  We can obtain a decay estimate:

\begin{lemma}[Breaking the scaling barrier]\label{little-scale} We have
$$ c(s) \lesssim_\mu \min(1, s^{-\delta_0/2})$$
for all $s > 0$.
\end{lemma}

\begin{proof}
Let $c(s)$ be the envelope from the proof of \eqref{apriori3}.  From Lemma \ref{psis-smooth} and 
\eqref{psis-iter-eq} we have
$$ \| P_k \psi^*_s(s) \|_{S_k(I \times \R^2)} \lesssim_{\eps}  \min( 2^k, 2^{-90k} ) + \mu^{2-\eps} c(s) s^{-1} \chi_{k \geq k(s)}^{\delta_2/10} \chi_{k \leq k(s)}^{11}
$$
for $0 < s \lesssim 1$ and $\eps > 0$.  This implies that
$$
s^{1+\frac{j}{2}} \| \partial_x^j \psi_s(s) \|_{S_k(I \times \R^2)}
\lesssim_\eps s + \mu^{2-\eps} c(s)$$
for all $0 \leq j \leq 10$, $0 \leq s \lesssim 1$, and $\eps$; by \eqref{sksk-star}, \eqref{dkd} and \eqref{sm} this implies that
$$ d_k(0,\Psi) \lesssim_\eps \mu^{-1} 2^{-k} + \mu^{1-\eps} c(2^{-2k})$$
for $k \geq 0$.  For $k < 0$ we have the trivial bound $d_k(0,\Psi) \lesssim c(2^{-2k})$ (by \eqref{csdef}); we can unify these bounds to obtain
$$ d_k(0,\Psi) \lesssim_\eps \mu^{-1} \langle 2^k \rangle^{-1} + (\mu^{1-\eps} + \langle 2^k \rangle^{-1}) c(2^{-2k})$$
for all $k$.  Inserting this into \eqref{csdef} and using \eqref{sm} we obtain
$$ c(2^{-2k}) \lesssim_\eps \mu^{-1} \langle 2^k \rangle^{-\delta_0} + (\mu^{1-\eps} + \langle 2^k \rangle^{-\delta_0}) c(2^{-2k})$$
and thus (if $\mu$ is small enough, and $\eps$ is chosen to be, say, $1/2$) we see that
$$ c(2^{-2k}) \lesssim_\mu \langle 2^k \rangle^{-\delta_0}$$
for all $k$ (note that this already follows from \eqref{envelo} if $2^{-k} \lesssim 1$). The claim follows.
\end{proof}

We now fix a time $t \in I$.
From Lemma \ref{little-scale} and Lemma \ref{strichlem} one has
$$ \| \Psi_{t,x}(s,t) \|_{L^\infty_x(\R^2)} \lesssim_\mu \min( 1, s^{-\delta_0/2} ) s^{-1/2}$$
and
$$ \| \psi_s(s,t) \|_{L^\infty_x(\R^2)} \lesssim_\mu \min( 1, s^{-\delta_0/2} ) s^{-1}$$
for all $s > 0$.  Converting this back to the original wave map $\phi$, we see that
$$ |\partial_{t,x} \phi(s,t,x)|_{\phi^* h} \lesssim_\mu \min( 1, s^{-\delta_0/2} ) s^{-1/2}$$
and
$$ |\partial_s \phi(s,t,x)|_{\phi^* h} \lesssim_\mu \min( 1, s^{-\delta_0/2} ) s^{-1}$$
for all $s > 0$ and $x \in \R^2$.  Observe that the right-hand sides here are locally integrable in $s$.  By the fundamental theorem of calculus on $\H$, we conclude that
$$ d_\H( \phi(s_1,t,x_1), \phi(s_2,t,x_2) ) \lesssim_\mu 1$$
whenever $0 < s_1,s_2 \lesssim 1$ and $x_1,x_2\in \R^2$ are such that $|x_2-x_1| \lesssim 1$, and $d_\H$ is the distance function on $\H$ induced by the metric $h$.  (In fact we even obtain some H\"older continuity estimates here, but we will not exploit this.)

From Lemma \ref{little-scale} and Lemma \ref{strichlem} we have
$$ \| \partial_x^j \psi_s \|_{L^2_x(\R^2)} \lesssim_{\mu,j} s^{-\delta_0/2} s^{-(j+1)/2}$$
for all $0 < s \leq 1$ and $j > 0$.  By interpolation this implies that
\begin{equation}\label{psisdecay-eq}
\|\psi_s(s)\|_{H^{1+\delta_0/2}_x(\R^2)} \lesssim_\mu s^{-\delta_0/4} s^{-1}
\end{equation}
for all $0 < s \leq 1$.  Similarly we have
$$ \| \partial_x^j \partial_t \psi_s \|_{L^2_x(\R^2)} \lesssim_{\mu,j} s^{-\delta_0/2} s^{-(j+2)/2}$$
and hence
\begin{equation}\label{psisdecay-eq2}
\|\partial_t \psi_s(s)\|_{H^{\delta_0/2}_x(\R^2)} \lesssim_\mu s^{-\delta_0/4} s^{-1}.
\end{equation}

Let $\frac{1}{100} \Z^2 \subset \R^2$ be the standard lattice with spacing $\frac{1}{100}$ in the spatial domain $\R^2$.  For each $x_0 \in \frac{1}{100} \Z^2$, let $U_{x_0} \in SO(m,1)$ be a Lorentz rotation that sends $\phi(1,t,x_0)$ to the origin $(1,0) \in \H$.  We write $\phi_{x_0} := U_{x_0}(\phi)$, $e_{x_0} := U_{x_0} \circ e$.  By the preceding discussion, we see that 
$$
\phi_{x_0}(s,t,x_1) = O(1)
$$
whenever $x_1 = x_0 + O(1)$ and $0 \leq t \leq 1$; if we set $\Phi_{x_0} := (\phi_{x_0},e_{x_0})$, this implies that
\begin{equation}\label{phixo}
\Phi_{x_0}(s,t,x_1) = O(1)
\end{equation}
for the same range of $t, x_1$.  Also, from Lemma \ref{strichlem} we have
\begin{equation}\label{psijok}
 |\partial_x^j \Psi_{t,x}(1,t,x)| \lesssim_{\mu,j} 1
 \end{equation}
for $j \geq 0$ and $x \in \R^2$.  From \eqref{psij-def}, \eqref{Adef} (applied to the rotated wave map $\phi_{x_0}$ and its associated rotated frame $e_{x_0}$) we have the differential equations
$$ \partial_{t,x} \phi_{x_0} = e_{x_0}( \psi_{t,x} )$$
and
$$ (\phi_{x_0}^* \partial_{t,x}) (e_{x_0})_a = (A_{t,x})_{ab} e_b.$$
We write these equations schematically using \eqref{phicor} as
\begin{equation}\label{phicor-tx}
\partial_{t,x} \Phi_{x_0} = \bigO( \Phi_{x_0} \Psi_{t,x} ) + \bigO( \Phi_{x_0}^2 \Psi_{t,x} )
\end{equation}
where $\Phi_{x_0} := (\phi_{x_0},e_{x_0})$.
From this, \eqref{phixo}, \eqref{psijok}, and many applications of the Leibniz rule \eqref{leibnitz} and Gronwall's inequality, we thus conclude that
\begin{equation}\label{eok-eq} 
|\partial_x^j \Phi_{x_0}(1,t,x)|, |\partial_x^j \partial_t \Phi_{x_0}(1,t,x)| \lesssim_j 1
\end{equation}
whenever $j \geq 0$ and $|x-x_0| \lesssim 1$.

For each $0 \leq s \leq 1$, define
$$ f(s) := \sup_{x_0 \in \frac{1}{100} \Z^2} \| \eta(\cdot - x_0) \Phi_{x_0}(s,t) \|_{H^{1+\delta_0/2}_x(\R^2)} + \| \eta(\cdot - x_0) \partial_t \Phi_{x_0}(s,t) \|_{H^{\delta_0/2}_x(\R^2)}.$$
From \eqref{eok-eq} we see that
\begin{equation}\label{fs-1}
f(1) \lesssim 1.
\end{equation}

We shall shortly show that
\begin{equation}\label{f-deriv-eq}
\limsup_{ds \to 0} \frac{|f(s+ds)-f(s)|}{|ds|} \lesssim_\mu s^{-\delta_0/4} s^{-1} f(s)
\end{equation}
for all $0 < s \leq 1$.   Note that $s^{-\delta_0/4} s^{-1}$ is integrable on this interval.  From this, Gronwall's inequality and \eqref{fs-1} we see that $f(s) \lesssim 1$ for all $0 < s \leq 1$; taking limits as $s \to 0$ we conclude $f(0) \lesssim 1$, from which \eqref{phit-smo} (and hence \eqref{apriori3}) easily follows.  Thus the only remaining task is to show \eqref{f-deriv-eq}.

By Minkowski's inequality and the fundamental theorem of calculus (and the smoothness of $\phi$) it suffices to show that
$$\| \eta(\cdot - x_0) \partial_s \Phi_{x_0}(s,t) \|_{H^{1+\delta_0/2}_x(\R^2)} + \| \eta(\cdot - x_0) \partial_t \partial_s \Phi_{x_0}(s,t) \|_{H^{\delta_0/2}_x(\R^2)} \lesssim_\mu s^{-\delta_0/4} s^{-1} f(s)$$
for all $0 < s \leq 1$ and $x_0 \in \frac{1}{100} \Z^2$.

Fix $x_0, s$.  By repeating the derivation of \eqref{phicor-tx} we have
\begin{equation}\label{phicor-s} 
\partial_s \Phi_{x_0} = \bigO( \Phi_{x_0} \psi_s ) + \bigO( \Phi_{x_0}^2 \psi_s );
\end{equation}
using the standard Sobolev product estimates
$$ \| uv \|_{H^{1+\delta_0/2}_x(\R^2)} \lesssim \|u\|_{H^{1+\delta_0/2}_x(\R^2)} \|v\|_{L^\infty_x(\R^2)} +
\|u\|_{L^\infty_x(\R^2)} \|v\|_{H^{1+\delta_0/2}_x(\R^2)} $$
(see e.g. \cite[Lemma A.8]{tao:cbms} for a proof) together with \eqref{psisdecay-eq} and \eqref{phixo}, we conclude that
$$ \| \eta(\cdot - x_0) \partial_s \Phi_{x_0}(s,t) \|_{H^{1+\delta_0/2}_x(\R^2)} \lesssim_\mu s^{-\delta_0/4} s^{-1} f(s).$$
Similarly, applying $\partial_t$ to \eqref{phicor-s} and using the above estimates and the product estimate
$$ \| uv \|_{H^{\delta_0/2}_x(\R^2)} \lesssim \|u\|_{H^{1+\delta_0/2}_x(\R^2)} \|v\|_{H^{\delta_0/2}_x(\R^2)}$$
we obtain
$$ \| \eta(\cdot - x_0) \partial_t \partial_s \Phi_{x_0}(s,t) \|_{H^{\delta_0/2}_x(\R^2)} \lesssim_\mu s^{-\delta_0/4} s^{-1}  f(s).$$
The claim \eqref{f-deriv-eq} (and hence \eqref{apriori3a}) follows.

\section{Construction of function spaces}\label{func-sec}

We now prove Theorem \ref{func}.  The spaces $S_k$, $S_{\mu,k}$, $N_k$ we need to employ are quite complicated.  Fortunately, we can take their definition directly from \cite{tao:wavemap2}, with few modifications, so that the verification of Theorem \ref{func} largely consists of citing the relevant results from that paper.  (The spaces in \cite{tao:wavemap2} were in turn based on those in \cite{tataru:wave2}.)

More precisely, we will take $S_k$, $N_k$ to be the spaces introduced in \cite[Theorem 3]{tao:wavemap2}, and constructed in \cite[Section 10]{tao:wavemap2}.  The exact definition of these spaces is complicated, but will not actually be needed for this paper, as we can largely deduce everything we need from \cite[Theorem 3]{tao:wavemap2}.  For instance:
\begin{itemize}
\item  The product estimate \eqref{prod1} follows immediately from \cite[Equation (125)]{tao:wavemap2}.
\item  The estimate \eqref{fl1l2} is precisely \cite[Equation (25)]{tao:wavemap2}.
\item  The estimate \eqref{physical} follows immediately from \cite[Equation (87)]{tao:wavemap2}.
\item  The estimate \eqref{physicaln-eq} is precisely \cite[Equation (26)]{tao:wavemap2}.
\item  The estimate \eqref{energy-est} is precisely \cite[Equation (27)]{tao:wavemap2}.
\item  The estimate \eqref{second-prod} is precisely \cite[Equation (29)]{tao:wavemap2} (and some Littlewood-Paley decomposition).
\item  The estimate \eqref{outgo} is precisely \cite[Equation (35)]{tao:wavemap2}.
\item  The estimate \eqref{lstrich} follows from \cite[Lemma 3.1]{krieger:2d} (or \cite[Lemma 6.7]{krieger:2d}).
\end{itemize} 

We will however point out here one feature of the $S_k$ spaces that was not emphasised in those papers, and which is needed in order to establish the continuity claim in Theorem \ref{func}.  As noted in passing in \cite{tao:wavemap2}, the spaces $S_k$ constructed in that paper enjoy the discrete scale invariance
\begin{equation}\label{disc}
\| \phi^{(j)} \|_{S_{k+j}(2^{-j} I)} = \| \phi \|_{S_k(I \times \R^2)}
\end{equation}
for all intervals $I$, $\phi \in \Sch(I \times \R^2)$, and $k,j \in \Z$, where $\phi^{(j)}(t,x) := \phi(2^j t, 2^j x)$.  It is also possible to establish a continuous analogue 
\begin{equation}\label{cont-eq} \| \phi^{(j)} \|_{S_{k+[j]}(2^{-j} I)} \lesssim \| \phi \|_{S_k(I \times \R^2)}
\end{equation}
of this scale invariance, where $j$ is now taken to be real instead of integer, and $[j]$ is the nearest integer to $j$ (rounding down, say).  Indeed, to establish this, one can use \eqref{disc} to reduce to the case $-1/2 < j \leq 1/2$, in which case the claim follows by modifying the proof of \cite[Lemma 9]{tao:wavemap2}.  We omit the details\footnote{There are several alternate ways to deal with this.  One is to adjust the definition of the spaces in \cite{tao:wavemap2} so that frequency parameters such as $k, j$ vary over the reals rather than the integers.  Another is to abandon any proof of continuity, and settle instead for the weaker property of quasicontinuity as in \cite{tao:wavemap2}; it is not hard to see that quasicontinuity serves as a reasonable substitute for continuity for the purposes of performing continuity arguments.}.

Let us now define the modified $S_k$ norm
$$ \| \phi \|_{\overline{S_k}(I \times \R^2)} := \sup_{j\in \R} 2^{-100|j|} \| \phi^{(j)} \|_{S_{k}(2^{-j} I \times \R^2)}$$
From \eqref{cont-eq}, \eqref{physical} we see that the $\overline{S_k}(I \times \R^2)$ norm is equivalent to the $S_k(I \times \R^2)$ norm up to absolute constants.  Furthermore, observe that $\| \phi^{(j)} \|_{\overline{S_k(2^{-j}(I \times \R^2))}}$ varies continuously in $j \in \R$.  The construction of $S_k(I \times \R^2)$ in \cite{tao:wavemap2} ensures that this norm is monotone in $I$, and so the $\overline{S_k}(I \times \R^2)$ norm is monotone also.

We now define $S_{\mu,k}(I \times \R^2)$ to be the atomic Banach space whose atoms $\phi$ are one of two types:

\begin{itemize}
\item (Small atoms) $\phi \in \Sch(I \times \R^2)$ with $\|\phi\|_{\overline{S_k(I \times \R^2)}} \leq \mu$.
\item (Integrable atoms) $\phi \in \Sch(I \times \R^2)$ with $\|\phi\|_{\overline{S_k(I \times \R^2)}} \leq 1$, and
$\| \partial_{t,x} \phi \|_{L^1_t L^\infty_x(I \times \R^2)} \leq \mu^5$.
\end{itemize}

The estimate \eqref{sksk-star} is now immediate, as is the monotonicity of $S_{\mu,k}(I \times \R^2)$ in $I$, while \eqref{physical-star} follows from \eqref{physical}.  Also, these spaces are clearly invariant under space and time translations.  Now we establish continuity of $\| \phi \|_{S_{\mu,k}(I \times \R^2)}$ in $I$ for $\phi \in \Sch(I \times \R^2)$, using a scaling argument of Tataru (cf. \cite[Lemma 3.2]{krieger:2d}).  By monotonicity and time translation symmetry it will suffice to show that $\| \phi \|_{S_{\mu,k}([-T,T] \times \R^2)}$ is continuous in $T$ for $\phi \in \Sch([-T,T] \times \R^2)$.  We can extend $\phi$ to lie in $\Sch(I \times \R^2)$ for some neighbourhood $I$ of $[-T,T]$.  From the continuity of $\| \phi^{(j)} \|_{\overline{S_k(2^{-j} I)}}$ and $\| \partial_{t,x} \phi^{(j)} \|_{L^1_t L^\infty_x(2^{-j} I)}$ in $j$ (in fact the latter quantity is invariant in $j$) we already know that $\| \phi^{(j)} \|_{S_{\mu,k}([-2^{-j} T,2^{-j} T] \times \R^2)}$ is continuous in $j$, so by the triangle inequality it suffices to show that
$$ \lim_{j \to 0} \| \phi^{(j)} - \phi \|_{S_{\mu,k}([-2^{-j} T,2^{-j} T] \times \R^2)} = 0.$$
But from \eqref{split-energy-est}, \eqref{sksk-star}, we can bound
\begin{align*}
\| \phi^{(j)} - \phi \|_{S_{\mu,k}([-2^{-j} T,2^{-j} T] \times \R^2)}
&\lesssim_\mu \sum_{k'} \chi_{k=k'}^{-\delta_1} [\| P_{k'} \partial_{t,x} (\phi^{(j)}-\phi)(0) \|_{L^2_x(\R^2)} \\
&\quad +
\| P_{k'} \Box( \phi^{(j)} - \phi ) \|_{L^1_t L^2_x([-2^{-j} T, 2^{-j} T] \times \R^2)}].
\end{align*}
Since $\phi$ is Schwartz, it is not difficult to see (by a dominated convergence argument) that the right hand side goes to zero as $j \to 0$, and the claim follows.

The estimate \eqref{lstrich-4} follows from \eqref{lstrich} (for small atoms) and from interpolating \eqref{lstrich} with the $L^1_t L^\infty_x$ estimate (for integrable atoms).  Now we establish the vanishing property \eqref{shrinko}.  By time translation we may take $t_0=0$, and then we may extend $\phi$ slightly so that it lies in $\Sch([-\eps,\eps] \times \R^2)$ for some $\eps > 0$.  It suffices to show that
$$ \limsup_{T \to 0} \| \phi \|_{S_{\mu,k}([-T,T] \times \R^2)} \lesssim \sum_{k'} \chi_{k=k'}^{-\delta_1} \| \partial_{t,x} P_{k'} \phi(t_0) \|_{L^2_x(\R^2)}.$$
By definition of $S_{\mu,k}$ and the triangle inequality, we can estimate
$$ \| \phi \|_{S_{\mu,k}([-T,T] \times \R^2)} \lesssim \sum_{k'} \| P_{k'} \phi \|_{S_{k}([-T,T] \times \R^2)} + \frac{1}{\mu^5} \| \partial_{t,x} P_{k'} \phi \|_{L^1_t L^\infty_x([-T,T] \times \R^2)}.$$
From \eqref{split-energy-est} we have
$$ \| P_{k'} \phi \|_{S_{k}([-T,T] \times \R^2)} \lesssim \chi_{k=k'}^{-\delta_1} [ \| P_{k'} \partial_{t,x} \phi(0) \|_{L^2_x(\R^2)} + \| P_{k'} \Box \phi \|_{L^1_t L^2_x([-T,T] \times \R^2)} ].$$
Since $\phi \in \Sch([-\eps,\eps] \times \R^2)$, an easy application of the dominated and monotone convergence theorems shows that
$$ \lim_{T \to 0} \sum_{k'} \| \partial_{t,x} P_{k'} \phi \|_{L^1_t L^\infty_x([-T,T] \times \R^2)} = 0$$
and
$$ \lim_{T \to 0} \sum_{k'} \chi_{k=k'}^{-\delta_1} \| P_{k'} \Box \phi \|_{L^1_t L^2_x([-T,T] \times \R^2)} = 0,$$
and the claim follows.

To prove \eqref{parreg}, we can use exactly the same argument used to prove \eqref{parreg-abstract} in Section \ref{initial-sec}; note from \cite[Equation (87)]{tao:wavemap2} that the $S_k(\R)$ norm (which the $S_k(I \times \R^2)$ norm is the restriction of) has a form extremely similar to that of \eqref{seminorm-eq}.  We omit the details.

Finally, we verify \eqref{trilinear-improv}, \eqref{trilinear-improv2}.  The $\eps=1$ cases of these inequalities are exactly \cite[Equation (31)]{tao:wavemap2}, so it suffices to establish the $\eps=0$ case.

We begin with \eqref{trilinear-improv}.  By symmetry we may take $\mu_2=\mu$.  We may reduce to the case when $\phi^{(2)}$ is an atom.  If it is a small atom, the claim again follows from \cite[Equation (31)]{tao:wavemap2}.  If instead $\phi^{(2)}$ is an integrable atom, we use \eqref{fl1l2} and H\"older's inequality to estimate
\begin{align*}
 \| P_k( \phi^{(1)} \partial_\alpha \phi^{(2)} \partial^\alpha \phi^{(3)}) \|_{N_k(I \times \R^2)} &\lesssim
 \|\phi^{(1)} \partial_\alpha \phi^{(2)} \partial^\alpha \phi^{(3)}\|_{L^1_t L^2_x(I \times \R^2)}\\
 &\lesssim \| \phi^{(1)} \|_{L^\infty_t L^\infty_x(I \times \R^2)} \| \partial_{t,x} \phi^{(2)} \|_{L^1_t L^\infty_x(I \times \R^2)} \| \partial_{t,x} \phi^{(3)} \|_{L^\infty_t L^2_x(I \times \R^2)},
 \end{align*}
 which is acceptable (with several powers of $\mu$ to spare) from \eqref{lstrich}, \eqref{outgo}, and the definition of an integrable atom.  The proof of \eqref{trilinear-improv2} is similar; if both $\phi^{(2)}$ and $\phi^{(3)}$ are small atoms, one can use \cite[Equation (31)]{tao:wavemap2}, and when there is an integrable atom, one uses H\"older's inequality as above.  The proof of Theorem \ref{func} is complete.

\end{document}